\documentclass[oneside,english]{amsart}
\usepackage[T1]{fontenc}
\usepackage[latin9]{inputenc}
\usepackage{textcomp}
\usepackage{mathrsfs}
\usepackage{amsbsy}
\usepackage{amstext}
\usepackage{amsthm}
\usepackage{amssymb}
\usepackage{stmaryrd}
\usepackage{stackrel}
\usepackage{graphicx}
\PassOptionsToPackage{normalem}{ulem}
\usepackage{ulem}

\makeatletter

\providecommand{\tabularnewline}{\\}

\numberwithin{equation}{section}
\numberwithin{figure}{section}
\theoremstyle{plain}
\newtheorem{thm}{\protect\theoremname}[section]
  \theoremstyle{plain}
  \newtheorem{prop}[thm]{\protect\propositionname}
  \theoremstyle{plain}
  \newtheorem{lem}[thm]{\protect\lemmaname}
  \theoremstyle{plain}
  \newtheorem{cor}[thm]{\protect\corollaryname}
 \theoremstyle{definition}
 \newtheorem*{defn*}{\protect\definitionname}
  \theoremstyle{remark}
  \newtheorem*{rem*}{\protect\remarkname}
  \theoremstyle{remark}
  \newtheorem{claim}[thm]{\protect\claimname}

\usepackage[all]{xy}

\makeatother

\usepackage{babel}
  \providecommand{\claimname}{Claim}
  \providecommand{\corollaryname}{Corollary}
  \providecommand{\definitionname}{Definition}
  \providecommand{\lemmaname}{Lemma}
  \providecommand{\propositionname}{Proposition}
  \providecommand{\remarkname}{Remark}
\providecommand{\theoremname}{Theorem}

\begin{document}

\title{On the bad reduction of certain $U(2,1)$ Shimura varieties}

\author{Ehud de Shalit and Eyal Z. Goren}



\address{Ehud de Shalit, Hebrew University of Jerusalem, Israel}

\address{ehud.deshalit@mail.huji.ac.il}

\address{Eyal Z. Goren, McGill University, Montréal, Canada}

\address{eyal.goren@mcgill.ca}

\maketitle
\global\long\def\Ver{{\rm Ver}}

\global\long\def\Hom{\text{Hom}}

\global\long\def\End{{\rm {\rm End}}}

\global\long\def\Fr{{\rm Fr}}

\global\long\def\KS{{\rm KS}}

\global\long\def\Lie{{\rm Lie}}

\global\long\def\Spec{{\rm Spec}}

\begin{center}
\textit{Dedicated to V. Kumar Murty on the occasion of his 60$^{th}$ birthday}
\end{center}

\begin{abstract}
Let $E$ be a quadratic imaginary field and let $p$ be a prime which
is inert in $E.$ We study three types of Picard modular surfaces in positive characteristic
$p$ and the morphisms between them. The first Picard surface, denoted $S$,
parametrizes triples $(A,\phi,\iota)$ comprised of an abelian
threefold $A$ with an action $\iota$ of the ring of integers $\mathcal{O}_{E}$,
and a principal polarization $\phi$. The second surface, $S_{0}(p)$, parametrizes, in addition,
a suitably restricted choice of a subgroup $H\subset A[p]$ of rank
$p^{2}$. The third Picard surface, $\widetilde{S}$, parametrizes
triples $(A,\psi,\iota)$ similar to those parametrized by $S$, but
where $\psi$ is a polarization of degree $p^{2}$. We study the components, singularities and naturally defined stratifications of these surfaces, and their behaviour under the morphisms. A particular role is played by a foliation we define on the blow-up of $S$ at its superspecial points. 
\end{abstract}

\tableofcontents{}

\section*{Introduction}

Let $E$ be a quadratic imaginary field and let $p$ be a prime which
is inert in $E.$ This paper is concerned with the detailed study
of three types of Picard modular surfaces in positive characteristic
$p$ and the morphisms between them. Deferring precise definitions
to the body of the paper, the first Picard surface, denoted $S$,
parametrizes triples $(A,\phi,\iota)$ comprised of a certain abelian
threefold $A$ with an action $\iota$ of the ring of integers $\mathcal{O}_{E}$,
and a principal polarization $\phi$. Unlike the other two, $S$ is
smooth. The second surface, $S_{0}(p)$, parametrizes, in addition,
a suitably restricted choice of a subgroup $H\subset A[p]$ of rank
$p^{2}$. The third Picard surface, $\widetilde{S}$, parametrizes
triples $(A,\psi,\iota)$ similar to those parametrized by $S$, but
where $\psi$ is a polarization of degree $p^{2}$. There are natural
morphisms providing us with a diagram \[ \xymatrix{& S_0(p)\ar[dr]^\pi\ar[dl]_{\tilde \pi} &\\ \tilde S && S.}\] 

From another perspective, there are three Shimura varieties associated
with the unitary group of $E$ of signature (2,1), having parahoric
level structure at $p.$ The above mentioned moduli spaces are the
special fibers at $p$ of the integral models of these Shimura varieties,
studied by Rapoport and Zink in \cite{Ra-Zi}.

Before describing the main results of this article, we provide some
background, context and motivation. Picard modular surfaces appear
in many places in the literature; the book by Langlands and Ramakrishnan
\cite{La-Ra} provides a strong motivation for their study as a test
case for the Langlands conjectures on modularity of $L$-functions,
as well as a guide to the literature at the time. The local structure
at $p$ of $S_{0}(p)$ and related moduli spaces was studied in Bellaïche's
thesis \cite{Bel}, and later in the work of Bültel-Wedhorn \cite{Bu-We}
and Koskivirta \cite{Kos}, where the authors applied it to lifting
problems of Picard modular forms, Galois representations, and congruence
relations for Hecke operators. However, the global structure of $S_{0}(p)$
and of the map $S_{0}(p)\to S$ remained opaque. Thus, one of our
original motivations was to make this global structure precise.

Unlike $S_{0}(p),$ there is little information in the literature
on $\widetilde{S}$, or in general on moduli spaces of abelian varieties
with non-separable polarizations. The main examples we are aware of
are \cite{Cri,dJ1,Nor,N-O}, and they tend to exhibit rather pathological
phenomena. It is desirable to have additional examples available,
and indeed $\widetilde{S}$, in contrast to loc. cit., has proven
to be extremely well-behaved.

Our main reason for studying the three Picard modular surfaces, was
however different. Motivated by questions on the canonical subgroup,
or by the search for a geometric proof of the congruence relation
(as in \cite{Bu-We,Kos}), it is desirable to have a surface parametrizing
tuples $(A,\phi,\iota,H)$, where $H$ is a finite flat subgroup scheme
which may reduce mod $p$ to the kernel of Frobenius. As this kernel
has rank $p^{3}$ and in characteristic $0$ the rank of a $p$-primary
$\mathcal{O}_{E}$-subgroup must be an even power of $p,$ such a
surface does not exist. To remedy the situation, one is forced to
consider a moduli space as above, but where $H$ is now of rank $p^{6}$.
In the context of modular curves this is akin to passing from $X_{0}(p)$
to $X_{0}(p^{2})$; a process which is, of course, unnecessary for
modular curves, but would be required for many Shimura varieties.

It turns out that it is beneficial to modify the moduli problem somewhat
and following \cite{dJ2} to consider a filtration of $H$ as part
of the datum. That is, (roughly) the following data: $(A,\phi,\iota,H_{0}\subseteq H)$,
where $(A,\phi,\iota,H_{0})$ is an object parametrized by $S_{0}(p)$
and $H$ is a suitable rank $p^{6}$ finite flat subgroup scheme.
We call this moduli problem $\mathcal{T}$, and one of our initial
observations is that 
\[
\mathcal{T}\cong S_{0}(p)\times_{\widetilde{S}}S_{0}(p).
\]
In characteristic $0,$ this surface is finite flat of degree $(p+1)(p^{3}+1)$
over $S,$ and represents the Hecke operator $T(p).$ This, therefore,
motivated both the introduction of $\widetilde{S}$ and the study
of the morphism $\widetilde{\pi}$. The study of the moduli space
$\mathcal{T}$ will be carried out in a subsequent paper. Nonetheless,
the foundations are laid down here.

While studying the three moduli spaces $S,S_{0}(p)$ and $\widetilde{S}$,
we discovered a new interesting phenomenon. The generic stratum of
$S$ in characteristic $p$ parametrizes $\mu$-ordinary abelian threefolds.
Although their $p$-divisible groups are all isomorphic, studying
their cotangent spaces we were able to distinguish in the tangent
space of $S$ a certain ``foliation'', amounting in this very simple
example to a line sub-bundle closed under the operation of raising
to power $p$ (see §\ref{tangent bundle}). The link between the cotangent
space of the universal abelian variety and that of $S$ is supplied
by the Kodaira-Spencer map. This foliation extends to the general
supersingular locus of $S,$ but fails to extend, in a way made precise,
to the superspecial points there. Moreover, we found two other ways
to characterize it: the first, as the foliation of ``unramified directions''
(in the sense of \cite{Ru-Sh}) for a map $\overline{\pi}:S_{0}(p)^{(p)}\to S$
derived from the map $\pi$ (Theorem \ref{unramified direction}).
The second, in terms of Moonen's generalized Serre-Tate coordinates
\cite{Mo} (Proposition \ref{Moonen}). Shimura curves embedded in
$S,$ as well as the supersingular curves in $S,$ are integral curves
of this foliation (Theorem \ref{IntCurves}). Does it have any other
global integral curves? We expect this new phenomenon to generalize
to other Shimura varieties of PEL type whose generic stratum is $\mu$-ordinary
but not ordinary; cf. our forthcoming paper \cite{dS-G3} where such a foliation is studied for unitary Shimura varieties of arbitrary signatures.

\subsection*{A summary of the results}

We now describe briefly the content of this paper. Chapter 1 reviews
the three Shimura varieties and their integral models. We explain
the precise relation between the moduli problem with parahoric level
structure as in \cite{Ra-Zi} and the Raynaud condition appearing
in \cite{Bel}. The last section reviews the embeddings of modular
curves and Shimura curves in the Picard modular surface. 

Chapter 2 deals with the Picard modular surface $S$, where the level
at $p$ is a hyperspecial maximal compact. The mod $p$ fiber is smooth,
and its stratification was studied by Vollaard in \cite{Vo}. It consists
of three strata. The dense open stratum $S_{\mu}$ parametrizes $\mu$-ordinary
abelian threefolds. Its complement $S_{ss}$ parametrizes supersingular
ones, and consists (at least when the tame level $N$ is large, depending
on $p$) of Fermat curves of degree $p+1,$ intersecting transversally
at their $\mathbb{F}_{p^{2}}$-rational points. These intersection
points support superspecial abelian threefolds (isomorphic, not only
isogenous, to a product of supersingular elliptic curves), and constitute
the third stratum $S_{ssp}.$ The non-singular locus of the curve
$S_{ss}$ supports supersingular, but not superspecial, abelian threefolds,
and is denoted $S_{gss}.$ This is the intermediate stratum. The number
of its irreducible components was determined in \cite{dS-G1} using
intersection theory on $S$ and a secondary Hasse invariant constructed
there. It turns out to be related to the second Chern number of $S,$
and via a result of Holzapfel, expressible as an $L$-value. Our contribution
to the study of $S$ in the present paper is: (a) We introduce the
foliation $TS^{+}$ in the tangent bundle of $S$, outside $S_{ssp}$,
and prove the results to which we alluded above; (b) We introduce
the blow-up $S^{\#}$ of $S$ at $S_{ssp}$ and give it a modular
interpretation. It has the advantage that the irreducible components
of $S_{ss}$ become, after blowing up, disjoint non-singular Fermat
curves (even when $N$ is small), i.e. all their intersections, including
self-intersections, are resolved. The exceptional divisor at every
blown-up point $x$ is a projective line $E_{x}$ defined over $\mathbb{F}_{p^{2}}.$
The components of $S_{ss}$ intersect $E_{x}$ at points $\zeta$
satisfying $\zeta^{p+1}=-1.$ Embedded Shimura curves, on the other
hand, intersect $E_{x}$ at $\mathbb{F}_{p^{2}}$-rational points
satisfying $\zeta^{p+1}\neq-1.$ The proofs of these results will
have to wait until Theorem \ref{S_0(p)-ssp} and §\ref{embeddedcurvesintersection}.

Chapter 3 is based on chapter III of Bellaïche's thesis \cite{Bel}
and describes the local models for the completed local rings of the
three Shimura varieties, at any point of the special fiber. We are
nevertheless interested not only in the completed local rings \emph{per
se,} but in the maps between them. The theory of local models yields
these maps only modulo $p$th powers of the maximal ideal. This is
evident already in the case of the germ of the map $X_{0}(p)\to X$
between two modular curves, with and without $\Gamma_{0}(p)$-level
structure, at a supersingular point. In this ``baby case'' the map
between the local models is
\[
k[[x]]\hookrightarrow k[[x,y]]/(xy),
\]
which is not even flat. The correct map, however, is known ever since
Kronecker to be
\[
k[[x]]\hookrightarrow k[[x,y]]/((x^{p}-y)(x-y^{p})),
\]
which is finite flat of degree $p+1$. Similar but more serious problems
arise when we study the maps between the completed local rings of
our three Picard surfaces. Luckily, a general theorem of Rudakov and
Shafarevich \cite{Ru-Sh} on the local structure of inseparable maps
between smooth surfaces, allows us to give a partial answer to our
question. In essence, it allows us to determine the maps between the
completed local rings of the \emph{analytic branches} through any
given point. Once again, results of this type have to await the study
of $S_{0}(p)$ and $\widetilde{S}$ in subsequent chapters, where
we relate them also to the foliation $TS^{+}$ mentioned above.

Chapter $4$ is the longest, and deals with the Picard surface $S_{0}(p)$
of Iwahori level structure, and the map $\pi$ from $S_{0}(p)$ to
$S.$ We caution that $\pi$ is neither finite nor flat. The special
fiber of $S_{0}(p)$ consists of vertical and horizontal components
intersecting transversally. There are two horizontal components, multiplicative
and étale. The multiplicative component maps under $\pi$ isomorphically
onto $S^{\#}$. The map from the étale component is purely inseparable
of degree $p^{3}$ and factors through Frobenius. The factored map
$\overline{\pi}_{et}$ is inseparable of degree $p,$ and we show
that its ``field of unramified directions'' is just the foliation
$TS^{+}$, which was defined before intrinsically on $S.$ The vertical
components of $\pi$ are $\mathbb{P}^{1}$-bundles over Fermat curves,
which we call the ``supersingular screens''. Above each superspecial
point $x\in S_{ssp}$ lies in $S_{0}(p)$ a ``comb'', whose base
$F_{x}$ is a $\mathbb{P}^{1}$ along which the two horizontal sheets
of $S_{0}(p)$ meet, and whose ``teeth'' $G_{x}[\zeta]$ belong
to the supersingular screens. For a more precise description we refer
to Theorems \ref{S_0(p)-1}, \ref{S_0(p)-gss} and \ref{S_0(p)-ssp}
and their corollaries.

Chapter 5 deals with $\widetilde{S}$ and the map $\widetilde{\pi}$.
Unlike $\pi,$ this map is finite flat of degree $p+1.$ Here again
there are horizontal and vertical components. This time $\widetilde{\pi}$
is an isomorphism on the étale component of $S_{0}(p)$ and purely
inseparable of degree $p$ on the multiplicative component. The maps
$\pi$ and $\widetilde{\pi}$ allow us to go back and forth between
$S$ and $\widetilde{S}$ and produce maps that we are able to analyze
easily in light of the modular interpretation. On the vertical components
of $S_{0}(p)$ (the supersingular screens) the map $\widetilde{\pi}$
is pretty intricate. We collect some results on it in the last section
of Chapter 5, but leave some other questions unanswered.

The appendix contains some ugly but unavoidable computations with
Dieudonné modules, that would have interrupted the presentation, had
they been left where needed.

\medskip{}

Deformation theory of $p$-divisible groups clearly is a central tool
in this work. Unfortunately, there are at least three traditional
approaches to it: Grothendieck's theory of crystals, contravariant
Dieudonné theory, and covariant Dieudonné-Cartier theory (not counting
displays, $p$-typical curves etc.). We made every effort to remain
faithful to the language and notation used by the various references
cited by us. This resulted, however, in a mixture of the three approaches.
A very useful guide, and a dictionary between the various languages,
can be found in the appendix to \cite{C-C-O}.

\bigskip{}

\textbf{Notation}
\begin{itemize}
\item If $A$ is an abelian scheme over a base $S,$ $A^{t}$ denotes its
dual abelian scheme.
\item If $H$ is a finite flat group scheme over a base $S,$ $H^{D}$ denotes
its Cartier dual.
\item If $S$ is a scheme over $\mathbb{F}_{p}$ we denote by $\Phi_{S}:S\to S$
the absolute Frobenius morphism of $S.$ If $X\to S$ is any scheme,
we denote by $X^{(p)/S},$ or simply by $X^{(p)},$ if no confusion
may arise, the fiber product
\[
X^{(p)}=S\times_{\Phi_{S},S}X
\]
and by $Fr_{X/S}:X\to X^{(p)}$ the unique morphism over $S$ such
that
\[
(\Phi_{S}\times1)\circ Fr_{X/S}=\Phi_{X}.
\]
\item If $A$ is an abelian scheme over $S$ then $\Fr=Fr_{A/S}:A\to A^{(p)}$
is an isogeny (the Frobenius of $A$). The Verschiebung $\Ver:A^{(p)}\to A$
is the isogeny dual to the Frobenius of $A^{t}$.
\item If $\lambda:A\to A^{t}$ is a polarization of an abelian scheme $A$
and $K=\ker\lambda$, we denote by $e_{\lambda}:K\times K\to\mathbb{G}_{m}$
the Mumford pairing on $K$. If $\lambda=n\phi$ where $\phi$ is
a principal polarization, then $e_{\lambda}$ is Weil's $e_{n}$-pairing
associated with $\phi.$
\item $E$ is a quadratic imaginary field, $\mathcal{O}_{E}$ its ring of
integers, $p$ a prime that remains inert in $E,$ $\kappa=\mathcal{O}_{E}/p\mathcal{O}_{E}$
and $\mathcal{O}_{p}$ is the completion of $\mathcal{O}_{E}$ at
$p.$ We write $\sigma$ for the non-trivial automorphism of $E,$
extended to $\mathcal{O}_{p}.$
\item If $R$ is an $\mathcal{O}_{p}$-algebra we denote by $\Sigma$ the
given homomorphism $\mathcal{O}_{p}\to R$ and $\overline{\Sigma}=\Sigma\circ\sigma.$ 
\item If $G$ is a commutative group scheme over a base $S$ we denote by
$\mathcal{O}_{E}\otimes G$ the $S$-group scheme representing the
functor $S'\mapsto\mathcal{O}_{E}\otimes_{\mathbb{Z}}G(S')$. It has
an obvious $\mathcal{O}_{E}$ action.
\item If $X$ is a non-singular algebraic variety over a field $k$ we denote
its tangent bundle by $TX$. The fiber of $TX$ at $x\in X(k)$ (the
tangent space at $x$) will be denoted by $T_{x}X=TX|_{x}.$
\item If $X$ is any scheme we denote by $X^{red}$ the same underlying
space, equipped with the reduced induced subscheme structure.
\end{itemize}
\textbf{Acknowledgments}: We thank Ben Moonen and George Pappas for
helpful discussions. We are grateful to the research institutes IHES,
Bures-sur-Yvette and MFO, Oberwolfach, where part of the research
on this paper has been done, for their hospitality.

\

\

\section{Three integral models with Parahoric level structure}

\subsection{Shimura varieties}

Let $E$ be a quadratic imaginary field. Let $\Lambda=\mathcal{O}_{E}^{3}$,
equipped with the hermitian form

\[
(u,v)=\,^{t}\bar{u}\left(\begin{array}{ccc}
 &  & 1\\
 & 1\\
1
\end{array}\right)v,
\]
which is of signature $(2,1)$ over $\mathbb{R}$. We denote by $e_{0},e_{1},e_{2}$
the three standard basis vectors. Let $\boldsymbol{G}$ be the group
of unitary similtudes $GU(\Lambda,(,))$, regarded as a linear algebraic
group over $\mathcal{\mathbb{Z}}$. The Shimura varieties in the title
will be associated with $\boldsymbol{G}$. More precisely, $G_{\infty}=\boldsymbol{G}(\mathbb{R})$
acts by projective linear transformations on $\mathbb{P}^{2}(\mathbb{C})$.
The bounded symmetric domain
\[
\mathscr{D}=\left\{ (z_{0}:z_{1}:z_{2})|\,\,\overline{z}_{0}z_{2}+\overline{z}_{1}z_{1}+\overline{z}_{2}z_{0}<0\right\} ,
\]
biholomorphic to the unit ball in $\mathbb{C}^{2}$, is preserved
by $G_{\infty}$, which acts on it transitively. Denote by $K_{\infty}$
the stabilizer of the ``center'' $(-1:0:1).$ For any compact open
subgroup $K_{f}\subset\boldsymbol{G}(\mathbb{A}_{f})$ we put $K=K_{\infty}K_{f}\subset\boldsymbol{G}(\mathbb{A})$. 

The Shimura variety $S_{K}$ is a quasi-projective variety over $E$
whose complex points are identified, as a complex manifold, with
\[
S_{K}(\mathbb{C})\simeq\boldsymbol{G}(\mathbb{Q})\setminus\boldsymbol{G}(\mathbb{A})/K\simeq\boldsymbol{G}(\mathbb{Q})\setminus[\mathscr{D}\times\boldsymbol{G}(\mathbb{A}_{f})/K_{f}].
\]

Fix an odd prime $p$ which is inert in $E,$ and let $N\ge3$ be
an integer such that $p\nmid N$. Let $\kappa=\mathcal{O}_{E}/p\mathcal{O}_{E}$
and denote by $\mathcal{O}_{p}$ the ring of integers in the completion
$E_{p}.$ Assume that $K_{f}=K_{p}K^{p}$ where $K^{p}\subset\boldsymbol{G}(\mathbb{A}_{f}^{p})$
is the principal level subgroup of level $N,$ and $K_{p}\subset G_{p}=\boldsymbol{G}(\mathbb{Q}_{p})$. 

In this paper we are interested in three choices of $K_{p}.$ As $p$
is inert in $E,$ $G_{p}$ is non-split, and its semi-simple rank
is 1. Its Bruhat-Tits building is a biregular tree of bi-degree $(p^{3}+1,p+1)$.
The vertices of degree $p^{3}+1$ are stabilized by hyperspecial maximal
compact subgroups of $G_{p}$, which are all conjugate to $K_{p}^{0}=\boldsymbol{G}(\mathbb{Z}_{p}).$
This subgroup is the stabilizer of the standard self-dual lattice
\begin{equation}
\Lambda_{0}=\Lambda\otimes\mathbb{Z}_{p}=\left\langle e_{0},e_{1},e_{2}\right\rangle _{\mathcal{O}_{p}}.\label{eq:Lambda_0}
\end{equation}
The vertices of degree $p+1$ are stabilized by special, but not hyperspecial,
maximal compact subgroups, which are all conjugate to the stabilizer
$\widetilde{K}_{p}^{0}$ of the lattice 
\[
\Lambda_{1}=\left\langle pe_{0},e_{1},e_{2}\right\rangle _{\mathcal{O}_{p}}.
\]
 Note that this is also the stabilizer of $p^{-1}\Lambda_{2}$, the
\emph{dual} lattice with respect to the hermitian pairing, where
\[
\Lambda_{2}=\left\langle pe_{0},pe_{1},e_{2}\right\rangle _{\mathcal{O}_{p}}.
\]
 We call the vertices of degree $p^{3}+1$ vertices of type (hs) and
the ones of degree $p+1$ of type (s). The vertices $v_{0}$ and $\widetilde{v}_{0}$
corresponding to $K_{p}^{0}$ and $\widetilde{K}_{p}^{0}$ are called
the \emph{standard }vertices of the respective types. The oriented
edge $(v_{0},\widetilde{v}_{0})$ is then stabilized by the standard
Iwahori subgroup
\[
K_{p}^{1}=K_{p}^{0}\cap\widetilde{K}_{p}^{0}.
\]

We denote by $S$ (resp. $\,\widetilde{S}$, resp. $\,S_{0}(p)$)
the Shimura variety over $E$ of level $K_{f}=K_{p}K^{p},$ where
$K^{p}$ is as above (of full tame level $N$) and $K_{p}=K_{p}^{0}$
(resp. $\,\widetilde{K}_{p}^{0},$ resp. $\,K_{p}^{1}$). The following
result is well-known.
\begin{prop}
\label{maps}The Shimura varieties S, $\widetilde{S}$ and $S_{0}(p)$
are non-singular quasi-projective surfaces over $E$ and the natural
maps
\[
\pi:S_{0}(p)\to S,\,\,\,\widetilde{\pi}:S_{0}(p)\to\widetilde{S}
\]
are finite étale of degrees $p^{3}+1$ and $p+1$ respectively.
\end{prop}
We denote by $\mathscr{S}$ (resp. $\,\mathscr{\widetilde{S}},$ resp.
$\,\mathscr{S}_{0}(p)$) the integral models of these varieties over
$\mathcal{O}_{p}$ constructed in chapter 6 of \cite{Ra-Zi}. They
are of relative dimension $2,$ $\mathscr{S}$ is smooth over $\mathcal{O}_{p}$,
but the other two are not. The relative surface $\mathscr{S}$ is
the integral model of the Picard modular surface which has been studied
in detail by Vollaard \cite{Vo} §§4-6. See \cite{dS-G1} for related
results. The surface $\mathscr{S}_{0}(p)$ has been studied to some
extent in Bellaïche's thesis \cite{Bel}. Previous to this paper,
little was known about $\mathscr{\widetilde{S}}$, apart from the
general facts that follow from \cite{Ra-Zi}. We review these three
integral models in the next section.

From a general theorem of G\"ortz \cite{Goertz}, or from the computations
of the local models cited in §\ref{local rings}, it follows that
all three integral models are \emph{flat} over $\mathcal{O}_{p}$,
and their \emph{special fibers are reduced}. As we shall later show,
they are also \emph{regular}.

\subsection{The moduli problems}

\subsubsection{\label{Raynaud}The Raynaud condition}

Let $R$ be a commutative $\mathcal{O}_{p}$-algebra and $H$ a finite
flat group scheme over $R$ of rank $p^{2}.$ Assume that we are given
a ring homomorphism $\iota:\mathcal{O}_{E}\rightarrow\End_{R}(H)$,
and that $H$ is killed by $p,$ or, equivalently, $\iota$ factors
through the field $\kappa=\mathcal{O}_{E}/p\mathcal{O}_{E}.$ Locally
on $Spec(R),$ $\mathcal{O}(H)=A$ is free of rank $p^{2}$; the zero section of $H$ is given by an $R$-homomorphism $\epsilon:A\rightarrow R$
whose kernel $I,$ the augmentation ideal, is free of rank $p^{2}-1.$
Letting $a\in\kappa^{\times}$ act on $A$ via $\iota(a)^{*},$ this
becomes a group action, which preserves $I$. Let $\omega:\kappa^{\times}\to\mathcal{O}_{p}^{\times}\to R^{\times}$
be the Teichmüller character, and for $1\le i\le p^{2}-1$ let
\[
I^{(i)}=\{f\in I|\,\forall a\in\kappa^{\times},\,\iota(a)^{*}(f)=\omega^{i}(a)f\}.
\]
Thanks to the fact that $p^{2}-1$ is invertible in $R$, these are
distinct $R$-submodules, and $I$ is their direct sum. Following
\cite{Bel,Ray}, we call $H$ \emph{Raynaud }if each $I^{(i)}$ is
free of rank 1 over $R$. The following facts are easily checked.
\begin{itemize}
\item Let $R\to R'$ be any base change. Then if $H$ is Raynaud, so is
$H\times_{Spec(R)}Spec(R')$. 
\item The converse holds if $Spec(R)$ is connected. In particular, it is
enough to check then the Raynaud condition at one geometric point.
\item The constant group scheme $\mathcal{O}_{E}\otimes\mathbb{Z}/p\mathbb{Z}$
and its dual $\mathcal{O}_{E}\otimes\mu_{p}$ are Raynaud.
\end{itemize}
It follows from the three properties that étale and multiplicative
(dual to étale) group schemes are automatically Raynaud. 

Assume now that $R=k$ is a perfect field containing $\kappa.$ Let
$M=M(H)$ be the covariant Dieudonné module\footnote{We adhere to the conventions of \cite{C-C-O}, Appendix B.3. Our $M(H)$
is denoted there $M_{*}(H).$ $F$ and $V$ can be regarded also as
$\sigma$ or $\sigma^{-1}$-linear maps on $M$. Recall that $V$
is induced by $\Fr:H\to H^{(p)}$ and $F$ is induced by $\Ver:H^{(p)}\to H.$} of $H.$ Since $H$ is killed by $p,$ $M$ is a $2$-dimensional
vector space over $k,$ equipped with linear maps
\[
F\colon M^{(p)}\to M,\,\,V\colon M\to M^{(p)},
\]
where $M^{(p)}=k\otimes_{\sigma,k}M$ and $\sigma(x)=x^{p}$ is the
Frobenius on $k.$ The action of $\kappa$ on $H$ induces an action
of $\kappa$ on $M$; we let $M(\Sigma)$ be the subspace on which
$\kappa$ acts through the natural embedding $\Sigma:\kappa\hookrightarrow k$,
and $M(\overline{\Sigma})$ the subspace on which it acts via $\overline{\Sigma}=\sigma\circ\Sigma.$
Then $M=M(\Sigma)\oplus M(\overline{\Sigma}).$ Note that $(M^{(p)})(\Sigma)=(M(\overline{\Sigma}))^{(p)}$
and vice versa. We call $M$ \emph{balanced} if both $M(\Sigma)$
and $M(\overline{\Sigma})$ are $1$-dimensional.
\begin{lem}
$H$ is Raynaud if and only if $M(H)$ is balanced.
\end{lem}
\begin{proof}
We may assume that $k$ is algebraically closed, as both conditions
are invariant under passage to an algebraic closure. If $H$ is étale,
it is constant, and must then be isomorphic, with the $\mathcal{O}_{E}$
action, to $\mathcal{O}_{E}\otimes\mathbb{Z}/p\mathbb{Z},$ whose
Dieudonné module is evidently balanced. Similarly, if $H$ is multiplicative. 

There remains the local\textendash local case. As a \emph{scheme,
}stripped of the group structure, $H$ is then either (i) $Spec(k[X]/(X^{p^{2}}))$
or (ii) $Spec(k[X,Y]/(X^{p},Y^{p})),$ where the second case occurs
if and only if $H$ is killed by the Frobenius morphism $\Fr\colon H\to H^{(p)}$.
Since $I$ is of codimension 1 and, in the local case, also nilpotent,
it coincides with the maximal ideal of $A=\mathcal{O}(H).$ The cotangent
space at the origin, $I/I^{2}$, is then $k\overline{X}$ in case
(i) and $k\overline{X}\oplus k\overline{Y}$ in case (ii).

In case (i) $\kappa$ may act on the one-dimesional $I/I^{2}$ by
$\Sigma$ or $\overline{\Sigma}$, and so does the group $\kappa^{\times}$
act. Either way, $\kappa^{\times}$ acts on $I^{i}/I^{i+1}$ ($1\le i\le p^{2}-1)$
via $\Sigma^{i}$ (or $\overline{\Sigma}^{i})$, so every character
$\omega^{i}:\kappa^{\times}\to k^{\times}$ must occur in $I$ with
multiplicity 1, and $H$ is automatically Raynaud. But in case (i)
we also have an exact sequence of finite flat $\mathcal{O}_{E}$-group
schemes
\[
0\to H_{1}\to H\overset{\Fr}{\to}H_{1}^{(p)}\to0.
\]
Here $H_{1}=\ker(\Fr\colon H\to H^{(p)})$ is a subgroup scheme of
rank $p,$ and $H_{1}^{(p)}$ is its image. It follows that in case
(i) $M(H)$ is an extension of $M(H_{1})^{(p)}$ by $M(H_{1}),$ so
is automatically balanced.

Case (ii) is the only case where the ``balanced'' condition may
fail. In this case $\Fr$ annihilates $H$ so $V=0$ on $M=M(H)$
and
\[
\Lie(H)=M[V]=M
\]
(see \cite{C-C-O}, B.3.5.6-3.5.7). We find that $M$ is balanced
if and only if $\Lie(H),$ equivalently its dual $I/I^{2}$, is balanced.
If this is the case, i.e. both $\Sigma$ and $\overline{\Sigma}$
occur in $I/I^{2}$, we may choose the variables $X$ and $Y$ so
that $\kappa^{\times}$ acts on $X$ via $\omega$ and on $Y$ via
$\omega^{p},$ so on $X^{i}Y^{j}$ ($i,j<p,$ not both $0$) it acts
via $\omega^{i+jp}$ and every character occurs with multiplicity
1 in $I$. Thus $H$ is Raynaud in this case. If, on the contrary,
$I/I^{2}$ is $\kappa^{\times}$-isotypical, we can not have $\dim I^{(i)}=1$
for every $i,$ and $H$ is not Raynaud.
\end{proof}
Let $H^{D}$ denote the Cartier dual of $H,$ which is also finite
flat of rank $p^{2},$ and endow it by an $\mathcal{O}_{E}$-action
$\iota^{D}:\mathcal{O}_{E}\rightarrow\End_{R}(H^{D})$ via the formula
\[
\iota^{D}(a)=\iota(a)^{t},
\]
i.e. for any $R$-algebra $R'$ and any $x\in H(R'),$ $y\in H^{D}(R')$,
\[
\left\langle x,\iota^{D}(a)y\right\rangle =\left\langle \iota(a)x,y\right\rangle \in(R')^{\times}.
\]

\begin{cor}
$H$ is Raynaud if and only if $H^{D}$ is Raynaud.
\end{cor}
\begin{proof}
$M(H^{D})$ (identified with the \emph{contravariant} Dieudonné module
of $H$) is the $k$-linear dual of $M(H),$ so one is balanced if
and only if the other is.
\end{proof}

\subsubsection{The moduli problem $(S)$}

We now define the three integral models for the Shimura varieties
with parahoric level structure at $p$ as moduli schemes for moduli
problems of PEL type. It is well known and easy to check that in the
generic fiber these moduli problems yield the given Shimura varieties.
For the relation with the models defined by Rapoport and Zink, and
the representability of the moduli problems, see  \ref{Relation to =00005BRa-Zi=00005D}
below. 

The Picard modular surface $S$ has a smooth integral model $\mathscr{S}$
over $\mathcal{O}_{p}$. It is a fine moduli scheme for the moduli
problem which assigns to each $\mathcal{O}_{p}$-algebra $R$ isomorphism
classes of tuples $\underline{A}=(A,\phi,\iota,\eta),$ where
\begin{itemize}
\item $A$ is an abelian $3$-fold over $R$.
\item $\phi\colon A\overset{\sim}{\to}A^{t}$ is a principal polarization.
\item $\iota\colon\mathcal{O}_{E}\to\End_{R}(A)$ is a ring homomorphism,
such that the Rosati involution induced by $\phi$ on $\End_{R}(A)$
preserves its image, and is given on it by $\iota(a)\mapsto\iota(\overline{a}).$
We furthermore require that $\Lie(A)$ becomes an $\mathcal{O}_{E}$-module
of type $(2,1)$ in the sense that it is the direct sum of a locally
free $R$-module of rank 2 on which $\iota(a)_{*}$ acts like the
image of $a$ in $R$, and a locally free rank 1 module on which it
acts like $\overline{a}.$ 
\item $\eta$:$\Lambda/N\Lambda\simeq A[N]$ is a full level-$N$ $\mathcal{O}_{E}$-structure
(recall $p\nmid N\ge3$).
\end{itemize}
Our reference to moduli problems and representability is the comprehensive
volume by Lan. In particular, we refer the reader to the precise definition
of level structure given there (\cite{Lan} 1.3.6.2), and to the condition
of \emph{étale liftability}. In addition to being compatible with
the $\mathcal{O}_{E}$-action, $\eta$ should carry the polarization
pairing 
\[
\left\langle ,\right\rangle :\Lambda/N\Lambda\times\Lambda/N\Lambda\to\mathbb{Z}/N\mathbb{Z}
\]
derived from $(,)$ to the Weil $e_{N}$-pairing induced by $\phi$
on $A[N]\times A[N].$ Part of the data involved in $\eta$ is an
isomorphism between the (étale) target groups of the two pairings:
$\nu_{N}:\mathbb{Z}/N\mathbb{Z}\simeq\mu_{N}$, making the last condition
meaningful. These isomorphisms form a torsor $\underline{\text{Isom}}(\mathbb{Z}/N\mathbb{Z},\mu_{N})$
under $(\mathbb{Z}/N\mathbb{Z})^{\times}$, and in this way $\nu_{N}$
becomes a morphism form $\mathcal{\mathscr{S}}$ to $\underline{\text{Isom}}(\mathbb{Z}/N\mathbb{Z},\mu_{N}),$
regarded as a scheme over $\mathcal{O}_{p}$ of relative dimension
0. We call $\nu_{N}$ the \emph{multiplier morphism}.

\subsubsection{\label{moduliStilde}The moduli problem $(\widetilde{S})$}

The Shimura variety $\widetilde{S}$ has an integral model $\mathscr{\widetilde{S}}$
over $\mathcal{O}_{p}.$ It is a fine moduli scheme for the moduli
problem which assigns to each $\mathcal{O}_{p}$-algebra $R$ isomorphism
classes of tuples $\underline{A}'=(A',\psi,\iota',\eta'),$ where
\begin{itemize}
\item $A'$ is an abelian $3$-fold over $R$.
\item $\psi\colon A'\to A'^{t}$ is a polarization of degree $p^{2}$.
\item $\iota'\colon\mathcal{O}_{E}\to\End_{R}(A')$ is a ring homomorphism,
satisfying the same requirements as for $(S).$ In addition, we require
that $\ker(\psi)$ is preserved by $\iota'(\mathcal{O}_{E})$ and
is Raynaud. 
\item $\eta'$ is a full level-$N$ $\mathcal{O}_{E}$-structure.
\end{itemize}

\subsubsection{\label{S_0(p)}The moduli problem $(S_{0}(p))$}

The Shimura variety $S_{0}(p)$ has an integral model $\mathscr{S}_{0}(p)$
over $\mathcal{O}_{p}.$ It is a fine moduli scheme for the moduli
problem which assigns to each $\mathcal{O}_{p}$-algebra $R$ isomorphism
classes of tuples $(\underline{A},H)=(A,\phi,\iota,\eta,H),$ where
\begin{itemize}
\item $\underline{A}$ is as in $(S)$
\item $H\subset A[p]$ is a Raynaud $\mathcal{O}_{E}$-subgroup scheme of
rank $p^{2},$ which is isotropic for the Weil pairing $e_{p}$ (the
Mumford pairing $e_{p\phi}$ attached to the polarization $p\phi).$
\end{itemize}

\subsubsection{\label{maps-1}The maps between the integral models}

There are projection maps
\[
\pi:\mathscr{S}_{0}(p)\to\mathscr{S},\,\,\,\,\,\,\,\widetilde{\pi}:\mathscr{S}_{0}(p)\to\widetilde{\mathscr{S}}
\]
extending the maps of Proposition \ref{maps}. The map $\pi$ is neither
finite, nor flat anymore. On the moduli problem, it is simply ``forget
$H$''.

The second map $\widetilde{\pi}$ is defined as follows. Pick $(\underline{A},H)\in\mathscr{S}_{0}(p)(R).$
Let $A'=A/H.$ Since $H$ is isotropic for $e_{p\phi},$ its annihilator
in this pairing is a finite flat subgroup scheme $H\subset H^{\perp}\subset A[p]$
and $A[p]/H^{\perp}\simeq H^{D}.$ We claim that $H^{\perp}/H$ is
Raynaud. We may assume that $R=k$ is an algebraically closed field
of characteristic $p$. As both $H$ and $H^{D}$ are Raynaud, $M(H)$
and $M(A[p]/H^{\perp})$ are balanced. It follows that $M(H^{\perp}/H)$
is also balanced, so $H^{\perp}/H$ is Raynaud. The polarization $p\phi$
descends canonically to a polarization $\psi:A'\to(A')^{t}$ whose
kernel is $\ker(\psi)=H^{\perp}/H.$ Its degree is $p^{2}.$ Finally
$\iota'$ and $\eta'$ are defined naturally from $\iota$ and $\eta.$
To check that we obtained a point of $\widetilde{\mathscr{S}}$, we
need only check one non-trivial\footnote{In characteristic 0, or if $H$ is étale, this is obvious, because
the Lie algebra is not changed, but in characteristic $p$ the type
of the Lie algebra may well change under an isogeny.} point, that $\Lie(A')$ is indeed of type $(2,1).$ This can be seen,
using the Raynaud condition, as follows. We may assume again that
$R=k$ is an algebraically closed field containing $\kappa.$ The
exact sequence
\[
0\to H\to A\to A'\to0
\]
yields, in covariant Dieudonné theory, exact sequences\footnote{\emph{A guide for the perplexed:} the covariant Dieudonné modules
of a finite flat group scheme (resp. $p$-divisible group) is defined
as the \emph{contravariant} Dieudonné module of its Cartier (resp.
Serre) dual. From the exact sequence $0\to H^{D}\to A'^{t}\to A^{t}\to0$
we get the top row of the diagram.} and a commutative diagram
\[
\begin{array}{ccccccccc}
0 & \to & M(A) & \to & M(A') & \to & M(H) & \to & 0\\
 &  & \downarrow V &  & \downarrow V &  & \downarrow V\\
0 & \to & M(A)^{(p)} & \to & M(A')^{(p)} & \to & M(H)^{(p)} & \to & 0
\end{array}
\]
where we have abbreviated $M(A)=M(A[p^{\infty}])$ etc. The snake
lemma yields
\[
\begin{array}{cccccc}
0 & \to & M(H)[V] & \to & M(A)^{(p)}/VM(A) & \to\\
 &  & \Vert &  & \Vert\\
0 & \to & \Lie(H) & \to & \Lie(A) & \to
\end{array}
\]

\[
\begin{array}{cccccc}
\to & M(A')^{(p)}/VM(A') & \to & M(H)^{(p)}/VM(H) & \to & 0\\
 & \Vert &  & \Vert\\
\to & \Lie(A') & \to & M(H)^{(p)}/VM(H) & \to & 0
\end{array}
\]
Thus the type of $\Lie(A')$ is also $(2,1)$ if and only if $M(H)[V]$
and $M(H)^{(p)}/VM(H)$ have the same type. But from the exact sequence
\[
0\to M(H)[V]\to M(H)\overset{V}{\to}M(H)^{(p)}\to M(H)^{(p)}/VM(H)\to0
\]
we see that this is the case if and only if $M(H)$ is balanced. We
conclude that $H$ being Raynaud is in fact a \emph{necessary and
sufficient} condition for $A'=A/H$ to be of type $(2,1)$ as well.

We shall see later that in contrast to $\pi,$ the map $\widetilde{\pi}$
\emph{is} finite flat of degree $p+1$.

If we denote by $f:A\to A'$ the canonical homomorphism with kernel
$H,$ and identify $A'^{t}$ with $A/H^{\perp},$ then $f'^{t}:A'^{t}\to A^{t}$
has kernel $A[p]/H^{\perp}$ and
\[
p\phi=f^{t}\circ\psi\circ f.
\]

\subsubsection{\label{Stilde0(p)}The moduli problem $(\widetilde{S}_{0}(p))$}

There is a fourth moduli problem that one can define. It turns out
to be equivalent to $(S_{0}(p))$, yet useful for later calculations
and for the study of the moduli problem $\mathscr{T}$ mentioned in
the introduction.

The moduli problem $(\widetilde{S}_{0}(p))$ assigns to every $\mathcal{O}_{p}$-algebra
$R$ isomorphism classes of tuples $(\underline{A}',J)$ where
\begin{itemize}
\item $\underline{A}'$ is as in $(\widetilde{S}$)
\item $J\subset A'[p]$ is a finite flat $\mathcal{O}_{E}$-subgroup scheme
of rank $p^{4},$ containing $\ker(\psi),$ such that $J/\ker(\psi)$
is Raynaud, and which is maximal isotropic for the Mumford pairing
$e_{p\psi}$.
\end{itemize}
Note that $\deg(p\psi)=p^{8}.$ 
\begin{prop}
\label{J}The moduli problems $(S_{0}(p))$ and $(\widetilde{S}_{0}(p))$
are equivalent, hence $(\widetilde{S}_{0}(p))$ is also represented
by $\mathscr{S}_{0}(p).$ 
\end{prop}
\begin{proof}
To pass from $(\underline{A},H)$ to $(\underline{A}',J)$ define
\[
\underline{A}'=\underline{A}/H,\,\,\,J=A[p]/H,
\]
and observe that $J/\ker(\psi)=A[p]/H^{\perp}$ is Raynaud, and that
$J$ is isotropic (hence, from degree considerations, maximal isotropic)
for $e_{p\psi}.$ To pass from $(\underline{A}',J)$ to $(\underline{A},H)$
define $A=A'/J$, descend $p\psi$ to obtain a principal polarization
$\phi$ on $A,$ and let $H=A'[p]/J.$ We leave to the reader the
verification that we obtain a point of $(\widetilde{S}_{0}(p))$,
as well as that these two constructions are inverse to each other.
\end{proof}
In terms of this new interpretation of $\mathscr{S}_{0}(p)$ the map
$\widetilde{\pi}$ is simply ``forget $J$''.
\begin{prop}
\label{flat and proper}The schemes $\mathscr{S},\mathscr{S}_{0}$(p)
and $\mathscr{\widetilde{S}}$ are regular. They are flat over $\mathcal{O}_{p}$
and their special fibers are reduced. The maps $\pi$ and $\widetilde{\pi}$
are surjective and proper.
\end{prop}
\begin{proof}
The ``flat'' and ``reduced'' assertions follow from the Main Result
of \cite{Goertz}, and from the fact that locally for the étale topology,
a neighborhood of a point in the special fiber of $\mathscr{S}_{0}(p)$
or $\mathscr{\widetilde{S}}$ is isomorphic to an open neighborhood
in the local model. Similarly, regularity follows from the determination
of the completed local rings of the three schemes in \cite{Bel} III.3.4.8.
Although Bellaïche does not use the formalism of \cite{Ra-Zi}, he
builds upon the earlier work of de Jong \cite{dJ2}, which except
for the notation, yields identical results for the completed local
rings as what one would get from the more general theory developed
by Rapoport and Zink.

Properness and surjectivity of $\pi$ and $\widetilde{\pi}$ are usually
proved along with the proof of the representability of $\mathscr{S}_{0}(p).$
For the map $\pi$ it is done in \cite{Bel} III.3.2.3. For the map
$\widetilde{\pi}$ the proof is similar, and we only sketch it. It
is best described with the new interpretation of $\mathscr{S}_{0}(p)$
as representing the moduli problem $(\widetilde{S}_{0}(p)).$ Consider
first a larger moduli problem $(\widetilde{S}_{0}(p)')$ obtained
from $(\widetilde{S}_{0}(p))$ by\emph{ relaxing} the Raynaud condition
on $J/\ker(\psi).$ One proves, following de Jong, that this modified
moduli problem is proper and surjective over $(\widetilde{S}).$ Properness
follows from the valuative criterion. The Raynaud condition is a closed
condition, a fact which secures the properness of $\widetilde{\pi}.$
Surjectivity clearly holds in the generic fiber. By \cite{Goertz}, the
generic fiber of $\mathscr{\widetilde{S}}$ is dense. Since $\widetilde{\pi}$
is already known to be proper, its image must be closed, hence is
everything.
\end{proof}

\subsubsection{Diamond operators}

If $a\in(\mathcal{O}_{E}/N\mathcal{O}_{E})^{\times}$ we denote by
$\left\langle a\right\rangle $ the automorphism of $\mathscr{S}$,
defined on the moduli problem by
\[
\left\langle a\right\rangle (A,\phi,\iota,\eta)=(A,\phi,\iota,\eta\circ a)=(A,\phi,\iota,\iota(a)\circ\eta).
\]
The same notation will be applied to the other moduli schemes.

\subsection{\label{Relation to =00005BRa-Zi=00005D}Translation into the language
of Rapoport and Zink}

The moduli problems that we defined in the preceding sections are
examples of the moduli problems defined in chapter 6 of \cite{Ra-Zi},
although the Raynaud condition is implicit there, as we shall now
explain. It follows (from general results of Kottwitz) that, as has
been claimed above, they are indeed representable by fine moduli schemes
when $N\ge3$. We remark that \cite{Bel} gives an independent proof
of the representability of $(S_{0}(p))$ by proving that it is \emph{relatively
representable} over $(S).$

Using the notation of \cite{Ra-Zi} we take $B=E,\,\mathcal{O}_{B}=\mathcal{O}_{E},\,V=E^{3}$
as before and $b^{*}=\overline{b}.$ Let $\mathcal{L},\mathcal{\widetilde{L}}$
and $\mathcal{L}_{0}(p)$ be the following \emph{self-dual lattice
chains }in $V_{p}$ (see (\ref{eq:Lambda_0})):
\[
\mathcal{L}=\{\cdots\subset p\Lambda_{0}\subset\Lambda_{0}\subset p^{-1}\Lambda_{0}\subset\cdots\},
\]

\[
\mathcal{\widetilde{L}}=\{\cdots\subset p\Lambda_{1}\subset\Lambda_{2}\subset\Lambda_{1}\subset p^{-1}\Lambda_{2}\subset\cdots\},
\]

\[
\mathcal{L}_{0}(p)=\{\cdots\subset p\Lambda_{0}\subset\Lambda_{2}\subset\Lambda_{1}\subset\Lambda_{0}\subset p^{-1}\Lambda_{2}\subset\cdots\}.
\]
View the three lattice chains as categories, inclusions as morphisms.
The moduli problem of type $(\mathcal{L})$, as defined in \cite{Ra-Zi}
Definition 6.9, is clearly our $(S);$ just set $A=A_{\Lambda_{0}}.$ 

The moduli problem of type $(\mathcal{\widetilde{L}})$ is our $(\widetilde{S}).$
Recall the definition of a ``principally polarized $\mathcal{\widetilde{L}}$-set
of abelian varieties of type $(2,1)$'' over a base ring $R$ as
above (\cite{Ra-Zi}, Definition 6.6). First, one is given the $\mathcal{\widetilde{L}}$-set
of abelian schemes $A_{\Lambda_{\bullet}}$ of type $(2,1)$. Then
one gives the ``principal polarization''\footnote{We apologize for the unintentional double meaning attributed to tilde.
We chose to denote the moduli problem $(\widetilde{S})$ with a tilde,
hence it made sense to denote the corresponding lattice chain also
$\mathcal{\widetilde{L}}$. In \cite{Ra-Zi}, passing to the \emph{dual}
$\mathcal{\widetilde{L}}$-set is also denoted by a tilde, hence the
tilde on $\widetilde{A}_{\Lambda_{\bullet}}$.} $\lambda:$$A_{\Lambda_{\bullet}}\simeq\widetilde{A}_{\Lambda_{\bullet}}$.
Note that the $\mathcal{\widetilde{L}}$-set $\widetilde{A}_{\Lambda_{\bullet}}$
is of type $(1,2)$ because $\lambda$ induces the Rosati involution
on the endomorphism ring, hence switches types. We set 
\[
A'=A_{\Lambda_{2}},\,\,\,\,A'^{t}=A_{\Lambda_{2}}^{t}\simeq A_{p^{-1}\Lambda_{2}}^{t}=\widetilde{A}_{\Lambda_{1}},\,\,\,\psi=\lambda\circ\rho_{\Lambda_{1},\Lambda_{2}}.
\]
Then $\psi$ is a polarization in the ordinary sense, of degree $p^{2}=[\Lambda_{1}:\Lambda_{2}]$.
If $R=k$ is an algebraically closed field in characteristic $p,$
\[
M(\ker(\psi))=M(A_{\Lambda_{1}})/M(A_{\Lambda_{2}})=\Lambda_{1}/\Lambda_{2}\otimes k
\]
(\cite{Ra-Zi} 6.10) is balanced, so $\ker(\psi)$ is Raynaud. Conversely,
if we are given data as in $(S)$, thanks to the fact that $\ker(\psi)$
is Raynaud the signature of $A''=A'/\ker(\psi)$ (with $\mathcal{O}_{E}$-action
induced by $\iota'$) is $(2,1)$ (as explained at the end of \ref{S_0(p)}),
so we can define
\[
A_{\Lambda_{2}}=A',\,\,\,A_{\Lambda_{1}}=A'',\,\,\,\rho_{\Lambda_{1},\Lambda_{2}}=\text{the canonical homomorphism},
\]
and ``polarize'' the resulting $\mathcal{\widetilde{L}}$-set by
letting $\lambda$ be the unique type-reversing isomorphism of $\mathcal{\widetilde{L}}$-sets
satisfying $\psi=\lambda\circ\rho_{\Lambda_{1},\Lambda_{2}}.$

The proof that the moduli problem of type $(\mathcal{L}_{0}(p))$
is our $(S_{0}(p))$ is in principle identical, and we only sketch
it. Once again, given the data $(S_{0}(p))$ we construct an $\mathcal{L}_{0}(p)$-set
of abelian varieties by interlacing the previous two constructions.
First, letting $A=A_{\Lambda_{0}}\simeq A_{p\Lambda_{0}}$ we use
the Raynaud condition on $H$ to ensure that $A'=A/H=A_{\Lambda_{2}}$
is of type $(2,1).$ Then we continue and define $A_{\Lambda_{1}}=A/H^{\perp}$
and the polarization $\lambda$ as before.

\subsection{\label{modular curve}Modular curves on the Picard modular surface}

\subsubsection{Embedding the modular curve}

Maps between Shimura data induce maps between Shimura varieties. Here
we have unitary groups of signature $(1,1)$ at infinity mapping (in
many ways) to our $\boldsymbol{G}.$ These group homomorphisms give
rise to morphisms of modular curves and Shimura curves to our Picard
modular surface. Rather than go through the familiar yoga of Shimura
data, we jump straight ahead to the moduli interpretation, thereby
giving the morphism \emph{on the level of integral structures}. We
give only one example, which will be explored in connection with the
geometry of the special fiber at $p$ later on.

Let $B_{0}$ be a fixed elliptic curve defined over $\mathcal{O}_{p}$
with complex multiplication by $\mathcal{O}_{E}$ and CM type $\Sigma$.
Such a curve exists because $(p)$ splits completely in the Hilbert
class field $H$ of $E$, and if $\mathfrak{P}$ is a prime divisor
of $(p)$ in $H,$ $B_{0}$ may be defined over $\mathcal{O}_{H,\mathfrak{P}}=\mathcal{O}_{p}$.
The reduction of $B_{0}$ modulo $p$ is a supersingular elliptic
curve defined over $\kappa$. Let $\phi_{0}:B_{0}\simeq B_{0}^{t}$
be the canonical principal polarization of $B_{0}$, and $\iota_{0}:\mathcal{O}_{E}\simeq\End(B_{0})$.

Recall that $p\nmid N\ge3.$ Let $-D$ be the discriminant of $E$
and $\delta=\sqrt{-D}$ a fixed square root of it in $E$. Assume
for simplicity that $D$ is odd and $(N,D)=1$ (otherwise the construction
below has to be modified slightly). Let $\mathscr{Z}_{0}$ be the
scheme parametrizing $\mathcal{O}_{E}$-isomorphisms $\eta_{0}:\mathcal{O}_{E}/N\mathcal{O}_{E}\simeq B_{0}[N]$.
It is étale of relative dimension $0$ over $\mathcal{O}_{p}$ and
comes with a ``multiplier morphism'' $\nu_{N}$ to $\underline{\text{Isom}}(\mathbb{Z}/N\mathbb{Z},\mu_{N})$.
Write
\[
\underline{B}_{0}=(B_{0},\phi_{0},\iota_{0},\eta_{0})\in\mathscr{Z}_{0}(R)
\]
for an $R$-valued point of $\mathscr{Z}_{0}$.

Let $\mathscr{X}=X_{0}(D;N)$ be the modular curve parametrizing elliptic
curves $B$ with a full level $N$ structure $\nu:(\mathbb{Z}/N\mathbb{Z})^{2}\simeq B[N]$
and a cyclic subgroup scheme $M$ of order $D.$ We view $\mathscr{X}$
as a scheme over $\mathcal{O}_{p}.$ It too comes equipped with a
``multiplier morpism'' $\nu_{N}$ to $\underline{\text{Isom}}(\mathbb{Z}/N\mathbb{Z},\mu_{N}).$
If we identify $\det(B[N])$ with $\mu_{N}$ via the Weil pairing,
then $\nu_{N}=\det\nu$. We remark that $\mathscr{X}$ is neither
complete (the cusps are missing) nor connected ($\det\nu$ is not
fixed), and that every subgroup scheme $M$ as above is étale, since
$D$ is invertible.

Let $R$ be an $\mathcal{O}_{p}$-algebra and $\underline{B}=(B,\nu,M)\in\mathscr{X}(R).$
Let $A_{1}(\underline{B})$ be the abelian surface $\mathcal{O}_{E}\otimes B/(\delta\otimes M)$.
As $D$ is odd, hence square-free, every class in $\mathcal{O}_{E}/\delta\mathcal{O}_{E}$
is represented by a rational integer. As $\delta$ kills $\delta\otimes M$,
this subgroup is $\mathcal{O}_{E}$-stable. It is also maximal isotropic
for the Mumford pairing induced by the canonical degree $D^{2}$ polarization
\[
\phi_{1}':\mathcal{O}_{E}\otimes B\to\delta^{-1}\mathcal{O}_{E}\otimes B=(\mathcal{O}_{E}\otimes B)^{t}.
\]
The identification $\delta^{-1}\mathcal{O}_{E}\otimes B=(\mathcal{O}_{E}\otimes B)^{t}$
is such that the resulting Weil $e_{n}$-pairing between $\mathcal{O}_{E}\otimes B[n]$
and $\delta^{-1}\mathcal{O}_{E}\otimes B[n]$ is
\[
e_{n}(\alpha\otimes u,\beta\otimes v)=e_{n}^{B}(u,v)^{Tr_{E/\mathbb{Q}}(\alpha\overline{\beta})},
\]
where $e_{n}^{B}$ is Weil's $e_{n}$-pairing on $B[n].$ We may therefore
descend $\phi_{1}'$ to obtain a principal polarization $\phi_{1}$
of $A_{1}(\underline{B})$. We let $\iota_{1}$ be the natural action
of $\mathcal{O}_{E}$ as endomorphisms of $A_{1}(\underline{B})$.
It is of type $(\Sigma,\overline{\Sigma}).$ Let 
\[
\eta_{1}=\text{id}\otimes\nu:(\mathcal{O}_{E}/N\mathcal{O}_{E})^{2}\simeq A_{1}[N],
\]
a full level-$N$ $\mathcal{O}_{E}$-structure.

Let $\underline{B}_{0}\in\mathscr{Z}_{0}(R)$ and $\underline{B}\in\mathscr{X}(R)$
be such that $\nu_{N}(\underline{B}_{0})=\nu_{N}(\underline{B}).$
Define 
\[
A(\underline{B}_{0},\underline{B})=B_{0}\times A_{1},\,\,\,\phi=\phi_{0}\times\phi_{1},\,\,\,\iota=\iota_{0}\times\iota_{1},\,\,\,\eta=\eta_{0}\times\eta_{1}.
\]
The structure $\underline{A}(\underline{B}_{0},\underline{B})=(A,\phi,\iota,\eta)\in\mathscr{S}(R).$
Indeed, the assumption $\nu_{N}(\underline{B}_{0})=\nu_{N}(\underline{B})$
allows us to define a multiplier for $\eta$ so that it becomes compatible
with $\phi$, and the rest is obvious. This construction depends functorially
on the input. In this way we have defined a morphism
\[
\mathscr{Z}_{0}\times_{\underline{\text{Isom}}(\mathbb{Z}/N\mathbb{Z},\mu_{N})}\mathscr{X}\to\mathscr{S}.
\]
A minor modification of this construction yields a morphism
\[
\mathscr{Z}_{0}\times_{\underline{\text{Isom}}(\mathbb{Z}/N\mathbb{Z},\mu_{N})}\mathscr{X}_{0}(p)\to\mathscr{S}_{0}(p),
\]
when we add a cyclic subgroup of order $p$ to the level.

\subsubsection{\label{CM points}Endomorphism rings of $\mathbb{\overline{F}}_{p}$
points of $\mathscr{S}$}

Let $D$ be an indefinite quaternion algebra over $\mathbb{Q}$ equipped
with a positive involution $\dagger$ and assume that $E$ embeds
in $D$ as a $\dagger$-stable subfield. Then
\[
D=E\oplus E\xi
\]
where $\xi^{2}>0$ is rational, $\xi a\xi^{-1}=\overline{a}$ for
$a\in E$, $a^{\dagger}=\overline{a}$ and $\xi^{\dagger}=\xi.$ Furthermore
$E$ is the unique quadratic imaginary $\dagger$-stable subfield
of $D.$ Let $\mathcal{O}_{D}$ be a maximal order in $D$ such that
$\mathcal{O}_{D}\cap E=\mathcal{O}_{E}.$ In this situation we may
define the Shimura curve $\mathscr{X}_{D}$ parametrizing abelian
surfaces $A_{1}$ with endomorphisms by $\mathcal{O}_{D},$ a principal
polarization inducing $\dagger$ as the Rosati involution on $D$,
and a full level $N$ structure. Precisely as for the modular curve,
we get a morphism from $\mathscr{Z}_{0}\times_{\underline{\text{Isom}}(\mathbb{Z}/N\mathbb{Z},\mu_{N})}\mathscr{X}_{D}$
to $\mathscr{S}$. Its image in $\mathscr{S}$ is called \emph{an
embedded Shimura curve}.

The points of $\mathscr{S}(\mathbb{\overline{F}}_{p})$ lying on the
embedded Shimura curves all represent non-simple abelian varieties.
There are, however, points $\underline{A}\in\mathscr{S}(\mathbb{\overline{F}}_{p})$
for which $A$ is simple. We use the Honda-Tate theorem to construct
them. More precisely, we construct $A$'s with $\End^{0}(A)=M$ a
CM field of degree 6. 

Let $L$ be a totally real non-Galois cubic field, in which $p$ decomposes
as $\mathfrak{p}\mathfrak{q},$ where $f(\mathfrak{p}/p)=2$ and $f(\mathfrak{q}/p)=1.$
Then $M=LE$ is a degree $6$ CM field and $\mathfrak{p}=P\overline{P}$
splits in $M$, while $\mathfrak{q}=Q$ remains inert. Let $\pi$
be an element of $M$ such that $(\pi)=P^{2h}Q^{h}$, where $h$ kills
the class of $P^{2}Q$ in the class group of $M$. Then $\pi\overline{\pi}=\epsilon p^{2h}$
for a unit $\epsilon$ of $L$. Replacing $\pi$ with $\epsilon^{-1}\pi^{2}$
and $h$ with $2h$ we may assume that $\epsilon=1.$ 

Let $q=p^{2h}.$ Then $\pi$ is a Weil $q$-number, and the Honda-Tate
theorem implies that there exists a simple $3$-dimensional abelian
variety over $\mathbb{F}_{q}$ with $\End(A)$ equal to an order of
$M$, and whose Frobenius of degree $q$ is $\pi$. It is easily seen
that $A$ is absolutely simple. Changing $A$ by an isogeny if necessary
we may assume that $\End(A)\mathcal{\supset O}_{E}$, and that $A$
carries a principal polarization. Of course, $End^{0}(A)=M.$

Since $\End^{0}(A)$, for any $\underline{A}\in\mathscr{S}(\mathbb{\overline{F}}_{p})$,
must contain a $6$-dimensional semi-simple $\mathbb{Q}$-algebra,
we see that the ``most general'' $\overline{\mathbb{F}}_{p}$-point
of $\mathscr{S}$ carries an abelian variety with CM by a field of
degree 6. Generic points of the special fiber of $\mathscr{S}$, by
contrast, have no endomorphisms except for $\iota(\mathcal{O}_{E}).$

\section{The structure of the special fiber of $\mathscr{S}$}

\subsection{Stratification}

Let $k$ be a fixed algebraic closure of $\kappa.$ Since we shall
have no use for the generic fibers of our integral models any more,
\emph{we denote from now on} \emph{by} $S,$ $\widetilde{S}$ \emph{and}
$S_{0}(p)$ \emph{their geometric special fibers, which are schemes
defined over} $k.$ We denote by $\mathcal{A}$ the universal abelian
scheme over $\mathscr{S}$, and by $\mathcal{A}_{x}$ its fiber over
a geometric point $x\in S(k).$

Let $\mathfrak{G}$ be the unique (up to isomorphism) connected 1-dimensional
$p$-divisible group over $k$ of height $2$. It is self-dual of
slope $1/2,$ and isomorphic to the $p$-divisible group of any supersingular
elliptic curve over $k$. Fix an embedding $\lambda:\mathcal{O}_{p}\hookrightarrow\End_{k}(\mathfrak{G})$
in which $a\in\mathcal{O}_{p}$ acts on $\Lie(\mathfrak{G})$ via
the natural homomorphism $\Sigma:\mathcal{O}_{p}\twoheadrightarrow\kappa\hookrightarrow k,$
and denote the pair $(\mathfrak{G},\lambda)$ by $\mathfrak{G}_{\Sigma}.$
Let $\mathfrak{G}_{\overline{\Sigma}}$ be the same $p$-divisible
group with the embedding $\lambda\circ\sigma$, under which the action
of $a\in\mathcal{O}_{p}$ on $\Lie(\mathfrak{G})$ is via $\overline{\Sigma}=\Sigma\circ\sigma.$ 

The following theorem is due to Vollaard \cite{Vo}, in particular
§6. See also \cite{dS-G1} Theorem 2.1.
\begin{thm}
(i) The special fiber $S$ of $\mathscr{\mathscr{S}}$ is the union
of 3 locally closed strata defined over $\kappa$. The $\mu$\emph{-ordinary}
stratum $S_{\mu}$ is open and dense, and $x\in S_{\mu}(k)$ if and
only if
\[
\mathcal{A}_{x}[p^{\infty}]\simeq(\mathcal{O}_{E}\otimes\mu_{p^{\infty}})\times\mathfrak{G}_{\Sigma}\times(\mathcal{O}_{E}\otimes\mathbb{Q}_{p}/\mathbb{Z}_{p})
\]
as $p$-divisible groups with $\mathcal{O}_{E}$-action. Its complement,
$S-S_{\mu}=S_{ss}$ is called the supersingular locus. It is a reduced
(but reducible) complete curve, and if $x\in S_{ss}(k)$ then $\mathcal{A}_{x}[p^{\infty}]$
is supersingular, i.e. its Newton polygon is of constant slope $1/2.$
The \emph{superspecial} locus $S_{ssp}\subset S_{ss}$ is $0$-dimensional
and a point $x\in S_{ssp}(k)$ if and only if 
\[
\mathcal{A}_{x}[p^{\infty}]\simeq\mathfrak{G}_{\Sigma}^{2}\times\mathfrak{G}_{\overline{\Sigma}}
\]
as $p$-divisible groups with $\mathcal{O}_{E}$-action. We let $S_{gss}=S_{ss}-S_{ssp}$
and call it the \emph{general supersingular} locus.

Oort's $a$-number 
\[
a(\mathcal{A}_{x})=\dim_{k}\text{\ensuremath{\Hom}\ensuremath{(\ensuremath{\alpha_{p}},\ensuremath{\mathcal{A}_{x}}[p])}}
\]
is $1$ if $x\in S_{\mu}(k)$ or $x\in S_{gss}(k)$ and $3$ if $x\in S_{ssp}(k).$
Let $\alpha_{p}(\mathcal{A}_{x})$ be the maximal $\alpha_{p}$-subgroup
of $\mathcal{A}_{x}[p].$ The action of $\kappa$ on\textup{ $\Lie(\alpha_{p}(\mathcal{A}_{x}))$}
has signature $\Sigma$ in the first two cases, and $(\Sigma,\Sigma,\overline{\Sigma})$
in the third case.

(ii) If $S'$ is a connected component of $S$ then $S'\cap S_{ss}$
is a connected component of $S_{ss}$. The non-singular locus of $S_{ss}$
is precisely $S_{gss}.$ The irreducible components of $S_{ss}$ are
Fermat curves, whose normalizations are isomorphic to the curve
\[
\mathscr{C}:\,\,x^{p+1}+y^{p+1}+z^{p+1}=0.
\]

(iii) If $N\ge N_{0}(p)$ (an integer depending on $p$) the following
also holds. The irreducible components of $S_{ss}$ are already non-singular,
and isomorphic to $\mathscr{C}.$ Any two of them intersect at most
at one point, and if they intersect, this point belongs to $S_{ssp}(k)$
and the intersection is transversal. There are $p^{3}+1$ superspecial
points on each irreducible component of $S_{ss}$, and there are $p+1$
irreducible components of $S_{ss}$ intersecting transversally at
each $x\in S_{ssp}(k).$
\end{thm}
Let $X$ be the geometric special fiber of the modular curve $\mathscr{X}$
which was constructed\footnote{We abuse notation and call the curve $\mathscr{Z}_{0}\times_{\underline{\text{Isom}}(\mathbb{Z}/N\mathbb{Z},\mu_{N})}\mathscr{X}$
simply $\mathscr{X}.$} in §\ref{modular curve}. It is a non-singular curve in $S.$ The
following corollary is clear from the description of the strata of
$S$.
\begin{cor}
The curve $X$ does not intersect $S_{gss}$. If $\underline{B}\in X(k)$
is such that $\underline{A}(\underline{B})\in S_{ssp}(k)$ then $B$
is supersingular, and vice versa.
\end{cor}

\subsection{\label{tangent bundle}The tangent bundle of $S$}

\subsubsection{The special line sub-bundle $TS^{+}$}

Outside $S_{ssp},$ one may define a natural line sub-bundle $TS^{+}$
of the tangent bundle $TS$ of $S.$ For this recall the following
facts from \cite{dS-G1}. Let $\Omega_{\mathcal{A}/S}$ be the sheaf
of relative differentials of the universal abelian variety $\mathcal{A}$,
and $\omega_{\mathcal{A}}=f_{*}\Omega_{\mathcal{A}/S}$ where $f:\mathcal{A}\to S$
is the structure morphism. Then $\omega_{\mathcal{A}}$ is a rank
3 vector bundle on $S,$ can be identified with the cotangent space
of $\mathcal{A}$ at the origin, and admits a decomposition
\[
\omega_{\mathcal{A}}=\mathcal{P}\oplus\mathcal{L}
\]
into a plane bundle $\mathcal{P}$ on which $\mathcal{O}_{E}$ acts
via $\Sigma$ and a line bundle $\mathcal{L}$ on which it acts via
$\overline{\Sigma}.$ Let $\Phi\colon S\to S$ be the absolute Frobenius
morphism of degree $p,$ and $\mathcal{A}^{(p)}=S\times_{\Phi,S}\mathcal{A}$
the base change of $\mathcal{A}.$ Similar notation will be employed
for the base change of the vector bundles $\mathcal{P}$ or $\mathcal{L}.$
The Verschiebung homomorphism $\Ver_{\mathcal{A}/S}\colon$$\mathcal{A}^{(p)}\to\mathcal{A}$
induces maps
\[
V_{\mathcal{P}}\colon\mathcal{P}\to\mathcal{L}^{(p)},\,\,\,V_{\mathcal{L}}\colon\mathcal{L}\to\mathcal{P}^{(p)},
\]
which, outside $S_{ssp}$, are both of rank 1. At the superspecial
points these maps vanish. Let
\[
\mathcal{P}_{0}=\ker(V_{\mathcal{P}}).
\]
Outside the superspecial points, $\mathcal{P}_{0}$ is a line sub-bundle
of $\mathcal{P}$. Outside $S_{ss}$, the lines $\mathcal{P}_{0}^{(p)}$
and $V_{\mathcal{L}}(\mathcal{L})$ are distinct, but along $S_{gss}$
they coincide. In fact,
\[
V_{\mathcal{P}}^{(p)}\circ V_{\mathcal{L}},
\]
which is a global section of $\mathcal{L}^{p^{2}-1}$, is the Hasse invariant (cf. \cite[Appendix B]{GN}; one of the main contributions of \cite{GN} is the construction of the Hasse invariant for unitary Shimura varieties over totally real fields, which is substantially more difficult),
and $V_{\mathcal{P}}^{(p)}\circ V_{\mathcal{L}}=0$ is the equation
defining $S_{ss}$ as a subscheme of $S.$

The Kodaira-Spencer isomorphism is an isomorphism
\[
\KS:\mathcal{P}\otimes\mathcal{L}\simeq\Omega_{S/k}=TS^{\vee}.
\]

\begin{defn*}
Outside $S_{ssp},$ we define $TS^{+}$ to be the annihilator of the
line bundle $\KS(\mathcal{P}_{0}\otimes\mathcal{L}).$ We call $TS^{+}$
the \emph{special sub-bundle} of $TS.$ By an \emph{integral curve}
of $TS^{+}$ we mean a nonsingular curve $C\subset S-S_{ssp}$ for
which $TS^{+}|_{C}=TC$, i.e. $TS^{+}$ is tangent to $C$. 
\end{defn*}
\begin{thm}
\label{IntCurves}(i) $S_{gss}$ is an integral curve of $TS^{+}.$

(ii) The modular curve $X_{ord}=X\cap S_{\mu}$ is an integral curve
of $TS^{+}.$
\end{thm}
\begin{proof}
Part (i), although not stated there in this form, was proved in \cite{dS-G2}
Proposition 3.11. For (ii) observe that if $x\in X_{ord}(k)\subset S_{\mu}(k)$
then we have the decomposition $\mathcal{A}_{x}=B_{0}\times\mathcal{A}_{1,x}$
where $\mathcal{A}_{1}$ is the abelian surface constructed along
$X$ from the universal elliptic curve $\mathcal{B}$ (and the universal
cyclic subgroup of rank $D)$ as in §\ref{modular curve}. For the
cotangent space we have accordingly
\[
\omega_{\mathcal{A}}|_{x}=\omega_{\mathcal{A}_{x}}=\omega_{B_{0}}\oplus\omega_{\mathcal{A}_{1,x}},
\]
where the first summand is of type $\Sigma$ and the second of type
$(\Sigma,\overline{\Sigma}).$ Thus
\[
\mathcal{P}|_{x}=\omega_{B_{0}}\oplus\omega_{\mathcal{A}_{1,x}}(\Sigma).
\]
As $\mathcal{A}_{1,x}$ is ordinary, $V$ is injective on $\omega_{\mathcal{A}_{1,x}}(\Sigma)$
and 
\[
\mathcal{P}_{0}|_{x}=\ker(V:\mathcal{P}|_{x}\to\mathcal{L}^{(p)}|_{x})=\omega_{B_{0}}.
\]
As $B_{0}$ is constant along $X$, $\KS(\mathcal{P}_{0}\otimes\mathcal{L}|_{x})\subset\Omega_{S/k}|_{x}$
annihilates the line $T_{x}X\subset T_{x}S.$ This proves that $T_{x}X=TS^{+}|_{x}$
as claimed.
\end{proof}
There are many modular curves and Shimura curves like $X$ on $S$,
and by similar arguments they are all integral curves of the special
sub-bundle. It would be interesting to know if these are the only
integral curves of $TS^{+}$ in $S_{\mu}.$ This is an ``André-Oort
type'' question. It would imply, in particular, that there are no
integral curves passing through the CM points constructed in §\ref{CM points}.
Note that in characteristic $p$ there could be many integral curves
tangent to a perfectly nice vector field. The curves $x-c+\lambda y^{p}=0$,
for varying $c$ and $\lambda$, are all tangent to the vector field
$\partial/\partial y$ in $\mathbb{A}^{2}$, and infinitely many of
them pass through any given point. The correct formulation of the
problem should probably ask for curves annihilated by a larger class
of differential operators. Such a class should contain, besides the
differential operators generated by $TS^{+}$, also ``divided powers''.

\subsubsection{A characterization in terms of generalized Serre-Tate coordinates}

We shall now give a \emph{second} characterization of $TS^{+}$, which
relates it to Moonen's work on generalized Serre-Tate coordinates
in $S_{\mu}$. For the following proposition see \cite{Mo}, Example
3.3.2 and 3.3.3(d) (case AU, $r=3,$ $m=1$).
\begin{prop}
\label{Moonen}Let $x\in S_{\mu}.$ Let $\mathfrak{\widehat{G}}$
be the formal group over $k$ associated with the $p$-divisible group
$\mathfrak{G}$ and let $\mathbb{\widehat{G}}_{m}$ be the formal
multiplicative group over $k.$ Then the formal neighborhood $Spf(\widehat{\mathcal{O}}_{S,x})$
of $x$ has a natural structure of a $\mathbb{\widehat{G}}_{m}$-torsor
over $\mathfrak{\widehat{G}}$. In particular, it contains a canonical
copy of $\mathbb{\widehat{G}}_{m}$ sitting over the origin of $\mathfrak{\widehat{G}}$.
\end{prop}
\begin{thm}
Let $x\in S_{\mu}$. Then the line $TS^{+}|_{x}$ is tangent to the
canonical copy of $\mathbb{\widehat{G}}_{m}$ in $Spf(\widehat{\mathcal{O}}_{S,x})$.
\end{thm}
At a point $x$ lying on a modular curve $X$ as above, the canonical
copy of $\mathbb{\widehat{G}}_{m}$ is identified with the classical
Serre-Tate coordinate on $X$, i.e. the formal completion of $X$
at $x$ coincides with $i(\mathbb{\widehat{G}}_{m})$ as a closed
formal subscheme of $Spf(\widehat{\mathcal{O}}_{S,x}).$ In this case
the theorem is a consequence of Theorem \ref{IntCurves}(ii). Our
claim can therefore be viewed as an extension of Theorem \ref{IntCurves}(ii)
to a general $\mu$-ordinary point, at which the formal curve $i(\mathbb{\widehat{G}}_{m})$
may no longer be ``integrated''.
\begin{proof}
Write $\mathbb{\widehat{G}}_{m}=Spf(k[[T-1]])$ with comultiplication
$T\mapsto T\otimes T$, and let $i:\mathbb{\widehat{G}}_{m}\hookrightarrow Spf(\widehat{\mathcal{O}}_{S,x})$
be the embedding of formal schemes given by Proposition \ref{Moonen}.
It sends the closed point $1$ of $\mathbb{\widehat{G}}_{m}$ to $x.$
Let $i_{*}$ be the induced map on tangent spaces
\[
i_{*}:T\mathbb{\widehat{G}}_{m}|_{1}\hookrightarrow TS|_{x}.
\]
We have to show that $i_{*}(\partial/\partial T)$ annihilates $\KS(\mathcal{P}_{0}\otimes\mathcal{L})|_{x}.$
This is equivalent to saying that when we consider the pull back $i^{*}\mathcal{A}$
of the universal abelian scheme to $\mathbb{\widehat{G}}_{m}$, \emph{its}
Kodaira-Spencer map kills $\mathcal{P}_{0}\otimes\mathcal{L}|_{1}$.
For this recall the definition of $\KS=\KS(\Sigma)$ from \cite{dS-G1},
§1.4.2.

Let $\mathfrak{S}=\mathbb{\widehat{G}}_{m}$ and write for simplicity
$\mathcal{A}$ for $i^{*}\mathcal{A}$. We then have the following
commutative diagram
\begin{equation}
\begin{array}{ccc}
\mathcal{P}=\omega_{\mathcal{A}/\mathfrak{S}}(\Sigma) & \hookrightarrow & H_{dR}^{1}(\mathcal{A}/\mathfrak{S})(\Sigma)\\
\downarrow\KS &  & \downarrow\nabla\\
\mathcal{L}^{\vee}\otimes\Omega_{\mathfrak{S}}^{1}\simeq\omega_{\mathcal{A}^{t}/\mathfrak{S}}^{\vee}(\Sigma)\otimes\Omega_{\mathfrak{S}}^{1} & \longleftarrow & H_{dR}^{1}(\mathcal{A}/\mathfrak{S})(\Sigma)\otimes\Omega_{\mathfrak{S}}^{1}
\end{array}\label{eq:KS}
\end{equation}
in which we identified $H^{1}(\mathcal{A},\mathcal{O})$ with $H^{0}(\mathcal{A}^{t},\Omega^{1})^{\vee}$
and used the polarization to identify the latter with $\omega_{\mathcal{A}/\mathfrak{S}}^{\vee}$,
reversing types. Here $\nabla$ is the Gauss-Manin connection, and
the tensor product is over $\mathcal{\widehat{O}_{\mathfrak{S}}}=k[[T-1]].$
Although $\nabla$ is a derivation, $\KS$ is a homomorphism of vector
bundles over $\mathcal{\widehat{O}_{\mathfrak{S}}}$. We shall show
that $\KS(\mathcal{P}_{0})=0,$ where $\mathcal{P}_{0}=\ker(V:\omega_{\mathcal{A}/\mathfrak{S}}\to\omega_{\mathcal{A}/\mathfrak{S}}^{(p)})\cap\mathcal{P}$.

At this point recall the filtration
\[
0\subset Fil^{2}=\mathcal{A}[p^{\infty}]^{m}\subset Fil^{1}=\mathcal{A}[p^{\infty}]^{0}\subset Fil^{0}=\mathcal{A}[p^{\infty}]
\]
of the $p$-divisible group of $\mathcal{A}$ over $\mathfrak{S}.$
The graded pieces are of height 2 and $\mathcal{O}_{E}$-stable. They
are rigid (do not admit non-trivial deformations as $p$-divisible
groups with $\mathcal{O}_{E}$ action) and given by
\[
gr^{2}=\mathcal{O}_{E}\otimes\mu_{p^{\infty}},\,\,\,gr^{1}=\mathfrak{G},\,\,\,gr^{0}=\mathcal{\mathcal{O}}_{E}\otimes\mathbb{Q}_{p}/\mathbb{Z}_{p}.
\]

For any $p$-divisible group $G$ over $\mathfrak{S}$ denote by $\mathbb{D}(G)$
the Dieudonné crystal associated to $G$, and let $D(G)=\mathbb{D}(G)_{\mathfrak{S}},$
\emph{cf.} \cite{Gro}. The $\mathcal{\widehat{O}_{\mathfrak{S}}}$-module
$D(G)$ is endowed with an integrable connection $\nabla$ and the
pair $(D(G),\nabla)$ determines $\mathbb{D}(G).$

In our case, we can identify $D(\mathcal{A}[p^{\infty}])$ with $H_{dR}^{1}(\mathcal{A}/\mathfrak{S})$,
and the connection with the Gauss-Manin connection. The above filtration
on $\mathcal{A}[p^{\infty}]$ induces therefore a filtration $Fil^{\bullet}$
on $H_{dR}^{1}(\mathcal{A}/\mathfrak{S})$ which is preserved by $\nabla$.
Since the functor $\mathbb{D}$ is contravariant, we write the filtration
as
\[
0\subset Fil^{1}H_{dR}^{1}(\mathcal{A}/\mathfrak{S})\subset Fil^{2}H_{dR}^{1}(\mathcal{A}/\mathfrak{S})\subset Fil^{3}=H_{dR}^{1}(\mathcal{A}/\mathfrak{S})
\]
where 
\[
Fil^{i}H_{dR}^{1}(\mathcal{A}/\mathfrak{S})=D(\mathcal{A}[p^{\infty}]/Fil^{i}\mathcal{A}[p^{\infty}]).
\]
For example, $Fil^{1}H_{dR}^{1}(\mathcal{A}/\mathfrak{S})$ is sometimes
referred to as the ``unit root subspace''. As $Fil^{2}\mathcal{A}[p^{\infty}]$
is of multiplicative type, $\ker(V:H_{dR}^{1}(\mathcal{A}/\mathfrak{S})\to H_{dR}^{1}(\mathcal{A}/\mathfrak{S})^{(p)})$
is contained in $Fil^{2}H_{dR}^{1}(\mathcal{A}/\mathfrak{S})$. In
particular,
\[
\mathcal{P}_{0}\subset Fil^{2}H_{dR}^{1}(\mathcal{A}/\mathfrak{S}).
\]
Let $G=\mathcal{A}[p^{\infty}]/\mathcal{A}[p^{\infty}]^{m}$, so that
$Fil^{2}H_{dR}^{1}(\mathcal{A}/\mathfrak{S})=D(G).$ It follows that
in computing $\KS$ on $\mathcal{P}_{0}$ we may use the following
diagram instead of (\ref{eq:KS}):

\begin{equation}
\begin{array}{ccc}
\mathcal{P}_{0} & \hookrightarrow & D(G)(\Sigma)\\
\downarrow\KS &  & \downarrow\nabla\\
\mathcal{L}^{\vee}\otimes\Omega_{\mathfrak{S}}^{1} & \longleftarrow & D(G)(\Sigma)\otimes\Omega_{\mathfrak{S}}^{1}
\end{array}\label{eq:KS2}
\end{equation}

Finally, we have to use the description of the formal neighborhood
of $x$ as given in \cite{Mo}. Since we are considering the pull-back
of $\mathcal{A}$ to $\mathfrak{S}$ only, and not the full deformation
over $Spf(\widehat{\mathcal{O}}_{S,x})$, the $p$-divisible groups
$Fil^{1}\mathcal{A}[p^{\infty}]$, and dually $G=\mathcal{A}[p^{\infty}]/Fil^{2}$,
are \emph{constant} over $\mathfrak{S}$. Thus over $\mathfrak{S}$
\[
G\simeq\mathfrak{G}\times(\mathcal{O}_{E}\otimes\mathbb{Q}_{p}/\mathbb{Z}_{p}),
\]
and $\nabla$ maps $D(\mathfrak{G})$ to $D(\mathfrak{G})\otimes\Omega_{\mathfrak{S}}^{1}$.
Since
\[
\mathcal{P}_{0}=\omega_{\mathfrak{G}}=D(\mathfrak{G})(\Sigma)
\]
as subspaces of $H_{dR}^{1}(\mathcal{A}/\mathfrak{S}),$ 
\[
\nabla(\mathcal{P}_{0})\subset\mathcal{P}_{0}\otimes\Omega_{\mathfrak{S}}^{1}.
\]
The bottom arrow in (\ref{eq:KS2}) comes from the homomorphism
\[
D(G)(\Sigma)\hookrightarrow H_{dR}^{1}(\mathcal{A}/\mathfrak{S})(\Sigma)\overset{pr}{\to}H^{1}(\mathcal{A},\mathcal{O})(\Sigma)\overset{\phi}{\simeq}H^{1}(\mathcal{A}^{t},\mathcal{O})(\overline{\Sigma})=\mathcal{L}^{\vee}.
\]
But the projection $pr$ kills $\mathcal{P}_{0}\subset\omega_{\mathcal{A}/\mathfrak{S}}.$
This concludes the proof.
\end{proof}
We shall later show that the line sub-bundle $TS^{+}$ has a \emph{third}
characterization, in connection with the ramification in the covering
$\pi\colon S_{0}(p)\to S.$ The definitions and the discussion of
this section have obvious generalizations to higher dimensional unitary
Shimura varieties. We intend to address them in a future work.

\subsection{The blow up of $S$ at the superspecial points\label{blow-up}}

We denote by $S^{\#}$ the surface over $k$ which is obtained by
blowing up the superspecial points on $S$. The fiber of $S^{\#}\to S$
above a superspecial point $x$ is a projective line which we denote
by $E_{x}$. It is canonically identified with $\mathbb{P}(T_{x}S).$

Since $S$ has a canonical model over $\kappa$ and the stratum $S_{ssp}$
is defined over $\kappa,$ $S^{\#}$ too has a canonical model over
$\kappa.$ In fact, it is the fine moduli space of a moduli problem
$(S^{\#})$ which is unique to characteristic $p.$ For any $\kappa$-algebra
$R$, $S^{\#}(R)$ classifies isomorphism classes of pairs $(\underline{A},\mathcal{P}_{0})$
where
\begin{itemize}
\item $\underline{A}\in S(R)$
\item $\mathcal{P}_{0}\subset\ker(V:\omega_{A/R}(\Sigma)\to\omega_{A/R}^{(p)}(\Sigma))$
is a line sub-bundle of $\mathcal{P}=\omega_{A/R}(\Sigma$) which
is annihilated by $V$.
\end{itemize}
If no geometric fiber of $A/R$ is superspecial then $\mathcal{P}_{0}$
is unique. At superspecial points, however, $V$ kills $\mathcal{P},$
so the additional data amounts to a choice of a line in the plane
$\mathcal{P}.$

If $N=1$ then $S$ is a stack defined over $\kappa$ and the superspecial
points are $\kappa$-rational. It follows that $\mathcal{P}$ is defined
over $\kappa$ too and we can equip each 
\[
E_{x}\simeq\mathbb{P}(\mathcal{P}|_{x})=\mathbb{P}(\mathcal{P}\otimes\mathcal{L}|_{x})\simeq\mathbb{P}(T_{x}S)
\]
with a canonical $\kappa$-rational structure. If $N>1$ then level
structure at $N$ forces superspecial points to be defined over larger
finite fields, but since $\mathcal{P}$ is independent of this extra
level structure, the tangent space and the exceptional divisor $E_{x}$
still carry a canonical $\kappa$ structure.

In practice we use a coordinate $\zeta$ on $E_{x}$ which is derived
from a particular choice of a basis for the Dieudonné module of $A_{x}$
at $x\in S_{ssp}.$ This will be explained in Theorem \ref{S_0(p)-ssp}
below.

\section{Local structure of the three integral models}

\subsection{\label{Raynaud classification}Raynaud's classification}

Recall that $k$ is our fixed algebraically closed field containing
$\kappa.$ In \cite{Ray} Raynaud classifies the finite flat group
schemes of rank $p^{2}$ over $k,$ which admit an action of $\kappa$
and satisfy the Raynaud condition discussed in \ref{Raynaud}. See
also \cite{Bel}, III.2.3. They are given in the following table.

\bigskip{}

\begin{center}%
\begin{tabular}{|c|c|c|c|c|c|c|c|}
\hline 
$H$ & $(a_{0},b_{0};a_{1},b_{1})$ & $\Lie(H)$ & $\Lie(\alpha_{p}(H))$ & $\alpha$ & $\beta$ & $\gamma$ & strata\tabularnewline
\hline 
$\kappa\otimes\mathbb{Z}/p\mathbb{Z}$ & $(0,1;0,1)$ & $\emptyset$ & $\emptyset$ & $0$ & $2$ & $1$ & $\mu$\tabularnewline
$\kappa\otimes\mu_{p}$ & $(1,0;1,0)$ & $\Sigma,\overline{\Sigma}$ & $\emptyset$ & $2$ & $0$ & $1$ & $\mu$\tabularnewline
$\kappa\otimes$$\alpha_{p}$ & $(0,0;0,0)$ & $\Sigma,\overline{\Sigma}$ & $\Sigma,\overline{\Sigma}$ & $2$ & $2$ & $1,2$ & $ssp$\tabularnewline
$\mathfrak{G}[p]_{\Sigma}$ & $(0,1;1,0)$ & $\Sigma$ & $\Sigma$ & $1$ & $1$ & $1$ & $gss/ssp$\tabularnewline
$\mathfrak{G}[p]_{\overline{\Sigma}}$ & $(1,0;0,1)$ & $\overline{\Sigma}$ & $\overline{\Sigma}$ & $1$ & $1$ & - & -\tabularnewline
$\alpha_{p^{2},\Sigma}$ & $(0,1;0,0)$ & $\Sigma$ & $\Sigma$ & $1$ & $2$ & $2$ & $gss$\tabularnewline
$\alpha_{p^{2},\overline{\Sigma}}$ & $(0,0;0,1)$ & $\overline{\Sigma}$ & $\overline{\Sigma}$ & $1$ & $2$ & - & -\tabularnewline
$\alpha_{p^{2},\Sigma}^{*}$ & $(0,0;1,0)$ & $\Sigma,\overline{\Sigma}$ & $\Sigma$ & $2$ & $1$ & $2$ & $gss$\tabularnewline
$\alpha_{p^{2},\overline{\Sigma}}^{*}$ & $(1,0;0,0)$ & $\Sigma,\overline{\Sigma}$ & $\overline{\Sigma}$ & $2$ & $1$ & - & -\tabularnewline
\hline 
\end{tabular}\end{center}

\bigskip{}

\bigskip{}

\textbf{Explanations}
\begin{itemize}
\item Each group scheme is designated by a vector $(a_{0},b_{0};a_{1},b_{1})$
with entries from $\{0,1\}$ where $a_{0}b_{0}=a_{1}b_{1}=0.$ There
are 9 possibilities. As a scheme $H=Spec(A)$ where $A=k[X,Y]/(X^{p}-b_{0}Y,Y^{p}-b_{1}X).$
The group structure (Hopf algebra structure on $A$) involves the
$a_{i}.$ It is completely determined by the condition that the Cartier
dual $H^{D}$ is obtained by interchanging $a_{0}$ with $b_{0}$,
$a_{1}$ with $b_{1}.$ The twist $\kappa\otimes_{\sigma,\kappa}H$
of $H$ is obtained by interchanging $a_{0}$ with $a_{1},$ and likewise
$b_{0}$ with $b_{1}.$
\item The column $\Lie(H)$ gives the signature of $\kappa$ on $\Lie(H),$
with multiplicities.
\item The column $\Lie(\alpha_{p}(H))$ gives the signature of $\kappa$,
with multiplicities, on the Lie algebra of the maximal $\alpha_{p}$-subgroup
of $H$ (whose dimension is Oort's $a$-number).
\item The invariants $\alpha,\beta$ are defined by 
\[
\alpha=\dim_{k}\Lie(H),\,\,\beta=\dim_{k}\Lie(H^{D}).
\]
 They satisfy
\[
\alpha=2-b_{0}-b_{1},\,\,\beta=2-a_{0}-a_{1}.
\]
The third invariant, $\gamma,$ is not an intrinsic invariant of $H,$
but rather of the way it sits as an isotropic subgroup of $A[p].$
Recall that if $(\underline{A},H)$ is a point of $S_{0}(p)(k),$
we have a filtration
\[
0\subset H\subset H^{\perp}\subset A[p]
\]
with graded pieces $A[p]/H^{\perp}\simeq H^{D}$ and $H^{\perp}/H\simeq\ker\psi$
(see §\ref{maps-1}). We then set $\gamma=\dim_{k}\Lie(H^{\perp}/H)$.
\item Finally, the last column indicates over which of the strata of $S$
such points $(\underline{A},H)$ lie. A hyphen indicates that an $H$
of the given type does not occur as an isotropic subgroup of $A[p]$
for $\underline{A}$ as in $(S).$ This is the contents of the next
lemma.
\end{itemize}
\begin{lem}
The subgroups $\mathfrak{G}[p]_{\overline{\Sigma}},$ $\alpha_{p^{2},\overline{\Sigma}}$
and $\alpha_{p^{2},\overline{\Sigma}}^{*}$ do not occur as isotropic
subgroups of $A[p]$ for any $\underline{A}$ as in $(S).$
\end{lem}
\begin{proof}
We do the first example first. Let $M=M(A[p])$ be the covariant Dieudonné
module of $A[p].$ It is a $6$-dimensional vector space over $k,$
with a $\kappa$ action of signature $(3,3),$ and maps $F:M^{(p)}\to M$
and $V:M\to M^{(p)}.$ The principal polarization $\phi$ induces
a non-degenerate alternating bilinear pairing $\left\langle ,\right\rangle =\left\langle ,\right\rangle _{M}:M\times M\to k$
satisfying, for $a\in\kappa,\,\,x\in M^{(p)},\,\,y,u,v\in M$
\[
\left\langle \iota(a)u,v\right\rangle =\left\langle u,\iota(\overline{a})v\right\rangle 
\]
\[
\left\langle Fx,y\right\rangle _{M}=\left\langle x,Vy\right\rangle _{M^{(p)}}.
\]
By $\left\langle ,\right\rangle _{M^{(p)}}$ we denote the base change
of $\left\langle ,\right\rangle _{M}$ to $M^{(p)}=k\otimes_{\sigma,k}M$.
The first property shows that $M_{0}$ and $M_{1},$ the $\Sigma$-
and $\overline{\Sigma}$-eigenspaces of $\kappa,$ are maximal isotropic
spaces for the pairing. The second property shows that
\[
\Lie(A)=\Lie(A[p])=M[V]=F(M^{(p)})
\]
is another maximal isotropic subspace, which, according to our assumption
on the signature of $A,$ intersects $M_{0}$ in a $2$-dimensional
space, and $M_{1}$ in a line. 

Now let $N=M(H)\subset M$ where $H$ is assumed to be of type $\mathfrak{G}[p]_{\overline{\Sigma}}$
and isotropic. Decompose $N=N_{0}\oplus N_{1}$ according to $\kappa$-type.
Then $\Lie(H)=N_{1}$ is orthogonal to $N_{0}$ (because $N$ is isotropic)
but also to $\Lie(A)_{0}=M_{0}[V]$ (because $\Lie(H)\subset\Lie(A)$
and $\Lie(A)$ is isotropic). Since $N_{0}$ is a line lying outside
the two-dimensional $\Lie(A)_{0},$ we deduce that $N_{1}$ is orthogonal
to all of $M_{0}$, contradicting the non-degeneracy of the pairing.

The argument for $H\simeq\alpha_{p^{2},\overline{\Sigma}}$ is the
same. To rule out $H\simeq\alpha_{p^{2},\overline{\Sigma}}^{*}$ we
need another argument, on the $\alpha_{p}$-subgroup. $\Lie(H)$ alone
does not distinguish it from $\alpha_{p^{2},\Sigma}^{*},$ which,
as we shall see later, does occur as a possible isotropic subgroup.
If $A$ is either $\mu$-ordinary or general supersingular, then the
$\alpha_{p}$-subgroup of $A$ is of rank $p$ and type $\Sigma,$
while the $\alpha_{p}$-subgroup of $\alpha_{p^{2},\overline{\Sigma}}^{*}$
is of rank $p$ and type $\overline{\Sigma}.$ Hence, $\alpha_{p^{2},\overline{\Sigma}}^{*}$
is not isomorphic to a subgroup scheme of $A[p].$ If $A$ is superspecial,
then its $p$-divisible group is $\mathfrak{G}^{3},$ and does not
admit a subgroup scheme of type $\alpha_{p^{2}}^{*}$ at all, because
the kernels of Verschiebung and Frobenius on $A^{(p)}$ coincide,
while $\alpha_{p^{2}}^{*}$ is killed by Frobenius but not by Verschiebung.
\end{proof}

\subsection{\label{local rings}The completed local rings}

\subsubsection{Generalities on local models}

The method of ``local models'' was introduced by de Jong \cite{dJ2}
and Deligne and Pappas \cite{De-Pa}, and developed further by Rapoport
and Zink in \cite{Ra-Zi}. See also \cite{P-R-S} and \cite{C-N}.
For a point $x$ in the special fiber of a given Shimura variety these
authors construct a generalized flag variety, and a point $x'$ on
it, so that suitable étale neighborhoods of $x$ and $x'$ become
isomorphic. This allows them to compute the isomorphism type of the
completed local rings of the original Shimura variety in terms of
linear-algebra data. For the arithmetic schemes $\mathscr{S},$ $\mathscr{S}_{0}(p)$
and $\mathscr{\widetilde{S}}$ these computations were done in \cite{Bel}
III.4.3, and in this section we shall quote results from there, adhering
as much as possible to the notation used by Bellaïche.

The method of local models is \emph{flawed} when it comes to functoriality
with respect to change of level at $p$. This is because Grothendieck's
theory of the Dieudonné crystal, on which it is based, is functorial
in divided power neighborhoods, but not beyond. This flaw appears
already in the case of the modular curve $X_{0}(p)$ mapping to the
$j$-line $X$. At a supersingular point $y\in X_{0}(p)(k)$ mapping
to $x\in X(k)$ we get, for the relation between local models in characteristic
$p$
\[
k[[u]]\hookrightarrow k[[u,v]]/(uv),
\]
while the correct model for the pair $\mathcal{\widehat{O}}_{x}\hookrightarrow\mathcal{\widehat{O}}_{y}$
is known to be, ever since Kronecker,
\[
k[[u]]\hookrightarrow k[[u,v]]/((u^{p}-v)(v^{p}-u)).
\]
Observe that modulo $p$th powers of the maximal ideal (where there
is a canonical divided power structure) the two models are isomorphic,
but over the whole formal neighborhood they are not. The second homomorphism
is finite flat of degree $p+1$ while the first is neither finite
nor flat.

Despite this flaw, relations between local models of Shimura varieties
of PEL type with parahoric level structure suffice to tell us the
relations between cotangent spaces, as well as the relations between
the infinitesimal deformation theories when we vary the level.

\subsubsection{The standard model}

Fix $y=[\underline{A},H]\in\mathscr{S}_{0}(p)(k)$. Let $x=\pi(y)\in\mathscr{S}(k)$
and $\widetilde{x}=\widetilde{\pi}(y)\in\widetilde{\mathscr{S}}(k)$.
Then $x$ is represented by the tuple $\underline{A}=(A,\phi,\iota,\eta)$
and $\widetilde{x}$ by $\underline{A}'=(A',\psi,\iota',\eta')$ where
$A'=A/H$ and $\psi$ is descended from $p\phi,$ i.e. if $h:A\to A'$
is the canonical isogeny with $\ker(h)=H$ then
\[
p\phi=h^{t}\circ\psi\circ h.
\]
Similarly
\[
\iota'(a)\circ h=h\circ\iota(a),\,\,\,\,\,\,\eta'=h\circ\eta.
\]

Associated with the data $(A,\phi,\iota,A',\psi,\iota',h)$ is the
following linear-algebra data. Let
\[
M_{1}=\mathbb{D}(A)_{W(k)},\,\,\,M_{2}=\mathbb{D}(A')_{W(k)}
\]
be the crystalline Dieudonné modules of the two abelian varieties.
Here $\mathbb{D}(A)$ is the (contravariant) Dieudonné crystal associated
to $A,$ \emph{cf. }\cite{Gro}. In this section we use crystalline
deformation theory as in \cite{Bel}. The translation to \emph{covariant}
Cartier-Dieudonné theory, which will be employed in later sections,
is standard (if painful), see the appendix to \cite{C-C-O}.

The modules $M_{i}$ are free $W(k)$-modules of rank 6, and decompose
under the action of $\mathcal{O}_{E}$ as a direct sum of two rank-3
submodules, denoted $M_{i}(\Sigma)$ and $M_{i}(\overline{\Sigma}).$
The isogeny $h$ induces an injective homomorphism
\[
D(h):M_{2}\to M_{1}
\]
respecting the $\mathcal{O}_{E}$-action, whose cokernel is a two-dimensional
vector space over $k$ of type $(1,1)$, as $H$ is Raynaud. The polarizations
result in type-reversing homomorphism
\[
B:M_{1}^{*}\simeq M_{1},\,\,\,B':M_{2}^{*}\to M_{2}
\]
where we have used the canonical identifications of $M_{i}^{*}=\Hom(M_{i},W(k))$
with the crystalline Dieudonné modules of the dual abelian varieties.
Clearly
\[
D(h)\circ B'\circ D(h)^{*}=pB.
\]

Denote by $\mathcal{M}_{1}$ the coherent sheaf on $\mathscr{S}$
which associates to a Zariski open $U$ the module
\[
\mathcal{M}_{1}(U)=\mathbb{D}(\mathcal{A})_{U}
\]
($\mathcal{A}$ being the universal abelian variety over $\mathscr{S})$
and define $\mathcal{M}_{2}$ similarly on $\mathscr{\widetilde{S}}.$
Denote by the same letters their pull-backs to $\mathscr{S}_{0}(p).$
Then the same sort of linear-algebra structure is induced on the sheaves
$\mathcal{M}_{i},$ the map $D(h)$ resulting from the canonical isogeny
\[
h:\mathcal{A}\to\mathcal{A}/\mathcal{H}
\]
where $\mathcal{H}$ is the universal subgroup scheme of $\mathcal{A}$
over $\mathscr{\mathcal{\mathscr{S}}}_{0}(p).$ The following is Théorème
III.4.2.5.3 of \cite{Bel}.
\begin{thm}
(i) There exist $W(k)$-bases $\{e_{1},\ldots,e_{6}\}$ of $M_{1}$
and $\{f_{1},\ldots,f_{6}\}$ of $M_{2}$ such that, if we denote
by $\{e_{i}^{*}\}$ and $\{f_{i}^{*}\}$ the dual bases, the following
properties hold.

(a) $M_{1}(\Sigma)$ is spanned by $\{e_{1},e_{2},e_{3}\},$ $M_{1}(\overline{\Sigma})$
is spanned by $\{e_{4},e_{5},e_{6}\},$ and similarly for $M_{2}.$

(b) The matrices of the homomorphisms $B,B'$ in these bases are given
by
\[
B=\left(\begin{array}{cccccc}
 &  &  &  &  & 1\\
 &  &  &  & 1\\
 &  &  & 1\\
 &  & -1\\
 & -1\\
-1
\end{array}\right),\,\,\,\,\,B'=\left(\begin{array}{cccccc}
 &  &  &  &  & p\\
 &  &  &  & 1\\
 &  &  & 1\\
 &  & -1\\
 & -1\\
-p
\end{array}\right),
\]
i.e. $B(e_{1}^{*})=-e_{6},$ $B'(f_{1}^{*})=-pf_{6}$ etc.

(c) The matrix of $D(h)$ is given by
\[
D(h)=\left(\begin{array}{cccccc}
1\\
 & p\\
 &  & 1\\
 &  &  & p\\
 &  &  &  & 1\\
 &  &  &  &  & 1
\end{array}\right),
\]
i.e. $D(h)(f_{1})=e_{1},\,\,D(h)(f_{2})=pe_{2}$ etc.

(ii) The structure $(\mathcal{M}_{1},\mathcal{M}_{2},B,B',D(h))$
is locally Zariski isomorphic to 
\[
(M_{1},M_{2},B,B',D(h))\otimes_{W(k)}\mathcal{O}.
\]
\end{thm}
\medskip{}

\subsubsection{The Hodge filtration}

Fix $y=[\underline{A},H]\in\mathscr{S}_{0}(p)(k)$ as above. The canonical
isomorphism
\[
M_{1}\otimes_{W(k)}k=\mathbb{D}(A)_{k}\simeq H_{dR}^{1}(A/k)
\]
defines a $3$-dimensional subspace 
\[
\omega_{0}\subset M_{1}\otimes_{W(k)}k
\]
which maps isomorphically to $\omega_{A/k},$ and similarly a $3$-dimensional
subspace $\omega_{0}'\subset M_{2}\otimes_{W(k)}k$ which maps to
$\omega_{A'/k}.$ These subspaces are $\mathcal{O}_{E}$-invariant
of type $(2,1).$ Furthermore, they are isotropic in the sense that
if we denote by $\omega_{0}^{\perp}$ the annihilator of $\omega_{0}$
in $M_{1}^{*}\otimes_{W(k)}k$, and similarly for $\omega_{0}',$
then
\[
B(\omega_{0}^{\perp})=\omega_{0},\,\,\,B'(\omega_{0}'^{\perp})\subset\omega_{0}'.
\]
Equality (rather than inclusion) holds with $B$ because $\phi,$
unlike $\psi,$ is principal. Finally, the map $D(h)$ maps $\omega_{0}'$
to $\omega_{0}.$
\begin{lem}
\label{invariants}(i) The invariants $(\alpha,\beta,\gamma)$ at
the point $y$ are given by the formulae
\[
\alpha=\dim_{k}\omega_{0}/D(h)(\omega_{0}')
\]
\[
\beta=\dim_{k}M_{1}\otimes_{W(k)}k/\left(\omega_{0}+D(h)(M_{2}\otimes_{W(k)}k)\right)
\]
\[
\gamma=\dim_{k}\omega_{0}'/B'(\omega_{0}'^{\perp}).
\]

(ii) $(\alpha,\beta,\gamma)$ form a complete set of invariants of
the structure 
\[
(M_{1}\otimes_{W(k)}k,M_{2}\otimes_{W(k)}k,B,B',D(h),\omega_{0},\omega'_{0}).
\]
 Namely, any two structures (over $k$) of this form having the same
set of invariants $(\alpha,\beta,\gamma)$ are isomorphic.
\end{lem}
\begin{proof}
Part (ii) is an exercise in linear algebra which we leave out to the
reader. In checking it observe that $\alpha$ determines the relative
position of $\omega_{0}'$ and $\ker D(h),$ $\beta$ determines the
relative position of $\omega_{0}$ and $\text{Im}D(h)$, while $\gamma$
is responsible for the relative position of $\ker B'$ and $\omega_{0}'^{\perp}.$
To prove (i) consider the diagram
\[
\begin{array}{ccccccccc}
 &  &  &  &  &  & \omega_{H^{D}}^{\lor}\\
 &  &  &  &  &  & \cap\\
\text{0} & \to & \omega_{A'} & \to & H_{dR}^{1}(A'/k) & \to & \omega_{A'^{t}}^{\lor} & \to & 0\\
 & {\scriptstyle h^{*}} & \downarrow &  & \downarrow &  & \downarrow & {\scriptstyle (h^{t})^{*\vee}}\\
0 & \to & \omega_{A} & \to & H_{dR}^{1}(A/k) & \to & \omega_{A^{t}}^{\lor} & \to & 0\\
 &  & \downarrow\\
 &  & \omega_{H}\\
 &  & \downarrow\\
 &  & 0
\end{array}
\]
This gives the formulae for $\alpha=\dim_{k}\omega_{H}$ and $\beta=\dim_{k}\omega_{H^{D}}=\dim_{k}\text{coker}(h^{t})^{*\lor}$.
The formula for $\gamma$ comes from the fact that if $K=H^{\perp}/H=\ker\psi$
then $\omega_{K}=\omega_{A'}/B'(\omega_{A'^{t}}).$
\end{proof}

\subsubsection{Deformations}

The following is a consequence of the main theorem of \cite{Gro},
characterizing deformations of an abelian variety $A$ (with extra
structure) by means of linear-algebra data. See also \cite{dJ2} and
\cite{Bel}, Proposition III.4.3.6.

Let $\mathcal{C}_{k}$ be the category of local Artinian rings $(R,\mathfrak{m}_{R})$
of residue field isomorphic to $k,$ equipped with an isomorphism
$R/\mathfrak{m}_{R}\simeq k$. Observe that every object of $\mathcal{C}_{k}$
comes with a canonical homomorphism $W(k)\to R.$ 

The local deformation problem $\mathcal{D}$ of the structure $(A,\phi,\iota,A',\psi,\iota',h)_{/k}$
associates to $R\in\mathcal{C}_{k}$ the set $\mathcal{D}(R)$ of
isomorphism classes of similar structures over $R$, equipped with
an isomorphism between their reduction modulo $\mathfrak{m}_{R}$
and the given structure over $k$. It is represented by the formal
scheme $\text{Spf}(\mathcal{\widehat{O}}_{\mathscr{S}_{0}(p),y}).$
The \emph{local model theorem} is the following.
\begin{thm}
\label{local model theorem}The local deformation problem $\mathcal{D}$
is equivalent to the deformation problem $\mathcal{\widetilde{D}}$
which associates to every $(R,\mathfrak{m}_{R})$ as above the set
of structures
\[
(\omega\subset M_{1}\otimes_{W(k)}R,\,\,\omega'\subset M_{2}\otimes_{W(k)}R)
\]
satisfying

(a) $\omega$ and $\omega'$ are rank-3 direct summands, $\mathcal{O}_{E}$-invariant
of type $(2,1),$ reducing modulo $\mathfrak{m}_{R}$ to $\omega_{0}$
and $\omega_{0}'$.

(b) $B(\omega^{\perp})=\omega,\,\,\,B'(\omega'^{\perp})\subset\omega'$.

(c) $D(h)(\omega')\subset\omega.$
\end{thm}
Similar results hold for the moduli problems represented by $\text{{Spf}}(\mathcal{\widehat{O}}_{\mathscr{S},x})$
and $\text{{Spf}}(\mathcal{\widehat{O}}_{\widetilde{\mathscr{S}},\widetilde{x}})$,
obtained by forgetting part of the data.

The theorem allows us to compute, quite easily, the complete local
rings $\mathbf{L}_{y},\,\text{}$$\mathbf{L}_{x}$ and $\mathbf{L}_{\widetilde{x}}$
representing the deformation problem $\widetilde{\mathcal{D}}$, and
deduce isomorphisms
\[
\mathcal{\widehat{O}}_{\mathscr{S}_{0}(p),y}\simeq\mathbf{L}_{y},\,\,\,\mathcal{\widehat{O}}_{\mathscr{S},x}\simeq\mathbf{L}_{x},\,\,\,\mathcal{\widehat{O}}_{\widetilde{\mathscr{S}},\widetilde{x}}\simeq\mathbf{L}_{\widetilde{x}}.
\]
Since the local deformation problems $\widetilde{\mathcal{D}}$ at
$x$ and $\widetilde{x}$ are obtained from the same problem at $y$
by forgetting part of the data, we get canonical homomorphisms
\begin{equation}
\mathbf{L}_{\widetilde{x}}\to\mathbf{L}_{y}\leftarrow\mathbf{L}_{x}\label{eq:diagram}
\end{equation}
between the local models. However, as remarked above, this diagram
is \emph{not }isomorphic to the corresponding diagram of homomorphisms
between the completed local rings of the Picard modular schemes. The
best one can get from the general theory is the following.
\begin{thm}
In the above situation the diagrams $\mathbf{L}_{\widetilde{x}}\to\mathbf{L}_{y}\leftarrow\mathbf{L}_{x}$
and $\mathcal{\widehat{O}}_{\widetilde{\mathscr{S}},\widetilde{x}}\to\mathcal{\widehat{O}}_{\mathscr{S}_{0}(p),y}\leftarrow\mathcal{\widehat{O}}_{\mathscr{S},x}$
become canonically isomorphic after one divides all the local rings
by the $p$th powers of their maximal ideals. In particular, they
induce isomorphic diagrams on cotangent spaces.
\end{thm}

\subsection{Computations}

\subsubsection{Local model diagrams}

Let $W=W(k)$ be the ring of Witt vectors of $k$. The scheme $\mathscr{S}$
is smooth over $W,$ so all its completed local rings are isomorphic
to $\mathbf{L}_{x}=W[[r,s]].$ In the following table we catalog the
diagrams (\ref{eq:diagram}) giving the local models at $x,\widetilde{x}$
and $y,$ and the maps between them.
\begin{prop}
\label{local rings-1}For a suitable choice of local parameters the
local model diagram is given by the following table (where $\mathbf{L}_{x}=W[[r,s]]$)

\medskip{}

\begin{center}%
\begin{tabular}{|c|c|c|c|c|}
\hline 
{\footnotesize{}$H$ }\emph{\footnotesize{}at}{\footnotesize{} $y=[\underline{A},H]$} & {\footnotesize{}$\mathbf{L}_{y}$} & {\footnotesize{}$\mathbf{L}_{\widetilde{x}}$} & \emph{\footnotesize{}maps} & \emph{\footnotesize{}in {[}Bel{]}}\tabularnewline
\hline 
{\footnotesize{}\uline{$\mu$}}\emph{\footnotesize{}\uline{-ord}}{\footnotesize{}:} &  &  &  & \tabularnewline
{\footnotesize{}$\kappa\otimes\mu_{p}$} & {\footnotesize{}$W[[r,s]]$} & {\footnotesize{}$W[[a,b]]$} & {\footnotesize{}$a\mapsto r,\,b\mapsto ps$} & \emph{\footnotesize{}II.1.c}\tabularnewline
{\footnotesize{}$\kappa\otimes\mathbb{Z}/p\mathbb{Z}$} & {\footnotesize{}$W[[a,b]]$} & {\footnotesize{}$W[[a,b]]$} & {\footnotesize{}$r\mapsto pa,\,s\mapsto pb$} & \emph{\footnotesize{}II.3}\tabularnewline
\hline 
\emph{\footnotesize{}\uline{gss}}{\footnotesize{}:} &  &  &  & \tabularnewline
{\footnotesize{}$\mathfrak{G}[p]$} & {\footnotesize{}$W[[a,c]]$} & {\footnotesize{}$W[[a,c]]$} & {\footnotesize{}$r\mapsto a,\,s\mapsto pc$} & \emph{\footnotesize{}II.2}\tabularnewline
{\footnotesize{}$\alpha_{p^{2}}^{*}$} & {\footnotesize{}$W[[r,s,c]]/(cs+p)$} & {\footnotesize{}$W[[a,b,c]]/(bc+p)$} & {\footnotesize{}$a\mapsto cr,\,b\mapsto s$} & \emph{\footnotesize{}I.1.b}\tabularnewline
{\footnotesize{}$\alpha_{p^{2}}$} & {\footnotesize{}$W[[a,b,c]]/(bc+p)$} & {\footnotesize{}$W[[a,b,c]]/(bc+p)$} & {\footnotesize{}$\,r\mapsto pa,\,s\mapsto b$} & \emph{\footnotesize{}I.2}\tabularnewline
\hline 
\emph{\footnotesize{}\uline{ssp}}\emph{\footnotesize{}:} &  &  &  & \tabularnewline
{\footnotesize{}$\mathfrak{G}[p]$} & {\footnotesize{}$W[[a,c]]$} & {\footnotesize{}$W[[a,c]]$} & {\footnotesize{}$r\mapsto a,\,s\mapsto pc$} & \emph{\footnotesize{}II.2}\tabularnewline
{\footnotesize{}$\kappa\otimes\alpha_{p}\,\text{(generic)}$} & {\footnotesize{}$W[[a,b,r]]/(ar+p)$} & {\footnotesize{}$W[[a,b]]$} & {\footnotesize{}$s\mapsto br$} & \emph{\footnotesize{}II.1.a}\tabularnewline
{\footnotesize{}$\kappa\otimes\alpha_{p}\,\,(\sqrt[p+1]{-1})$} & {\footnotesize{}$W[[a,b,r]]/(abr+p)$} & {\footnotesize{}$W[[a,b,c]]/(bc+p)$} & {\footnotesize{}$s\mapsto br,\,c\mapsto ar$} & \emph{\footnotesize{}I.1.a}\tabularnewline
\hline 
\end{tabular}\end{center}

\medskip{}
\end{prop}
\textbf{Explanations}
\begin{itemize}
\item The first column indicates the stratum to which $x$ belongs and the
possible Raynaud types of the subgroup $H$ in the fiber of $\pi$
above $x.$ The parentheses distinguishing the two cases where $H\simeq\kappa\otimes\alpha_{p}$
refer to the value of the coordinate $\zeta$ on the projective line
$F_{x}\subset\pi^{-1}(x)$. This line maps isomorphically to $E_{x}\subset S^{\#}$
and we endow it with the coordinate $\zeta$ as in Section \ref{blow-up}
and Theorem \ref{S_0(p)-ssp} below. The last entry in the table refers
to points where $\zeta^{p+1}=-1,$ ``generic'' refers to all the
rest.
\item The last column refers to the enumeration of the various cases in
Bellaïche's thesis \cite{Bel} III.4.3.8 (\emph{cas.sous-cas.sous-sous-cas}).
\end{itemize}
The table implies that the special fiber $S_{0}(p)$ of $\mathscr{\mathscr{S}}_{0}(p)$
is equidimensional of dimension 2. As we shall see in Theorems \ref{S_0(p)-1}
and \ref{S_0(p)-gss}, it is the union of three smooth surfaces intersecting
transversally. These surfaces are the closures of the strata denoted
below by $Y_{m},Y_{et}$ and $Y_{gss}.$ The first two are irreducible,
but the third has several connected components. The non-singular points
of $S_{0}(p)$, lying on only \emph{one} of these surfaces, support
an $H$ of type $\kappa\otimes\mu_{p},$ $\kappa\otimes\mathbb{Z}/p\mathbb{Z}$
or $\mathfrak{G}[p]$. The points lying on the intersection of two
of them support an $H$ of type $\alpha_{p^{2}}^{*},\alpha_{p^{2}}$
or $\kappa\otimes\alpha_{p}$ (generic). The remaining points, represented
by the last row in the table, are those where all three surfaces meet.

The special fiber $\widetilde{S}$ of $\widetilde{\mathscr{S}}$ is
the union of two smooth surfaces intersecting transversally. One of
them, which is the closure of $\widetilde{\pi}(Y_{m})=\widetilde{\pi}(Y_{et})$,
is irreducible. The other one, which is the closure of $\widetilde{\pi}(Y_{gss})$,
has several connected components. A point $\widetilde{x}=\widetilde{\pi}(y)$
lies on the intersection of these two surfaces if and only if $y$
supports an $H$ of type $\kappa\otimes\alpha_{p}\,\,(\sqrt[p+1]{-1})$,
$\alpha_{p^{2}}^{*}$ or $\alpha_{p^{2}}.$

In the next subsections we work out two sample cases from the table,
explaining how one arrives at the given description of the local model
diagram.

\subsubsection{First example\label{example}}

Assume that $x=\pi(y)$ is a gss point and $y\in S_{0}(p)(k)$ is
such that $H\simeq\alpha_{p^{2},\Sigma}$ (case I.2 in \cite{Bel}).
Here the invariants $(\alpha,\beta,\gamma)=(1,2,2)$. Using Lemma
\ref{invariants} one deduces that we may take, without loss of generality,
\[
\omega_{0}=\left\langle e_{1},e_{3},e_{5}\right\rangle _{k},\,\,\,\omega'_{0}=\left\langle f_{2},f_{3},f_{5}\right\rangle _{k}.
\]
A little computation yields that the most general deformation satisfying
(a) (b) and (c) of Theorem \ref{local model theorem} is given by
\[
\omega=\left\langle e_{1}-se_{2},e_{3}-re_{2},e_{5}+re_{4}+se_{6}\right\rangle _{R}
\]

\[
\omega'=\left\langle f_{2}+cf_{1},f_{3}+acf_{1},f_{5}+af_{4}+bf_{6}\right\rangle _{R},
\]
where $r,s,a,b,c\in\mathfrak{m}_{R}$ satisfy the relations
\[
bc+p=0,\,\,\,\,b=s,\,\,\,\,pa=r.
\]
It follows that
\[
\mathbf{L}_{\widetilde{x}}=W(k)[[a,b,c]]/(bc+p)=\mathbf{L}_{y}\supset\mathbf{L}_{x}=W(k)[[r,s]].
\]
In the special fiber we get
\[
\mathbf{L}_{\widetilde{x}}\otimes_{W(k)}k=k[[a,b,c]]/(bc)=\mathbf{L}_{y}\otimes_{W(k)}k\leftarrow\mathbf{L}_{x}\otimes_{W(k)}k=k[[r,s]]
\]
where $s\mapsto b$ and $r\mapsto0.$
\begin{cor}
The map $\mathcal{\widehat{O}}_{\widetilde{S},\widetilde{x}}\to\mathcal{\widehat{O}}_{S_{0}(p),y}$
is an isomorphism. Identify $\mathcal{\widehat{O}}_{S_{0}(p),y}$
with $\mathbf{L}_{y}\otimes_{W(k)}k.$ There are two analytic branches
of $S_{0}(p)$ through $y,$ given by $c=0$ and $b=0$, namely the
closed embeddings of formal schemes
\[
\mathfrak{W}=\text{Spf}(k[[a,b]])\hookrightarrow\mathfrak{Y}=\text{Spf (}\mathcal{\widehat{O}}_{S_{0}(p),y})\hookleftarrow\text{Spf}(k[[a,c]])=\mathfrak{Z}.
\]
The map $\Omega_{S/k}|_{x}\to\Omega_{\mathfrak{W}/k}|_{y}$ maps $ds\mapsto db,\,\,\,dr\mapsto0.$
The map $\Omega_{S/k}|_{x}\to\Omega_{\mathfrak{Z}/k}|_{y}$ is identically
0.
\end{cor}
\begin{proof}
The map $\mathcal{\widehat{O}}_{\widetilde{\mathscr{S}},\widetilde{x}}\to\mathcal{\widehat{O}}_{\mathscr{S}_{0}(p),y}$
is an isomorphism even before we reduce these rings modulo $p.$ Indeed,
both are $3$-dimensional complete regular local rings, and the map
between them induces an isomorphism on the cotangent spaces $\mathfrak{m}/\mathfrak{m}^{2}$,
hence is an isomorphism. Here we use the fact that the map between
\emph{cotangent spaces} coincides with the corresponding map on the
local models, which happens to be an isomorphism.

The two branches of $\text{Spf}(\mathcal{\widehat{O}}_{S_{0}(p),y})$
can be read off the reduction modulo $p$ of the local model $\mathbf{L}_{y}.$
As both branches are smooth over $k$, and so is the base $S$ at
$x,$ the maps on cotangent spaces are easily calculated from the
local models.
\end{proof}

\subsubsection{Second example\label{example2}}

For our second example assume that $x$ is an ssp point and $y$ is
such that $H\simeq\kappa\otimes\alpha_{p}$ and $\zeta^{p+1}=-1$
(case I.1.a in \cite{Bel}). In this case $(\alpha,\beta,\gamma)=(2,2,2)$
and we may assume that
\[
\omega_{0}=\left\langle e_{1},e_{3},e_{5}\right\rangle _{k},\,\,\,\omega'_{0}=\left\langle f_{2},f_{3},f_{4}\right\rangle _{k}.
\]
The most general deformation satisfying (a) (b) and (c) of Theorem
\ref{local model theorem} is given by
\[
\omega=\left\langle e_{1}-re_{2},e_{3}-se_{2},e_{5}+se_{4}+re_{6}\right\rangle _{R}
\]

\[
\omega'=\left\langle f_{2}+abf_{1},f_{3}+bf_{1},f_{4}+af_{5}+cf_{6}\right\rangle _{R},
\]
where $r,s,a,b,c\in\mathfrak{m}_{R}$ satisfy the relations
\[
bc+p=0,\,\,\,\,s=-rb,\,\,\,\,c=ra.
\]
The local models are therefore
\[
\mathbf{L}_{\widetilde{x}}=W(k)[[a,b,c]]/(bc+p)\to\mathbf{L}_{y}=W(k)[[r,a,b]]/(rab+p)\leftarrow\mathbf{L}_{x}=W(k)[[r,s]]
\]
and the maps between them are given by $c\mapsto ra$, $s\mapsto-rb.$
Modulo $p$th powers of the maximal ideals these are also the maps
between the completed local rings of the Picard modular surfaces at
the corresponding points.

\section{The global structure of $S_{0}(p)$}

As before, fix an algebraic closure $k$ of $\kappa.$ In this section
we concentrate on the structure of the geometric special fiber $S_{0}(p)$
over $k.$

\subsection{The $\mu$-ordinary strata}

\subsubsection{Lots of Frobenii}

Let $Y=S_{0}(p),$ and let
\[
Y^{\sigma}=\Phi_{k}^{*}Y
\]
be its base change under the Frobenius of $k.$ This is a fine moduli
space for tuples $(\underline{A}_{1},H_{1})$ as in the moduli problem
$(S_{0}(p))$ except that the signature of the $\mathcal{O}_{E}$-action
on the Lie algebra of $A_{1}$ is now $(1,2)$ rather than $(2,1).$ 

This $Y^{\sigma}$ carries the universal abelian variety $\mathcal{A}_{1}=\mathcal{A}^{\sigma}=\Phi_{k}^{*}\mathcal{A}.$
It should be distinguished from $\mathcal{A}^{(p)}=\Phi_{Y}^{*}\mathcal{A},$
which lies over $Y$. The same remark and notation applies to the
universal subgroup scheme $H.$ The following diagram illustrates
the situation.

\[
\begin{array}{ccccccc}
\mathcal{A} & \overset{Fr_{\mathcal{A}/Y}}{\longrightarrow} & \mathcal{A}^{(p)} & \longrightarrow & \mathcal{A}^{\sigma} & \longrightarrow & \mathcal{A}\\
 & \searrow & \downarrow & \oblong & \downarrow & \oblong & \downarrow\\
 &  & Y & \overset{Fr_{Y/k}}{\longrightarrow} & Y^{\sigma} & \longrightarrow & Y\\
 &  &  & \searrow & \downarrow & \oblong & \downarrow\\
 &  &  &  & Spec(k) & \overset{\Phi_{k}}{\longrightarrow} & Spec(k)\\
 &  &  &  &  & \searrow & \downarrow\\
 &  &  &  &  &  & Spec(\mathbb{F}_{p})
\end{array}
\]
The three squares are Cartesian. The composition of the arrows in
the three top rows are the maps $\Phi_{\mathcal{A}}$, $\Phi_{Y}$
and $\Phi_{k}.$

Consider now an $R$-valued point $\xi:Spec(R)\to Y$ and let $A=\xi^{*}\mathcal{A}$
be the abelian scheme over $Spec(R)$ represented by $\xi$ (we suppress
the role of $H$ and the PEL structure). Consider
\[
Fr_{Y/k}(\xi)=Fr_{Y/k}\circ\xi:Spec(R)\to Y^{\sigma}.
\]
Then 
\[
\text{\ensuremath{A_{1}=}}Fr_{Y/k}(\xi)^{*}\mathcal{A}_{1}=\xi^{*}Fr_{Y/k}^{*}\Phi_{k}^{*}\mathcal{A}=\xi^{*}\Phi_{Y}^{*}\mathcal{A}=\Phi_{R}^{*}A=A^{(p)}.
\]

In the moduli-problem language this means that for $(\underline{A},H)\in Y(R)$
\[
Fr_{Y/k}((\underline{A},H))=(\underline{A}^{(p)},H^{(p)}).
\]
The Frobenius $Fr_{A/R}$ is an isogeny $Fr_{A/R}:A\to A^{(p)}$.
All of the above holds (forgetting the group $H$) also for $S$ instead
of $S_{0}(p).$

\subsubsection{The $\mu$-ordinary strata}

We study the part of $S_{0}(p)$ lying over $S_{\mu}$, together with
the map $\pi.$ Recall that we work over the algebraically closed
field $k$. We are motivated by the familiar diagram of maps of modular
curves (which takes advantage of the fact that $X_{0}(p)$ is defined
over $\mathbb{F}_{p})$

\begin{center}%
\begin{tabular}{ccc}
$X_{0}(p)_{et}$ & $\overset{Fr_{X/k}}{\to}$ & $X_{0}(p)_{et}$\tabularnewline
$\pi\downarrow$ & $\rho\nearrow$ & $\wr\downarrow\overline{\pi}$\tabularnewline
$X_{0}(1)$ & $=$ & $X_{0}(1)$\tabularnewline
\end{tabular}\end{center}where $\pi(A,H)=A,\,\,\,\,\overline{\pi}(A_{1},H_{1})=A_{1}/H_{1}$
and $\rho(A)=(A^{(p)},A^{(p)}[\Ver]$).
\begin{figure}
\caption{\label{Figure 1}The structure of $S_{0}(p)$}

\bigskip{}

\includegraphics[scale=0.5]{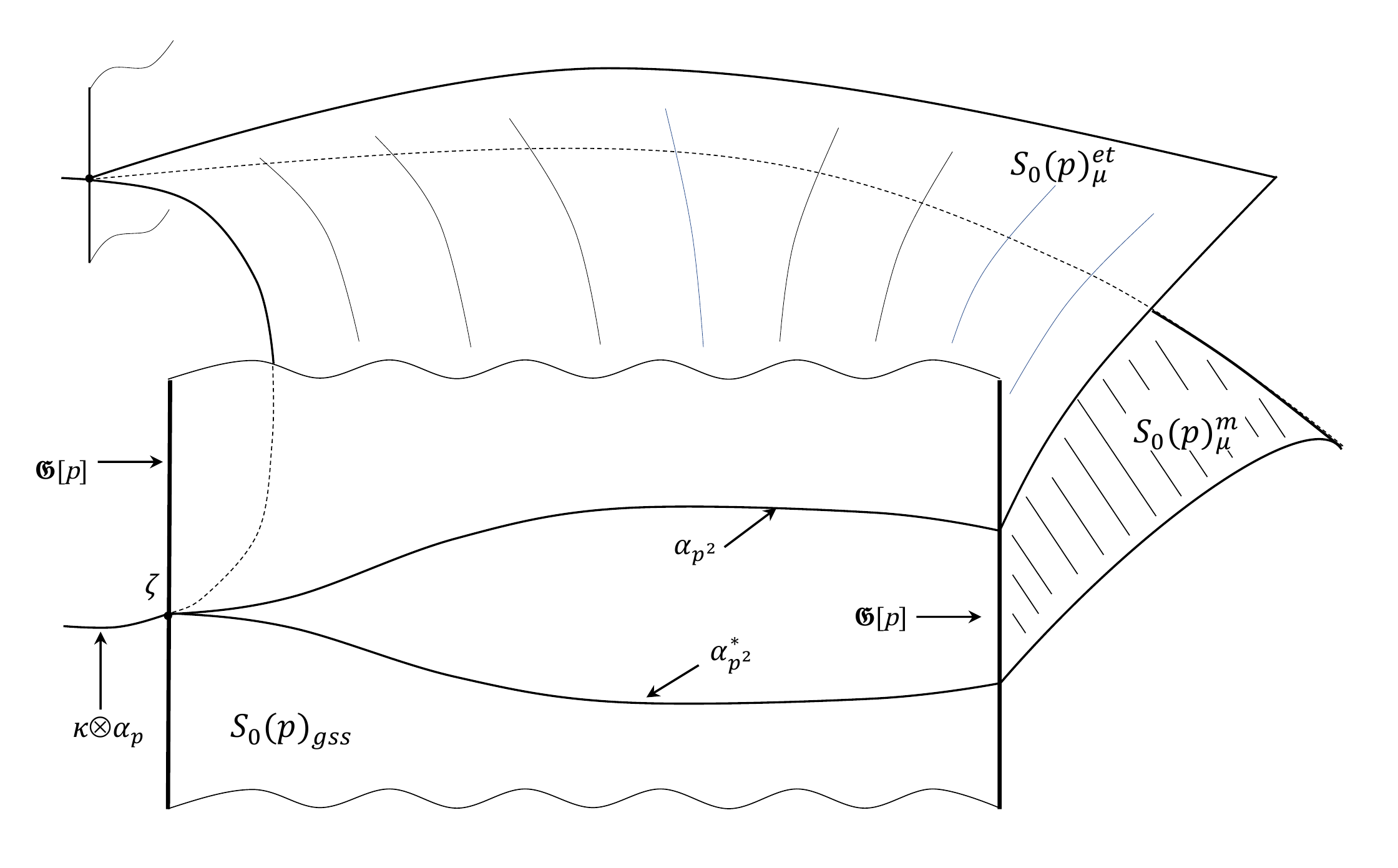}
\end{figure}

\begin{thm}
\label{S_0(p)-1}\emph{(i)} Let $Y_{\mu}=\pi^{-1}(S_{\mu})\subset S_{0}(p).$
Then $Y_{\mu}$ is the disjoint union of two open sets $Y_{m}$ and
$Y_{et}.$ A point $(\underline{A},H)\in S_{0}(p)(k)$ lies on $Y_{m}$
if and only if $H\simeq\kappa\otimes\mu_{p}$, and on $Y_{et}$ if
and only if $H\simeq\kappa\otimes\mathbb{Z}/p\mathbb{Z}$.\smallskip{}

\emph{(ii) }The map $\pi\colon Y_{\mu}\to S_{\mu}$ is finite flat
of degree $p^{3}+1$. Restricted to $Y_{m}$ it yields an isomorphism
\[
\pi_{m}\colon Y_{m}\simeq S_{\mu}.
\]
Its inverse is the section
\[
\sigma_{m}:S_{\mu}\to Y_{m},\,\,\,\,\sigma_{m}(\underline{A})=(\underline{A},A[p]^{m}),
\]
cf. the proof below for the notation.

\smallskip{}

\emph{(iii) }Consider next $Y_{et}$ and its base change $Y_{et}^{\sigma}$
under the Frobenius of $k$. Let $(\underline{A}_{1},H_{1})\in$$Y_{et}^{\sigma}(R)$
for some $k$-algebra $R$. Then there exists a point $\underline{A}\in S_{\mu}(R)$
such that $\underline{A}_{1}\simeq\underline{A}^{(p)}=\Phi_{R}^{*}\underline{A}$.
In fact, let
\[
K_{1}=H_{1}+H_{1}^{\perp}[\Fr],
\]
where $H_{1}^{\perp}$ is the annihilator of $H_{1}$ under the pairing
$e_{p\phi_{1}}$ on $A_{1}[p]$. Then $K_{1}$ is a finite flat, maximal
isotropic, $\mathcal{O}_{E}$-stable subgroup scheme of $A_{1}[p]$.
Let $B=A_{1}/K_{1},$ and descend the polarization, endomorphisms,
and level-$N$ structure from $A_{1}$ to $B.$ Then
\[
\underline{B}^{(p)}\simeq\left\langle p\right\rangle \underline{A}_{1}
\]
 so we may take $\underline{A}=\left\langle p\right\rangle ^{-1}\underline{B}$.
Moreover, under the isomorphism $\underline{A}_{1}\simeq\underline{A}^{(p)}$
\[
K_{1}\simeq A^{(p)}[\Ver].
\]

\smallskip{}

\emph{(iv) }Restricted to $Y_{et},$ $\pi$ yields a map $\pi_{et}$,
which is of degree $p^{3}$ and totally ramified, i.e. $1-1$ on $k$-points.
It factors as
\[
\pi_{et}=\overline{\pi}_{et}\circ Fr_{Y/k}
\]
where $Fr_{Y/k}\colon Y_{et}\to Y_{et}^{\sigma}$ is the relative
Frobenius morphism, and $\overline{\pi}_{et}\colon Y_{et}^{\sigma}\to S_{\mu}$
is totally ramified of degree $p.$

In fact, identify $Y_{et}^{\sigma}$ with the moduli space for tuples
$(\underline{A}_{1},H_{1})$ as before. Let $K_{1}$ and $\underline{A}$
be as in part (iii). Then the following holds:
\begin{equation}
\overline{\pi}_{et}((\underline{A}_{1},H_{1}))=\left\langle p\right\rangle ^{-1}(\underline{A}_{1}/K_{1})=\underline{A}.\label{eq:pi_bar_etale}
\end{equation}
In addition, if $(\underline{A}_{1},H_{1})=Fr_{Y/k}((\underline{A},H))=(\underline{A}^{(p)},H^{(p)})$
for some $(\underline{A},H)\in Y_{\mu}(R)$, then $K_{1}=A^{(p)}[\Ver].$

\smallskip{}

\emph{(v) }For any $R$-valued point $\underline{A}$ of $S_{\mu},$
$H=\Fr(A^{(p)}[\Ver])$ is a finite flat, rank $p^{2},$ isotropic,
Raynaud subgroup scheme of $A^{(p^{2})}[p]$. Furthermore, it is étale.
Define a map
\[
\rho_{et}:S_{\mu}\to Y_{et}^{\sigma^{2}}=Y_{et}
\]
by setting
\[
\rho_{et}(\underline{A})=(\underline{A}^{(p^{2})},\Fr(A^{(p)}[\Ver])).
\]
Then $\rho_{et}$ is finite flat and totally ramified of degree $p$.
We have
\[
\rho_{et}\circ\pi_{et}=Fr_{Y/k}^{2}:Y_{et}\to Y_{et}^{\sigma^{2}}=Y_{et},\,\,\,\,\rho_{et}\circ\overline{\pi}_{et}=Fr_{Y^{\sigma}/k}.
\]

\smallskip{}
\end{thm}
The following diagram summarizes what was said about the maps $\pi_{et}$,$\overline{\pi}_{et},\rho_{et}$.

\bigskip{}

\begin{center}%
\begin{tabular}{ccccccc}
$Y_{et}=Y_{et}$ & {\small{}$\stackrel{Fr_{Y/k}}{\longrightarrow}$} & $\,\,\,Y_{et}^{\sigma}$  & {\small{}$\stackrel{Fr_{Y^{\sigma}/k}}{\longrightarrow}$} & $Y_{et}$ &  & \tabularnewline
 & {\small{}$\pi_{et}\searrow\,\,\,$} & {\small{}$\,\,\,\,\downarrow\overline{\pi}_{et}$} & {\small{}$\,\,\,\,\,\,\nearrow\rho_{et}$} & {\small{}$\,\,\,\,\,\,\downarrow\overline{\pi}_{et}^{\sigma}$} & {\small{}$\,\,\searrow\pi_{et}$} & \tabularnewline
 &  & $S_{\mu}$ & {\small{}$\stackrel{Fr_{S/k}}{\longrightarrow}$} & $S_{\mu}^{\sigma}$ & {\small{}$\stackrel{Fr_{S^{\sigma}/k}}{\longrightarrow}\,\,\,\,\,$} & $S_{\mu}$.\tabularnewline
\end{tabular}\end{center}\bigskip{}

\begin{proof}
(i) Let $Y_{\mu}=\pi^{-1}(S_{\mu}).$ This is an open subset of $S_{0}(p).$
If $R$ is any $k$-algebra and $\underline{A}\in S_{\mu}(R),$ then
the group scheme $A[p]_{/R}$ admits a canonical filtration by finite
flat $\mathcal{O}_{E}$-subgroup schemes
\[
Fil^{3}A[p]=0\subset Fil^{2}A[p]=A[p]^{m}\subset Fil^{1}A[p]=A[p]^{0}\subset Fil^{0}A[p]=A[p].
\]
Here $Fil^{1}$ is the maximal connected subgroup-scheme and is of
rank $p^{4},$ while $Fil^{2}$ is the maximal subgroup scheme of
multiplicative type (connected, with étale Cartier dual), and is of
rank $p^{2}.$ It is also equal to the annihilator of $Fil^{1}$ under
the pairing $e_{p\phi}.$ Moreover, the graded pieces are rigid in
formal neighborhoods. This means that over any Artinian neighborhood
$Spec(R)$ of a point, we have isomorphisms ($gr^{i}=Fil^{i}/Fil^{i+1})$
\[
gr^{2}A[p]\simeq\kappa\otimes\mu_{p},\,\,\,gr^{1}A[p]\simeq\mathfrak{G}[p]_{\Sigma},\,\,\,gr^{0}A[p]\simeq\kappa\otimes\mathbb{Z}/p\mathbb{Z},
\]
as $R$-group schemes with $\mathcal{O}_{E}$-action. We remark that
the filtration and the rigidity of its graded pieces hold for the
whole $p$-divisible group. If $R=k$ (or any other perfect field),
$A[p]$ splits canonically as the product of the three graded pieces.
As these are pairwise non-isomorphic, the only rank-$p^{2}$ $\mathcal{O}_{E}$-subgroup
schemes of $A[p]$ are then the unique copies of $\kappa\otimes\mu_{p},$
$\mathfrak{G}[p]_{\Sigma}$ or $\kappa\otimes\mathbb{Z}/p\mathbb{Z}$
in it. They are all Raynaud. Only the first and the last are isotropic
for the Weil pairing. Thus, if $x\in S_{\mu}(k),$ there are only
two points of $Y_{\mu}(k)$ above $x.$ We call $Y_{m}$ the component
of $Y_{\mu}$ containing the $k$-points $(\underline{A},H)$ where
$H\simeq\kappa\otimes\mu_{p},$ and $Y_{et}$ the component containing
the $k$-points where $H\simeq\kappa\otimes\mathbb{Z}/p\mathbb{Z}.$
That these are indeed connected components follows from the above
mentioned rigidity.

(ii) Let
\[
\sigma_{m}:S_{\mu}\to Y_{m}
\]
be the morphism defined on $R$-points $(R$ any $k$-algebra) by
$\underline{A}\mapsto(\underline{A},A[p]^{m}).$ It is a section of
the map $\pi,$ both $\pi\circ\sigma_{m}$ and $\sigma_{m}\circ\pi$
are the identity maps, hence $\pi$ induces an isomorphism on $Y_{m}$. 

This is not the case on $Y_{et},$ as we can not split the filtration
of $A[p]$ functorially over arbitrary $k$-algebra, only over perfect
fields. Let us prove that $\pi_{et}:Y_{et}\to S_{\mu}$ is finite
flat and totally ramified of degree $p^{3}.$ It follows from the
computations of the completed local rings in §\ref{local rings} that
$Y_{et}$ is non-singular. The map $\pi_{et}$ is quasi-finite and
proper (see Proposition \ref{flat and proper}), hence finite. Any
finite surjective morphism between non-singular varieties is automatically
flat (\cite{Eis} 18.17). In fact, the same argument, using regularity
of the arithmetic schemes, proves that on the scheme $\mathscr{S}_{0}(p)'$
obtained by removing $Y_{ss}=\pi^{-1}(S_{ss})$ from the special fiber
of $\mathscr{S}_{0}(p)$, the map $\pi$ is finite flat to $\mathscr{S}'=\mathscr{S}-S_{ss}.$
Since the degree in the generic fiber is $p^{3}+1$, so must be the
degree in the special fiber. Since $\pi$ was shown to be an isomorphism
on $Y_{m}$, on $Y_{et}$ it is finite flat of degree $p^{3},$ and
of course, totally ramified ($1-1$ on geometric points). For another
proof see \cite{Bel} III.3.5.12.

(iii,iv) Since $Y_{et}^{\sigma}$ is reduced, every $R$-point of
$Y_{et}^{\sigma}$ is a base-change of an $R'$-point under a homomorphism
$R'\to R,$ where $R'$ is reduced. We may therefore assume in the
proof of (iii) and (iv) that $R$ is \emph{reduced}. 

We begin by showing that if $(\underline{A}_{1},H_{1})$ is an $R$-point
of $Y_{et}^{\sigma},$ then $K_{1}=H_{1}+H_{1}^{\perp}[\Fr]$ is a
finite flat subgroup-scheme of rank $p^{3}$ contained in $A_{1}[p].$
It is enough to prove this for the universal abelian scheme $\mathcal{A}_{1}$
over $Y_{et}^{\sigma}$, and its universal subgroup $H_{1}.$ We use
the criterion for flatness, saying that if $f:X'\to X$ is a finite
morphism of schemes, $X$ is reduced, and all the fibers of $f$ have
the same rank, then $f$ is also flat (\cite{Mu}, p.432). By the
open-ness of the flat locus of a morphism, if $X$ is a variety over
a field $k,$ it is enough to check the constancy of the fiber rank
at \emph{closed} points of $X$. We shall use this criterion here
for group schemes over $Y_{et}^{\sigma}$, noting that the base is
a non-singular variety. First, $H_{1}^{\perp}$ is clearly finite
flat of rank $p^{4}$ over $Y_{et}^{\sigma}$ and $H_{1}^{\perp}[\Fr]=H_{1}^{\perp}\cap\mathcal{A}_{1}[\Fr]$
is a closed, hence finite, subgroup scheme. Its fiber rank (over the
closed points of $Y_{et}^{\sigma}$!) is constantly $p$, so it is
also flat. Next, $H_{1}\cap H_{1}^{\perp}[\Fr]=H_{1}[\Fr]=0$. Thus,
as a subgroup functor of $\mathcal{A}_{1}[p],$
\[
H_{1}+H_{1}^{\perp}[\Fr]\simeq(H_{1}\times H_{1}^{\perp}[\Fr])/(H_{1}\cap H_{1}^{\perp}[\Fr])\simeq H_{1}\times H_{1}^{\perp}[\Fr]
\]
is a finite flat group scheme of rank $p^{3}$.

Define $\overline{\pi}_{et}$ to be the morphism sending $(\underline{A}_{1},H_{1})\in Y_{et}^{\sigma}(R)$
to $\left\langle p\right\rangle ^{-1}\underline{B}$ where $B=A_{1}/K_{1}.$
The type of $\Lie(B)$ will now be $(2,1)$, as can be easily checked.
Since $K_{1}$ is a maximal isotropic subgroup scheme for the Weil
pairing on $A_{1}[p],$ the polarization $p\phi_{1}$ on $A_{1}$
descends to a principal polarization of $B.$ The tame level-$N$
structure on $A_{1}$ gives rise to a tame level-$N$ structure on
$B$. This completes the definition of $\overline{\pi}_{et}.$ 

If $(\underline{A}_{1},H_{1})=(\underline{A}^{(p)},H^{(p)})$ for
$(\underline{A},H)\in Y_{et}(R)$, and $R$ is \emph{reduced}, then
$K_{1}$ is of rank $p^{3}$ and killed by $\Ver$, as can be checked
fiber-by-fiber. This shows that
\[
K_{1}=A^{(p)}[\Ver],
\]
hence $A_{1}/K_{1}\simeq A$ via $\Ver:A^{(p)}\to A.$ The polarization
$p\phi_{1}$ descends back to $\phi$ because $\phi_{1}=\phi^{(p)}.$
Finally, if 
\[
\eta:\Lambda/N\Lambda\simeq A[N]
\]
is the level-$N$ structure on $A$ and $\eta_{1}=\eta^{(p)},$ then
\[
\Ver\circ\eta^{(p)}=\left\langle p\right\rangle \circ\eta,
\]
concluding the proof that $\overline{\pi}_{et}(\underline{A}_{1},H_{1})=\underline{A}.$
This holds in particular when $R=k,$ which is enough to prove 
\[
\pi_{et}=\overline{\pi}_{et}\circ Fr_{Y/k}.
\]
We remark that for a reduced $R$, to conclude that $K_{1}=A^{(p)}[\Ver]$
we did not have to know that $H_{1}$ was of the form $H^{(p)},$
only that $A_{1}=A^{(p)}.$ Caution must be exercised when $R$ is
non-reduced though, because it is then possible to have $A^{(p)}\simeq B^{(p)}$
without $A\simeq B$. The isogeny $\Ver$ should be labeled by $A$
or $B$, and the given isomorphism between $A^{(p)}$ and $B^{(p)}$
may not carry $\ker(\Ver_{A})$ to $\ker(\Ver_{B})$. 

In general, applying the same argument to $(A_{1}^{(p)},H_{1}^{(p)})$
implies that
\[
H_{1}^{(p)}+H_{1}^{(p)\perp}[\Fr]=A_{1}^{(p)}[\Ver]
\]
so
\[
B^{(p)}=A_{1}^{(p)}/(H_{1}^{(p)}+H_{1}^{(p)\perp}[\Fr])=A_{1}^{(p)}/A_{1}^{(p)}[\Ver]\simeq A_{1}.
\]
By the remark above, $K_{1}=B^{(p)}[\Ver].$ We emphasize, however,
that the group $H_{1}$ need not be a Frobenius base change of a similar
subgroup of $B.$ To guarantee that the level-$N$ structures also
match we have to twist $\underline{B}$ by the diamond operator $\left\langle p\right\rangle ^{-1}$
and set $\underline{A}=\left\langle p\right\rangle ^{-1}\underline{B}.$
Then $\underline{A}_{1}\simeq\underline{A}^{(p)}.$

(v) The finite subgroup scheme $\mathcal{A}^{(p)}[\Fr]\cap\mathcal{A}^{(p)}[\Ver]$
is flat over $S_{\mu}$, as it has constant fiber rank $p$ and the
base is reduced. The image
\[
\Fr(\mathcal{A}^{(p)}[\Ver])\subset\mathcal{A}^{(p^{2})}[p],
\]
is isomorphic to the quotient of $\mathcal{A}^{(p)}[\Ver]$ by $\mathcal{A}^{(p)}[\Fr]\cap\mathcal{A}^{(p)}[\Ver],$
hence is also finite and flat of rank $p^{2}.$ It is isotropic, $\mathcal{O}_{E}$-stable
and Raynaud. By base change from the universal case, for any $R$-valued
point $\underline{A}$ of $S_{\mu},$ $H=\Fr(A^{(p)}[\Ver])$ is a
finite flat, rank $p^{2},$ isotropic, Raynaud subgroup scheme of
$A^{(p^{2})}[p]$. It is easily seen to be étale. Since $\rho_{et}$
is defined functorially in terms of the moduli problem, it is a well-defined
morphism.

It is enough to verify the equality $\rho_{et}\circ\pi_{et}=Fr_{Y/k}^{2}$
on $k$-valued points $(\underline{A},H)\in Y_{et}(k),$ namely that
\[
\Fr(A^{(p)}[\Ver])=H^{(p^{2})},
\]
but if $A$ is $\mu$-ordinary this is clear. The relation $\rho_{et}\circ\overline{\pi}_{et}=Fr_{Y^{\sigma}/k}$
follows from $\rho_{et}\circ\pi_{et}=Fr_{Y/k}^{2}$ since $\pi_{et}=\overline{\pi}_{et}\circ Fr_{Y/k}$
and $Fr_{Y/k}$ is faithfully flat. The remaining assertions on $\rho_{et}$
also follow from this relation.
\end{proof}
\begin{cor}
Over $Y_{et}$ the universal abelian scheme $\mathcal{A}\simeq\mathcal{A}_{1}^{(p)}=Y_{et}\times_{\Phi_{Y},Y_{et}}\mathcal{A}_{1}$
for another abelian scheme $\mathcal{A}_{1}$ of type $(1,2).$
\end{cor}
\begin{proof}
In part (iii) of the theorem we showed the same for the universal
abelian variety $\mathcal{A}_{1}$ over $Y_{et}^{\sigma}.$ The corollary
follows by base-changing back to $Y_{et}$, or by repeating the arguments
throughout with type $(1,2)$ replacing type $(2,1).$
\end{proof}

\subsubsection{A lemma on ramification}

Before we continue our study of $Y_{\mu}$ we need the following result.
\begin{lem}
\label{Liedtke}Let $\pi:Y\to X$ be a finite flat totally ramified
morphism of degree $p$ between non-singular surfaces over $k$, an
algebraically closed field of characteristic $p$. Let $\pi(y)=x.$
Then there exist local parameters $u,v$ at $y\in Y$ so that $\pi^{*}:\mathcal{\widehat{O}}_{X,x}\hookrightarrow\mathcal{\widehat{O}}_{Y,y}$
is
\[
k[[u^{p},v]]\hookrightarrow k[[u,v]].
\]
The class of $u^{p}$ modulo $\mathfrak{\widehat{m}}_{X,x}^{2}$ spans
$\ker(\pi^{*}:\Omega_{X/k}|_{x}\to\Omega_{Y/k}|_{y}),$ and is therefore
independent of any choice.
\end{lem}
\begin{proof}
See \cite{Ru-Sh} Theorem 4, and the Corollary at the bottom of p.
1215 there.
\end{proof}
\begin{defn*}
We call the line in $T_{x}X$ which is the annihilator of $\ker(\pi^{*}:\Omega_{X/k}|_{x}\to\Omega_{Y/k}|_{y})$
the \emph{unramified direction }at\emph{ $x$,} and denote it by $T_{x}X^{ur}.$\emph{
}Then $TX^{ur}$ is a line sub-bundle of $TX.$ 
\end{defn*}
If $C\subset X$ is a non-singular curve such that for every $x\in C$
\[
T_{x}C=T_{x}X^{ur}\subset T_{x}X
\]
(an \emph{integral curve} for $TX^{ur}$), then $\pi:\pi^{-1}(C)^{red}\to C$
is indeed unramified, hence an isomorphism, because $\pi^{*}$ is
injective on
\[
\Omega_{C/k}=\Omega_{X/k}/TC^{\perp}=\Omega_{X/k}/\ker(\pi^{*}).
\]

\subsubsection{The unramified direction of $\overline{\pi}_{et}$}

The morphism $\pi_{et}$ is ``too ramified'', and we study it via
the factorization $\pi_{et}=\overline{\pi}_{et}\circ Fr_{Y/k}.$ Since
$\overline{\pi}_{et}$ is of degree $p,$ it admits, as we have just
seen, an ``unramified direction''. In §\ref{tangent bundle} we
have defined the special sub-bundle $TS^{+}$ in $TS$ outside the
superspecial locus. We shall now show that over $S_{\mu}$ it coincides
with the sub-bundle of unramified directions for $\overline{\pi}_{et}$.
Thus the latter can be defined intrinsically in terms of the automorphic
vector bundles on $S,$ without any reference to the covering $\pi.$
\begin{thm}
\label{unramified direction}Let $x=\pi_{et}(y)=\overline{\pi}_{et}(y^{(p)})\in S_{\mu}.$
The unramified direction at $x$ for the map $\overline{\pi}_{et}$
is $T_{x}S^{+}.$ Equivalently, under the Kodaira-Spencer isomorphism
\[
\ker(\Omega_{S_{\mu}/k}\to\Omega_{Y_{et}^{\sigma}/k})=\KS(\mathcal{P}_{0}\otimes\mathcal{L}).
\]
\end{thm}
\begin{proof}
More precisely, we need to prove that over $Y_{et}^{\sigma}$
\[
\ker(\overline{\pi}_{et}^{*}\Omega_{S_{\mu}/k}\to\Omega_{Y_{et}^{\sigma}/k})=\overline{\pi}_{et}^{*}(\KS(\mathcal{P}_{0}\otimes\mathcal{L})).
\]
In parts (iii) and (iv) of Theorem \ref{S_0(p)-1} we have seen that
if we denote by $\mathcal{A}_{1}$ the universal abelian scheme over
$Y_{et}^{\sigma}$ then $\mathcal{A}_{1}=\mathcal{B}^{(p)}$, where
$\mathcal{B}=\overline{\pi}_{et}^{*}\mathcal{A},$ and the morphism
$\overline{\pi}_{et}$ is induced from $\Ver:\mathcal{A}_{1}\to\mathcal{B}$,
followed by $\left\langle p\right\rangle ^{-1}$ on the level-$N$
structure.

Consider the abelian scheme $\mathcal{C}=\mathcal{A}_{1}/\mathcal{H}_{1}$
(over $Y_{et}^{\sigma}$) where $\mathcal{H}_{1}$ is the universal
étale subgroup scheme of $\mathcal{A}_{1}.$ The isogeny\textbf{ $\Ver:\mathcal{A}_{1}\to\mathcal{B}$}
factors as
\[
\Ver:\mathcal{A}_{1}\overset{\psi}{\to}\mathcal{C}\overset{\varphi}{\to}\mathcal{B}
\]
where $\psi$ is the isogeny with kernel $\mathcal{H}_{1}$ and $\varphi$
the isogeny with kernel $\mathcal{A}_{1}[\Ver]/\mathcal{H}_{1}.$
Notice that although $\Ver:\mathcal{A}_{1}\to\mathcal{B}$ is pulled
back from a similar isogeny over $S_{\mu},$ only over $Y_{et}^{\sigma}$
does it factor through $\mathcal{C}$ because $\mathcal{H}_{1}$ is
\emph{not }the pull-back of a group scheme on $S_{\mu}.$ Consider
now the diagram
\[
\begin{array}{ccc}
\overline{\pi}_{et}^{*}\mathcal{P}=\omega_{\mathcal{B}}(\Sigma) & \overset{\KS_{\mathcal{B}}}{\to} & \Omega_{Y_{et}^{\sigma}}\otimes\omega_{\mathcal{B}^{t}}^{\vee}(\Sigma)\\
\downarrow\varphi^{*} &  & \downarrow1\otimes(\varphi^{t})^{*\vee}\\
\omega_{\mathcal{C}}(\Sigma) & \overset{\KS_{\mathcal{C}}}{\to} & \Omega_{Y_{et}^{\sigma}}\otimes\omega_{\mathcal{C}^{t}}^{\vee}(\Sigma)
\end{array}
\]
resulting from the functoriality of the Kodaira-Spencer maps with
regard to the isogeny $\varphi.$ Here $\KS_{\mathcal{B}}$ is the
Kodaira-Spencer map for the family $\mathcal{B}\to Y_{et}^{\sigma}$
and likewise for $\mathcal{C}$. Note that as $\mathcal{B}=\overline{\pi}_{et}^{*}\mathcal{A}$,
$\KS_{\mathcal{B}}$ is the composition of the isomorphism
\[
\overline{\pi}_{et}^{*}(\KS):\overline{\pi}_{et}^{*}(\mathcal{P})\overset{\sim}{\to}\overline{\pi}_{et}^{*}(\Omega_{S_{\mu}})\otimes\overline{\pi}_{et}^{*}(\mathcal{L})^{\vee}
\]
(we identify $\mathcal{L=\omega_{\mathcal{A}}}(\overline{\Sigma})$
with $\omega_{\mathcal{A}^{t}}(\Sigma)$ via the polarization as usual)
and the map induced by
\[
\overline{\pi}_{et}^{*}:\overline{\pi}_{et}^{*}(\Omega_{S_{\mu}})\to\Omega_{Y_{et}^{\sigma}}.
\]

The kernel of the left vertical arrow $\varphi^{*}$ is precisely
$\overline{\pi}_{et}^{*}(\mathcal{P}_{0}).$ On the right hand side,
however, $1\otimes(\varphi^{t})^{*\vee}$ is injective. This stems
from the fact that the type of $\mathcal{C}$ (an étale quotient of
$\mathcal{A}_{1}$) is $(1,2)$ while the type of $\mathcal{B}$ is
$(2,1)$. Thus the type of $\mathcal{C}^{t}$ is $(2,1)$ and that
of $\mathcal{B}^{t}$ (1,2). The map $(\varphi^{t})^{*}$ being surjective
on the $\Sigma$-part of the cotangent spaces, its dual is injective.

We conclude that $\KS_{\mathcal{B}}(\overline{\pi}_{et}^{*}(\mathcal{P}_{0}))=0,$
hence
\[
\overline{\pi}_{et}^{*}(\KS(\mathcal{P}_{0}\otimes\mathcal{L}))\subset\ker(\overline{\pi}_{et}^{*}:\overline{\pi}_{et}^{*}\Omega_{S_{\mu}}\to\Omega_{Y_{et}^{\sigma}}).
\]
As both sides are line bundles which are direct summands of the locally
free rank 2 sheaf $\overline{\pi}_{et}^{*}\Omega_{S_{\mu}},$ the
inclusion is an equality between line sub-bundles, as desired. Their
annihilators in $TS$ are the ``special sub-bundle'' $TS^{+}$ and
the ``line-bundle of unramified directions'' $TS^{ur}$, hence these
two are also equal.
\end{proof}
In the next section we shall see that the theorem extends to the gss
locus. In fact, the same proof applies, once we extend the morphism
$\overline{\pi}_{et}$ and the factorization $\pi_{et}=\overline{\pi}_{et}\circ Fr_{Y/k}.$
See the proof of Theorem \ref{S_0(p)-gss} (iii).

\subsection{The gss strata}

Recall that the supersingular locus $S_{ss}\subset S$ is the union
of Fermat curves crossing transversally at the superspecial locus
$S_{ssp}.$ The complement of these crossing points was denoted $S_{gss}$
and is therefore a disjoint union of open Fermat curves. In this section
we study its pre-image under the morphism $S_{0}(p)\to S$ and show
that it is a $\mathbb{P}^{1}$-bundle, intersecting transversally
with the horizontal components of $S_{0}(p).$ Understanding the pre-image
of $S_{ssp}$ will be taken up in the next section.

\subsubsection{The $\mathbb{P}^{1}$-bundles}
\begin{thm}
\emph{\label{S_0(p)-gss}(i) }Let $Y_{gss}=\pi^{-1}(S_{gss})^{red}.$
Then $Y_{gss}$ has the structure of a $\mathbb{P}^{1}$-bundle over
the non-singular curve $S_{gss},$ with two distinguished non-intersecting
non-singular curves
\[
Z_{et}\text{ and }Z_{m}.
\]
A point $y=(\underline{A},H)\in S_{0}(p)(k)$ lies on $Z_{et}$ if
and only if $H\simeq\alpha_{p^{2},\Sigma}$ and on $Z_{m}$ if and
only if $H\simeq\alpha_{p^{2},\Sigma}^{*}.$ The fiber $\pi^{-1}(x)$
$(x\in S_{gss}(k))$ intersects each of the curves $Z_{et}$ or $Z_{m}$
at a unique point. At all other $k$-points $(\underline{A},H)$ of
$Y_{gss},$ the group $H\simeq\mathfrak{G}[p]_{\Sigma}$. 

\emph{(ii) }The closure $\overline{Y}_{m}$ of $Y_{m}$ intersects
$Y_{gss}$ transversally in $Z_{m}$. Let $Y_{m}^{\dagger}=Y_{m}\cup Z_{m}$,
a locally closed subscheme of $S_{0}(p)$, and $S_{\mu}^{\dagger}=S_{\mu}\cup S_{gss}.$
Then $Y_{m}^{\dagger}$ is a non-singular surface. The map $\pi_{m}\colon Y_{m}^{\dagger}\overset{\sim}{\to}S_{\mu}^{\dagger}$
is an isomorphism, and the section $\sigma_{m}:S_{\mu}\to Y_{m}$
extends to a section of $\pi_{m}$ over $S_{\mu}^{\dagger}.$

\emph{(iii)} The closure $\overline{Y}_{et}$ of $Y_{et}$ intersects
$Y_{gss}$ transversally in $Z_{et}$. Let $Y_{et}^{\dagger}=Y_{et}\cup Z_{et},$
a locally closed subscheme of $S_{0}(p).$ Then $Y_{et}^{\dagger}$
is a non-singular surface. The morphism $\overline{\pi}_{et}$ of
Theorem \ref{S_0(p)-1} extends to a morphism
\[
\overline{\pi}_{et}:Y_{et}^{\dagger}{}^{\sigma}\to S_{\mu}^{\dagger},
\]
which is finite flat totally ramified of degree $p.$ The factorization
$\pi_{et}=\overline{\pi}_{et}\circ Fr_{Y/k}$ extends to $Y_{et}^{\dagger}.$

Restricted to $Z_{et}$ the map $\pi_{et}$ is totally ramified of
degree $p$ and $\overline{\pi}_{Z}=\overline{\pi}_{et}|_{Z_{et}^{\sigma}}$
is an isomorphism from $Z_{et}^{\sigma}$ onto $S_{gss}.$

\emph{(iv) Setting }
\[
\rho_{et}(\underline{A})=(\underline{A}^{(p^{2})},\Fr(A^{(p)}[\Ver]))
\]
extends the map $\rho_{et}$ to a finite flat totally ramified map
of degree $p$ from $S_{\mu}^{\dagger}$ to $Y_{et}^{\dagger}.$ We
have
\[
\rho_{et}\circ\pi_{et}=Fr_{Y/k}^{2}:Y_{et}^{\dagger}\to Y_{et}^{\dagger\sigma^{2}}=Y_{et}^{\dagger},\,\,\,\,\rho_{et}\circ\overline{\pi}_{et}=Fr_{Y^{\sigma}/k}.
\]
\end{thm}
The proof of the theorem will be given in the next subsection. We
caution the reader that the scheme-theoretic pre-image of $S_{gss}$
under $\overline{\pi}_{et}$ is not reduced. It is rather a nilpotent
thickening of degree $p$ of the reduced curve $Z_{et}^{\sigma}$
in $Y_{et}^{\dagger\sigma}.$ Similarly the scheme-theoretic pre-image
$\pi^{-1}(S_{gss})$ is non-reduced along $Z_{et}$, and only there.

We also caution that the formula (\ref{eq:pi_bar_etale}) giving $\overline{\pi}_{et}$
on $Y_{et}^{\sigma}$ is no longer valid for its continuous extension
to $Z_{et}^{\sigma}$. The group functor $H_{1}+H_{1}^{\perp}[\Fr]$
is represented by a finite flat group scheme on each of $Y_{et}^{\sigma}$
and $Z_{et}^{\sigma}$ separately, but even though the ranks of these
group schemes are the same ($p^{3})$, they do not glue to give a
group scheme over the whole of $Y_{et}^{\dagger\sigma}$. Indeed,
at a closed point of $Y_{et}^{\sigma}$ this group is the kernel of
$\Ver$, but this does not hold at closed points of $Z_{et}^{\sigma}.$\footnote{If $H_{1}$ and $H_{2}$ are finite flat subgroup schemes of a finite
flat group scheme $G,$ then $H_{1}\cap H_{2}$ is a finite subgroup
scheme, but is not necessarily flat. If it is flat, then the sum $H_{1}+H_{2},$
being isomorphic as a group functor to $H_{1}\times H_{2}/(H_{1}\cap H_{2}),$
is again represented by a finite flat group scheme. In general, however,
the group-functor-quotient of a finite flat group scheme by a closed
(hence finite) non-flat subgroup scheme, need not be represented by
a group scheme at all, let alone by a finite flat group scheme. Thus
the sum of two subgroup schemes need not be a group scheme!}

\medskip{}

The following diagram summarizes what the extensions of the maps $\pi_{et},\overline{\pi}_{et},\rho_{et}$
to the gss strata look like.

\bigskip{}

\begin{center}%
\begin{tabular}{ccccc}
$Z_{et}$ & $\overset{Fr_{Z/k}}{\longrightarrow}$ & $Z_{et}^{\sigma}$ & $\overset{Fr_{Z^{\sigma}/k}}{\longrightarrow}$ & $Z_{et}^{\sigma^{2}}=Z_{et}$\tabularnewline
 & {\footnotesize{}$\pi_{et}\searrow p$} & $\simeq\downarrow\overline{\pi}_{et}$ & {\footnotesize{}$p\nearrow\rho_{et}$} & \tabularnewline
 &  & $S_{gss}$ &  & \tabularnewline
\end{tabular}\end{center}

\bigskip{}

\begin{cor}
\label{Theta}(i) The maps $\overline{\pi}_{et}$ and $\sigma_{m}$
induce an isomorphism
\[
\sigma_{m}\circ\overline{\pi}_{et}:Z_{et}^{\sigma}\overset{\sim}{\to}Z_{m}.
\]
(ii) Setting $\theta=\rho_{et}\circ\pi_{m}:Y_{m}^{\dagger}\to Y_{et}^{\dagger}$
gives a commutative diagram of totally ramified finite flat morphisms
between surfaces, and similarly between embedded curves (the diagonal
arrows are embeddings):
\[
\begin{array}{ccccccc}
Z_{m} &  & \overset{\theta}{\longrightarrow} &  & Z_{et}\\
| & \searrow &  &  & \vdots & \searrow\\
\simeq &  & Y_{m}^{\dagger} &  & \overset{\theta}{\longrightarrow} &  & Y_{et}^{\dagger}\\
\downarrow &  & | &  & \downarrow &  & |\\
S_{gss} & \cdots & \simeq & \overset{Fr^{2}}{\dashrightarrow} & S_{gss} &  & \pi_{et}\\
 & \searrow\,\,\, & \downarrow &  &  & \searrow & \downarrow\\
 &  & S_{\mu}^{\dagger} &  & \overset{Fr^{2}}{\longrightarrow} &  & S_{\mu}^{\dagger}
\end{array}
\]
The map $\theta$ is of degree $p,$ and so is $\theta|_{Z_{m}}.$
In particular, the latter factors through the Frobenius of the curve
$Z_{m}$ and yields an isomorphism $Z_{m}^{\sigma}\overset{\sim}{\to}Z_{et}.$ 

If $Z'_{m}$ and $Z'_{et}$ are two $\kappa$-components of $Z_{m}$
and $Z_{et}$ (i.e. defined and irreducible over $\kappa$) which
map to the same $\kappa$-component $S'_{gss}$ of $S_{gss}$ then
$\theta(Z_{m}')=Z_{et}'$.
\end{cor}
\begin{proof}
The commutativity is easily checked in terms of the moduli problem.
The degrees are calculated from the fact that $\pi_{m}$ is an isomorphism,
$\pi_{et}$ has degree $p^{3}$ on $Y_{et}^{\dagger}$ and degree
$p$ on $Z_{et}$, while $Fr_{S/K}^{2}$ has degree $p^{4}$ on $S_{\mu}^{\dagger}$
and degree $p^{2}$ on $S_{gss}$. To summarize, in the front square
we have $p^{3}\times p=p^{4}\times1,$ and in the back square we have
$p\times p=p^{2}\times1.$ The assertion about $\kappa$-components
follows from the fact that $Fr_{S/k}^{2}$ preserves these components.
\end{proof}
\begin{rem*}
We believe that if $N=1$ (working with stacks) the geometrically
irreducible components of $S_{gss}$ are already defined over $\kappa$,
hence $\theta$ exchanges the irreducible components of $Z_{m}$ and
$Z_{et}$ within the same irreducible component of $Y_{gss}$. This
is clearly not the case when $N>1.$ Compare with supersingular points
on the modular curve $X(N).$
\end{rem*}

\subsubsection{Proof of Theorem \ref{S_0(p)-gss}}

We first quote \cite{Bu-We}, Proposition 3.6. In the notation used
there, the Dieudonné module of $A[p],$ for $A$ supersingular but
not superspecial, is the ``Dieudonné space'' $\overline{\underline{B}}(3).$
Our Dieudonné module $M$ differs from the one appearing in \cite{Bu-We},
(3.2)(2) by a ``Frobenius twist''. This is because we use covariant
Dieudonné theory, while \cite{Bu-We} employs Cartier theory. See
\cite{C-C-O}, Appendix B.3.10, where the first (used here) is denoted
$M_{*},$ and the second (used in \cite{Bu-We}) is denoted $E_{*}.$
\begin{prop}
\label{Braid-1}Let $\underline{A}\in S_{gss}(k),$ and let $M=M(A[p])$
be the covariant Dieudonné module of $A[p].$ Then $M$ has a basis
over $k$ denoted $\{e_{1},e_{2},e_{3},f_{1},f_{2},f_{3}\}$ such
that

(i) $\mathcal{O}_{E}$ acts on the $e_{i}$ via $\Sigma$ and on the
$f_{i}$ via $\overline{\Sigma}$.

(ii) The antisymmetric pairing induced by the principal polarization
$\phi$ is given by $\left\langle e_{i},f_{j}\right\rangle =(-1)^{j}\delta_{ij},\,\,\left\langle e_{i},e_{j}\right\rangle =\left\langle f_{i},f_{j}\right\rangle =0.$

(iii) $F$ and $V$ are given by the following table:

\medskip{}

\begin{center}%
\begin{tabular}{|c|c|c|c|c|c|c|}
\hline 
 & $e_{1}$ & $e_{2}$ & $e_{3}$ & $f_{1}$ & $f_{2}$ & $f_{3}$\tabularnewline
\hline 
\hline 
$F$ & $-f_{3}$ & $0$ & $0$ & $0$ & $e_{1}$ & $e_{2}$\tabularnewline
$V$ & $0$ & $0$ & $f_{1}$ & $e_{2}$ & $e_{3}$ & $0$\tabularnewline
\hline 
\end{tabular}\end{center} 

\medskip{}

By this we mean that $Fe_{1}^{(p)}=-f_{3},\,\,Ve_{3}=f_{1}^{(p)}$,
etc. In particular, $\Lie(A)=M[V]=\left\langle e_{1},e_{2},f_{3}\right\rangle $.
\end{prop}
Let $\underline{A}\in S_{\mu}^{\dagger}(R),$ where $R$ is an arbitrary
$k$-algebra.
\begin{lem}
\label{alpha_p}The $R$-subgroup scheme $\alpha_{p}(A^{(p)})=A^{(p)}[\Fr]\cap A^{(p)}[\Ver]$
is finite flat of rank $p$, and $\mathcal{O}_{E}$-stable.
\end{lem}
\begin{proof}
We have already encountered the lemma when $A$ was $\mu$-ordinary.
The extension to the gss stratum works the same. It is enough to prove
the lemma for the universal abelian scheme $\mathcal{A}$ over $S_{\mu}^{\dagger}.$
In this case $\alpha_{p}(\mathcal{A}^{(p)})$ is clearly finite and
$\mathcal{O}_{E}$-stable, and its fibers all have the same rank $p$,
as follows from Proposition \ref{Braid-1}. Let us make this point
clear, because the proposition only deals with fibers over closed
points. Let $\xi$ be any point of $S_{\mu}^{\dagger}$ (not necessarily
closed), and $\overline{\{\xi\}}$ its closure (a point, a curve,
or an irreducible surface). By the open-ness of the flat locus there
is a non-empty connected open subset $\xi\in U\subset\overline{\{\xi\}}$
such that $\alpha_{p}(\mathcal{A}^{(p)})|_{U}$ is finite and flat
over $U,$ hence all its fibers, at all the geometric points of $U,$
have the same rank. But $U(k)$ is Zariski dense in $U,$ and at a
$k$-point the proposition tells us that the rank is $p.$ Hence the
rank is $p$ at $\xi$ as well. Since $S_{\mu}^{\dagger}$ is \emph{reduced},
by \cite{Mu}, Corollary on p. 432, $\alpha_{p}(\mathcal{A}^{(p)})$
is also flat.
\end{proof}
\begin{prop}
The finite flat group scheme $A[p]_{/R}$ has a canonical filtration
\[
Fil^{3}A[p]=0\subset Fil^{2}A[p]\subset Fil^{1}A[p]\subset Fil^{0}A[p]=A[p]
\]
by finite flat group schemes, which agrees with the canonical filtration
over $S_{\mu}.$ The graded pieces are $\mathcal{O}_{E}$-stable,
rank $p^{2}$ and Raynaud. Furthermore, $Fil^{1}A[p]=Fil^{2}A[p]^{\perp}$
(with respect to the Weil pairing). Over $S_{gss}$ every geometric
fiber of $gr^{2}A[p]$ is of type $\alpha_{p^{2},\Sigma}^{*},$ $gr^{1}A[p]$
is of type $\kappa\otimes\alpha_{p},$ and $gr^{0}A[p]$ is of type
$\alpha_{p^{2},\Sigma}.$ Let $R=k$ and assume that $\underline{A}\in S_{gss}(k).$
Then, with the notation of Proposition \ref{Braid-1},
\[
Fil^{2}M=\left\langle e_{2},f_{3}\right\rangle ,\,\,\,Fil^{1}M=\left\langle e_{1},e_{2},f_{1},f_{3}\right\rangle .
\]
\end{prop}
We remark that unlike $\mu$-ordinary abelian varieties, the above
filtration \emph{does not split}, \emph{even if} $R=k.$ As we shall
see, $A[p]$ does not admit a subgroup scheme of type $\kappa\otimes\alpha_{p}$
at all, and while it does admit a unique subgroup scheme of type $\alpha_{p^{2},\Sigma}$,
this subgroup scheme is contained in $Fil^{1}A[p],$ so does not lift
$gr^{0}A[p].$
\begin{proof}
Define
\[
Fil^{2}A[p]=\Ver(A^{(p)}[\Fr])\simeq A^{(p)}[\Fr]/A^{(p)}[\Fr]\cap A^{(p)}[\Ver].
\]
This image exists because it is a quotient by a finite flat subgroup
scheme. It is a closed subgroup scheme of $A[p].$ Since $A^{(p)}[\Fr]$
is finite flat of rank $p^{3},$ the Lemma implies that $Fil^{2}A[p]$
is finite flat of rank $p^{2}.$ It is furthermore isotropic for the
Weil pairing $e_{p\phi}$ on $A[p]$ associated with the principal
polarization $\phi$. By Cartier duality
\[
Fil^{1}A[p]=Fil^{2}A[p]^{\perp}
\]
is finite flat of rank $p^{4}.$ These group schemes are clearly $\mathcal{O}_{E}$-stable.

The remaining assertions concern the geometric fibers of $A[p],$
so we assume that $R=k.$ Over the $\mu$-ordinary locus this is the
same filtration that we encountered before. Assume that we are over
$S_{gss},$ and use Proposition \ref{Braid-1}. Let $M=M(A[p]).$
Since $F$ is induced by $\Ver$ and $V$ is induced by $\Fr,$ we
have to compute $F(M^{(p)}[V]).$ This turns out to be $\left\langle e_{2},f_{3}\right\rangle .$
A simple check of the table in §\ref{Raynaud classification} reveals
that $\text{\ensuremath{gr^{2}=}}Fil^{2}A[p]$ is of type $\alpha_{p^{2},\Sigma}^{*}.$
Similar computations apply to $gr^{1}$ and $gr^{0}$.
\end{proof}
We can now complete the proof of Part (i) of Theorem \ref{S_0(p)-gss}.
From the analysis of the local models it follows that $Y_{gss}$ is
a non-singular surface, mapping under the map $\pi$ to the non-singular
curve $S_{gss}.$ This is clear at points where $H\simeq\mathfrak{G}[p].$
At a point $y\in Y_{gss}$ where $H\simeq\alpha_{p^{2},\Sigma}$ or
$H\simeq\alpha_{p^{2},\Sigma}^{*}$ the formal neighborhood of $y$
in $S_{0}(p)$ has two non-singular analytic branches which intersect
transversally. Since there are \emph{at least} two irreducible components
of $S_{0}(p)$ passing through $y,$ the vertical component $Y_{gss}$
and (at least) one horizontal component, we conclude that there are
\emph{precisely} two such components, and that they are non-singular
at $y.$ In particular, $Y_{gss}$ is non-singular at $y$ too.

By the Noether-Enriques Theorem (\cite{Bea} Theorem III.4 and Proposition
III.7) it is enough to prove that for any $x\in S_{gss}(k),$ the
scheme-theoretic fiber
\[
Y_{x}\subset Y_{gss}
\]
of the map $\pi:Y_{gss}\to S_{gss}$ is isomorphic to $\mathbb{P}^{1}.$
We rely on the computation of local models at points $y\in Y_{x}$
in \cite{Bel} III.4.3.8. These show that for any $y\in Y_{x}$ the
map
\[
\pi^{*}:\Omega_{S_{gss},x}\to\Omega_{Y_{gss},y}
\]
is injective, and $\pi:Y_{gss}\to S_{gss}$ is smooth at $y.$ We
do not reproduce these computations here, but remark that the most
problematic points turn out to be the $y$ that lie on $Z_{et}$ (where
$H\simeq\alpha_{p^{2},\Sigma}$). At such points the claim follows
from §\ref{example}, as the analytic branch of $S_{0}(p)$ at $y$
determined by $Y_{gss}$ is the one denoted there $\mathfrak{W},$
while $S_{gss}\subset S$ is given infinitesimally by the equation
$r=0.$ $Y_{x}$ is therefore \emph{a reduced} \emph{non-singular}
curve.

Let $M$ be the covariant Dieudonné module of $A[p],$ where $A=\mathcal{A}_{x},$
see Proposition \ref{Braid-1}. The fiber $Y_{x}$ represents the
\emph{relative} moduli problem, sending a $k$-algebra $R$ to the
set of finite flat rank $p^{2}$ isotropic Raynaud $\mathcal{O}_{E}$-subgroup
schemes $H\subset A_{R}[p].$ Note that since $A_{R}$ is a constant
abelian scheme over $Spec(R)$ both $\Fr$ and $\Ver$ are defined
on it, base-changing from $k$ to $R$ the corresponding isogenies
of $A.$ Let
\[
\alpha_{p}(A_{R})=A_{R}[\Fr]\cap A_{R}[\Ver].
\]
This is a constant (finite flat) subgroup scheme of rank $p,$ and
if $R=k$, its Dieudonné submodule is $\left\langle e_{2}\right\rangle .$
Let
\[
\beta_{p}(A_{R})=A_{R}[\Fr^{2}]\cap A_{R}[\Ver^{2}]\cap A_{R}[p],
\]
another constant (finite flat) subgroup scheme, of rank $p^{3}$.
If $R=k,$ its Dieudonné submodule is $\left\langle e_{2},f_{1},f_{3}\right\rangle .$
We claim that
\[
\alpha_{p}(A_{R})\subset H\subset\beta_{p}(A_{R}),
\]
hence classifying $H_{/R}$ is the same as classifying finite flat
rank $p$ subgroups of $\beta_{p}(A_{R})/\alpha_{p}(A_{R}).$ Since
$Y_{x}$ is a reduced non-singular curve, it is enough to check these
inclusions when $R$ is reduced and of finite type over $k.$ Since
the closed points of $Spec(R)$ are then dense, we may assume $R=k.$
But over $k$, $\Ver$ and $\Fr$ are nilpotent on $H,$ which is
of rank $p^{2},$ so both $\Ver^{2}$ and $\Fr^{2}$ must kill it.
On the other hand, $H$ must contain an $\alpha_{p}$-subgroup, because
it is local with a local Cartier dual. 

Now $\beta_{p}(A_{R})/\alpha_{p}(A_{R})$ is nothing but $\alpha_{p}^{2}$
(of type $(\overline{\Sigma},\overline{\Sigma})$) and it is well-known
that the moduli problem of classifying its rank-$p$ subgroups is
represented by $\mathbb{P}_{/k}^{1}.$ One checks that the isotropy
and Raynaud conditions are automatically satisfied for such an $H.$ 

Let $R=k.$ The subgroup scheme $H$ is completely determined by its
Dieudonné submodule
\[
N_{\lambda}=\left\langle e_{2},\lambda_{1}f_{1}+\lambda_{3}f_{3}\right\rangle 
\]
where $\lambda=(\lambda_{1}:\lambda_{3})\in\mathbb{P}^{1}(k).$ Here
$N_{0}=N_{(0:1)}=M(H)$ if $H=Fil^{2}(A[p])\simeq\alpha_{p^{2}}^{*}.$
Similarly, $N_{\infty}=N_{(1:0)}=M(H)$ where $H\simeq\alpha_{p^{2}}$
because $N_{(1:0)}$ is killed by $F$ and $V^{2}$ but not by $V.$
For all other values of $\lambda\neq0,\infty,$ $N_{\lambda}=M(H)$
where $H$ is of type $\mathfrak{G}[p]_{\Sigma},$ because $N_{\lambda}$
is killed by $F^{2}$ and $V^{2}$ but the kernels of $F$ or $V$
are only $1$-dimensional.

\medskip{}

Part (ii): Let us show that the totality of points $(\underline{A},H)\in Y_{gss}(k)$
where $H\simeq\alpha_{p^{2}}^{*}$, makes up a \emph{curve $Z_{m},$
}that $\pi$ induces an isomorphism of this curve onto $S_{gss}$,
and that the closure of $Y_{m}$ intersects $Y_{gss}$ transversally
in this curve. For this purpose, consider the section
\[
\sigma_{m}:S_{\mu}^{\dagger}\to S_{0}(p)
\]
mapping an $R$-valued point $\underline{A}$ to $(\underline{A},H)$,
where $H=Fil^{2}A[p]=\Ver(A^{(p)}[\Fr]).$ The image of the section
is a surface isomorphic to the base, intersecting $Y_{\mu}$ in its
connected component $Y_{m}$ and $Y_{gss}$ in the curve $Z_{m}.$
Finally, the transversality of the intersection of the closure of
$Y_{m}$ and $Y_{gss}$ follows from the calculation of the completed
local ring of $S_{0}(p)$ at a point $y\in Z_{m},$ see §\ref{local rings}.

\medskip{}

Part (iii): We turn our attention to the points $(\underline{A},H)\in Y_{gss}(k)$
where $H\simeq\alpha_{p^{2}}.$ The condition $\Ver(H^{(p)})=0$ is
a \emph{closed} condition on the moduli problem $S_{0}(p)$. It is
satisfied throughout $Y_{et}$ and on $Y_{gss}$ it holds precisely
at the given points where $H\simeq\alpha_{p^{2}}.$ We claim that
this set forms a curve $Z_{et}$, which is the intersection of the
closure of $Y_{et}$ and $Y_{gss}.$ Indeed, $\pi$ being proper,
the closure of $Y_{et}$ must meet \emph{every} fiber $Y_{x}$ for
$x\in S_{gss}(k)$, and such a fiber has a unique point where $H\simeq\alpha_{p^{2}}.$
That the intersection is transversal follows as before from §\ref{local rings}.

Write $Y_{et}^{\dagger}=Y_{et}\cup Z_{et}.$ The computations in §\ref{local rings}
show that $Y_{et}^{\dagger}$ is non-singular. So is $Y_{et}^{\dagger\sigma}$.

We claim that since $\pi:Y_{et}^{\dagger}\to S$ factors through $Fr_{Y/k}:Y_{et}^{\dagger}\to Y_{et}^{\dagger\sigma}$
over the dense open set $Y_{et}$, it factors through $Fr_{Y/k}$
everywhere. Indeed, consider the local ring $\mathcal{O}_{S,x}$ at
$x=\pi(y)\in S$, where $y\in Y_{et}^{\dagger}$ is a closed point.
Let $y^{(p)}=Fr_{Y/k}(y)\in Y_{et}^{\dagger\sigma}.$ For the function
fields we have
\[
k(S)\subset k(Y_{et}^{\dagger\sigma})=k(Y_{et}^{\dagger})^{p}\subset k(Y_{et}^{\dagger}).
\]
Thus $\mathcal{O}_{S,x}\subset k(Y_{et}^{\dagger})^{p}\cap\mathcal{O}_{Y_{et}^{\dagger},y}.$
But the ring on the right is just $\mathcal{O}_{Y_{et}^{\dagger\sigma},y^{(p)}}$,
because $y$ is the \emph{unique }point above $y^{(p)}$ in $Y_{et}^{\dagger}$
and $\mathcal{O}_{Y_{et}^{\dagger\sigma},y^{(p)}}$ is normal. For
every affine subset $U=Spec(R)\subset Y_{et}^{\dagger}$ the ring
$R$ is the intersection of all the $\mathcal{O}_{Y_{et}^{\dagger},y}$
for closed points $y\in U$$,$ and similarly for $Fr_{Y/k}(U)\subset Y_{et}^{\dagger\sigma}.$
This proves the claim.

Thus $\overline{\pi}_{et}$ extends to a morphism from $Y_{et}^{\dagger\sigma}$
to $S_{\mu}^{\dagger}$. It is a finite morphism, because $\pi_{et}:Y_{et}^{\dagger}\to S_{\mu}^{\dagger}$
is finite. Both source and target are non-singular surfaces, so by
\cite{Eis} 18.17 it is also flat, totally ramified of degree $p$.
It therefore defines a line sub-bundle $TS^{ur}$ of unramified directions
in the tangent bundle there, as in Lemma \ref{Liedtke}, now over
all of $S_{\mu}^{\dagger}.$ Recall that the special sub-bundle $TS^{+}$
was defined on the whole of $S_{\mu}^{\dagger}$ as well. The two
line sub-bundles $TS^{+}$ and $TS^{ur}$ coincide over $S_{\mu}$
(Theorem \ref{unramified direction}), hence also over $S_{gss}$,
by continuity. 

As $TS^{+}$ is tangent to $S_{gss}$ along the general supersingular
stratum, we get, from the discussion following Lemma \ref{Liedtke},
that $\overline{\pi}_{Z}\colon Z_{et}^{\sigma}\to S_{gss}$ is unramified.
As it is also totally ramified (bijective on $k$-points), it is an
isomorphism.

In retrospect, we can look at the factorization $\pi_{et}=\overline{\pi}_{et}\circ Fr_{Y/k}$
also from the moduli point of view as follows. Consider the abelian
scheme $\mathcal{B}=\overline{\pi}_{et}^{*}\mathcal{A}$ which is
the pull-back of the universal abelian scheme over $S_{\mu}^{\dagger}$
to $Y_{et}^{\dagger\sigma}.$ Consider also the universal abelian
scheme $\mathcal{A}_{1}$ over $Y_{et}^{\dagger\sigma}$. Over the
dense open subset $Y_{et}^{\sigma}$ $\mathcal{A}_{1}\simeq\mathcal{B}^{(p)},$
as was shown in the proof of Theorem \ref{S_0(p)-1}. It follows that
this relation persists over $Z_{et}^{\sigma}$, and \emph{a-fortiori}
we may define $\overline{\pi}_{et}$ by sending $(\underline{A}_{1},H_{1})\in Y_{et}^{\dagger\sigma}(R)$
to $\left\langle p\right\rangle ^{-1}\Ver(\underline{B}^{(p)})\in S_{\mu}^{\dagger}(R).$

\medskip{}

Part (iv): By Lemma \ref{alpha_p}, and the arguments used before,
$\Fr(A^{(p)}[\Ver])$ is a finite flat rank-$p^{2}$ isotropic Raynaud
subgroup scheme of $A^{(p^{2})}[p],$ for any $A\in S_{\mu}^{\dagger}(R),$
for any $k$-algebra $R.$ Since $\rho_{et}$ is now defined functorially
in terms of the moduli problems, it is a well defined morphism. The
argument is identical to the one used for the proof of Part (v) of
Theorem \ref{S_0(p)-1}. 

Since the equality $\rho_{et}\circ\pi_{et}=Fr_{Y/k}^{2}$ has already
been established on $Y_{et}=Y_{et}(k)$, it extends by continuity
to $Y_{et}^{\dagger}.$ The relation $\rho_{et}\circ\overline{\pi}_{et}=Fr_{Y^{\sigma}/k}$
follows from $\rho_{et}\circ\pi_{et}=Fr_{Y/k}^{2}$ since $\pi_{et}=\overline{\pi}_{et}\circ Fr_{Y/k}$
and $Fr_{Y/k}$ is faithfully flat. The remaining assertions on $\rho_{et}$
also follow from this relation. This concludes the proof of Theorem
\ref{S_0(p)-gss}.

\subsubsection{A closer look at Example \ref{example}}

It is instructive to look again at the diagram
\[
\mathcal{\widehat{O}}_{S,x}\hookrightarrow\mathcal{\widehat{O}}_{S_{0}(p),y}
\]
at a point $y\in Z_{et}(k).$ We have found the local models $\mathcal{\widehat{O}}_{S,x}\simeq k[[r,s]]$
and $\mathcal{\widehat{O}}_{S_{0}(p),y}\simeq k[[a,b,c]]/(bc)$. The
map between the local models is
\[
r\mapsto0,\,\,\,s\mapsto b.
\]
This is far from the correct map between the completed local rings,
which should be injective. Let $\mathcal{\widehat{O}}_{\mathfrak{W}}$
and $\mathcal{\widehat{O}}_{\mathfrak{Z}}$ be the quotients of $\mathcal{\widehat{O}}_{\mathfrak{Y}}=\mathcal{\widehat{O}}_{S_{0}(p),y}$
which were introduced in §\ref{example}. The first is obtained by
modding out $(c)$, and is the analytic branch determined by the inclusion
$Y_{gss}\subset S_{0}(p)$. The second is obtained by modding out
$(b)$, and is the analytic branch determined by the inclusion $Y_{et}^{\dagger}\subset S_{0}(p)$.
\begin{claim}
The diagram $\mathcal{\widehat{O}}_{S,x}\to\mathcal{\widehat{O}}_{\mathfrak{W}}$
is isomorphic to the diagram
\[
k[[r,s]]\to k[[a,b]],\,\,\,s\mapsto b+a^{p},\,\,r\mapsto0,
\]
and the diagram $\mathcal{\widehat{O}}_{S,x}\to\mathcal{\widehat{O}}_{\mathfrak{Z}}$
is isomorphic to the diagram
\[
k[[r,s]]\hookrightarrow k[[a,c]],\,\,\,\,r\mapsto c^{p^{2}},\,\,s\mapsto a^{p}.
\]
\end{claim}
This is more than could be deduced from the local models alone. 
\begin{proof}
After a change of variable we may assume that $r=0$ is the equation
of $S_{gss}$ in a formal neighborhood of $x$ on $S.$ Therefore
$r$ maps to $0$ in $\mathcal{\widehat{O}}_{\mathfrak{W}}$. The
local parameter $s$ projects (modulo $(r))$ to a local parameter
of the curve $S_{gss}.$ We already know that it should map to $b$
modulo $p$th powers. Since $b=c=0$ is the formal equation of the
curve $Z_{et}$ (the intersection of the two analytic branches) on
$\mathfrak{W},$ and since the map $Z_{et}\to S_{gss}$ is purely
inseparable of degree $p,$ we see that we may choose $a$ so that
$s\mod(b)=a^{p}.$ A last change of variables allows us to assume
that actually $s=b+a^{p}$.

The second diagram is treated similarly. Here the key point is to
recall that the map $\pi_{et}$ from $\mathfrak{Z}$ to $Spf(\widehat{\mathcal{O}}_{S,x})$
factors through $\Fr$. The resulting map $\overline{\pi}_{et}$ on
$\mathfrak{Z}^{\sigma}$ was shown to be of degree $p$ and unramified
in the direction of $S_{gss}$.
\end{proof}
Both diagrams are compatible with $\mathcal{\widehat{O}}_{S,x}\to\mathcal{\widehat{O}}_{\mathfrak{Y}}=\mathcal{\widehat{O}}_{S_{0}(p),y}$
being given by
\[
r\mapsto c^{p^{2}},\,\,\,s\mapsto b+a^{p}.
\]

\subsection{The ssp strata}

\subsubsection{The superspecial combs}

We now turn our attention to the superspecial strata of $S_{0}(p).$
Let $x\in S_{ssp}(k)$ and $Y_{x}=\pi^{-1}(x).$ We shall contend
ourselves with the determination of the \emph{reduced }scheme $Y_{x}^{red},$
of finite type over $k.$ The scheme theoretic pre-image of $x$ will
not be reduced along the component denoted below $F_{x}$, see the
discussion following the theorem.
\begin{thm}
\emph{\label{S_0(p)-ssp}(i) }$Y_{x}^{red}$ is the union of $p+2$
projective lines, arranged as follows. One irreducible component,
which we call $F_{x}$, intersects the remaining $p+1$ projective
lines transversally, each at a different point $\zeta\in F_{x}.$
With a natural choice of a coordinate on $F_{x},$ this $\zeta$ can
be taken to be\footnote{This is a non-trivial statement, as it has consequences for the cross
ratio of the intersection points, which is independent of the chosen
coordinate on the basis of the comb.} a $p+1$ root of $-1.$ These $p+1$ projective lines, which we label
as $G_{x}[\zeta]$, are disjoint from each other. 

A point $(\underline{A},H)\in Y_{x}(k)$ lies on $F_{x}$ if and only
if $H\simeq\kappa\otimes\alpha_{p}$. If this is the case, the invariant
$\gamma(\underline{A},H)=1$ if $(\underline{A},H)$ lies on a non-singular
point of $Y_{x}$, and is equal to $2$ if it lies at the intersection
of $F_{x}$ and some $G_{x}[\zeta]$ (i.e. if it is the point $\zeta$).
Finally, if $(\underline{A},H)$ lies on $G_{x}[\zeta]$ but not on
$F_{x},$ the group $H\simeq\mathfrak{G}[p]_{\Sigma}.$

\emph{(ii)} The \emph{closure $\overline{Y}_{\mu}$} of $Y_{\mu}=Y_{m}\cup Y_{et}$
in $S_{0}(p)$ intersects $Y_{x}^{red}$ in $F_{x}.$ 

\emph{(iii) }Let $W$ be the \emph{closure} of an irreducible component
of $Y_{gss}.$ Then $W$ is a $\mathbb{P}^{1}$-bundle over an irreducible
component $\mathscr{C}=\pi(W)$ of $S_{ss}$. If $x\in S_{ssp}$ and
$W_{x}=W\cap Y{}_{x}^{red}$ then $W_{x}$ is one of the $G_{x}[\zeta].$
Precisely one such $W$ passes through $G_{x}[\zeta]$ for a given
$x$ and $\zeta$. Thus the closures of the irreducible components
of $Y_{gss}$ do not intersect each other. 

\emph{(iv)} The closures of the curves $W\cap Z_{et}$ and $W\cap Z_{m}$
intersect $G_{x}[\zeta]$ at the point $\zeta=G_{x}[\zeta]\cap F_{x}.$
\end{thm}
\begin{figure}

\caption{\label{figure ssp}The fiber of $S_{0}(p)$ above a superspecial point}

$$ \xymatrix{ & ^{\mathbb{P}^1}\ar@{-}[ddd]^(.3){\mathfrak{G}[p]}&  & ^{\mathbb{P}^1}\ar@{-}[ddd]^(.3){\mathfrak{G}[p]}&^{\mathbb{P}^1}\ar@{-}[ddd]^(.3){\mathfrak{G}[p]} & \\ & &\dots & & & \\ \ar@{-}[rrrrr]^(.35){\kappa\otimes\alpha_p}&\bullet & & \bullet&\bullet & ^{\mathbb{P}^1}\\ & & ^{\zeta}\ar@/^/@{..>}[ru] \ar@/^/@{..>}[rru] \ar@/_/@{..>}[lu] & & &  \\ } $$

\end{figure}
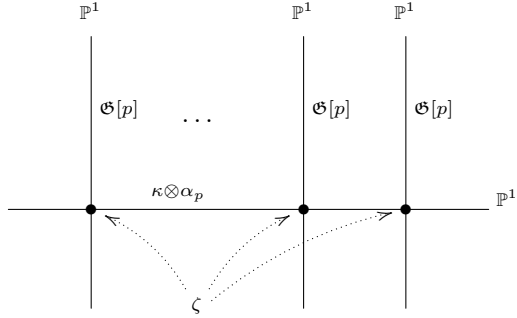

See Figures \ref{Figure 1}, \ref{figure ssp}. We refer to the irreducible
components $W$ of the closure of $Y_{gss}$ as the \emph{supersingular
(ss) screens}. We refer to the $Y_{x}$ for $x$ superspecial as the
\emph{superspecial (ssp) combs. }The component $F_{x},$ which we
draw horizontally, is called the \emph{base} of the comb, and the
vertical components $G_{x}[\zeta]$ are called its \emph{teeth}. The
points $\zeta$ are called the \emph{roots} of the teeth.
\begin{proof}
(i) Let $A=\mathcal{A}_{x}.$ We first analyze what happens on the
level of Dieudonné modules. Fix a model of $\mathfrak{G}_{\Sigma}$
over $k,$ let $\mathfrak{G}_{\overline{\Sigma}}=\mathfrak{G}_{\Sigma}^{\sigma}$
and fix the polarization
\[
\lambda:\mathfrak{G}[p]_{\Sigma}\overset{\sim}{\to}\mathfrak{G}[p]_{\Sigma}^{D}=\mathfrak{G}[p]_{\overline{\Sigma}}
\]
so that the resulting pairing on $\mathfrak{G}[p]_{\Sigma}$, $(x,y)\mapsto\left\langle x,\lambda(y)\right\rangle $
is alternating. The group scheme $A[p]$ is isomorphic to
\[
\mathfrak{G}[p]_{\Sigma}^{2}\times\mathfrak{G}[p]_{\overline{\Sigma}},
\]
so that the polarization induced on it by $\phi_{x}$ is the product
$\lambda^{2}\times\lambda^{\sigma}$ of the polarizations of the three
factors. Consequently \cite{Bu-We}, the polarized Dieudonné module
$\text{\ensuremath{M=}}M(A[p])$ is given by $M=\left\langle e_{1},e_{2},e_{3},f_{1},f_{2},f_{3}\right\rangle _{k}$,
where the endomorphisms act on the $e_{i}$ via $\Sigma$ and on the
$f_{i}$ via $\overline{\Sigma},$ where $\left\langle e_{i},f_{j}\right\rangle =\delta_{ij},$
$\left\langle e_{i},e_{j}\right\rangle =\left\langle f_{i},f_{j}\right\rangle =0$
and where the action of $F$ and $V$ is given by the table

\medskip{}
\begin{center}

\begin{tabular}{|c|c|c|c|c|c|c|}
\hline 
 & $e_{1}$ & $e_{2}$ & $e_{3}$ & $f_{1}$ & $f_{2}$ & $f_{3}$\tabularnewline
\hline 
\hline 
$F$ & $0$ & $0$ & $-f_{3}$ & $e_{1}$ & $e_{2}$ & $0$\tabularnewline
$V$ & $0$ & $0$ & $f_{3}$ & $-e_{1}$ & $-e_{2}$ & $0$\tabularnewline
\hline 
\end{tabular}\end{center} 

\medskip{}
By this we mean $Fe_{3}^{(p)}=-f_{3},\,\,\,Ve_{3}=f_{3}^{(p)}$ etc. 

Let $H\subset A[p]$ be as in $(S_{0}(p))$. Since $M(H)$ is balanced
we may write
\[
M(H)=\left\langle \alpha_{1}e_{1}+\alpha_{2}e_{2}+\alpha_{3}e_{3},\,\,\beta_{1}f_{1}+\beta_{2}f_{2}+\beta_{3}f_{3}\right\rangle .
\]
The conditions that have to be satisfied are $V(M(H))\subset M(H)^{(p)},\,\,\,F(M(H)^{(p)})\subset M(H),$
and the isotropy condition
\[
\alpha_{1}\beta_{1}+\alpha_{2}\beta_{2}+\alpha_{3}\beta_{3}=0.
\]
Observe that $M(H)$ contains $\beta_{1}^{p}e_{1}+\beta_{2}^{p}e_{2}.$
If $\alpha_{3}\neq0$ this forces $\beta_{1}=\beta_{2}=0,$ and then
the isotropy condition gives also $\beta_{3}=0,$ an absurd. Therefore
$\alpha_{3}=0.$ We distinguish two cases.

\emph{Case I} (the base of the comb): $\beta_{1}=\beta_{2}=0.$ This
case is characterized by the fact that $M(H)$ is killed by both $V$
and $F,$ so that $H\simeq\kappa\otimes\alpha_{p}.$ We may take $\beta_{3}=1$
and $H$ is classified by
\[
\zeta=(\alpha_{1}:\alpha_{2})\in\mathbb{P}^{1}(k).
\]
Consider in this case the group $H^{\perp}/H.$ Its Dieudonné module
is given by
\[
M(H^{\perp}/H)=\left\langle e_{1},e_{2},-\alpha_{2}f_{1}+\alpha_{1}f_{2},f_{3}\right\rangle \text{\,\,mod\,\,}\left\langle \alpha_{1}e_{1}+\alpha_{2}e_{2},f_{3}\right\rangle .
\]
An easy check shows that $H^{\perp}/H$ is of type $\mathfrak{G}[p]_{\Sigma}$,
unless $\zeta^{p+1}=-1,$ where it is of type $\kappa\otimes\alpha_{p}.$
The invariant $\gamma(\underline{A},H)=\dim_{k}\Lie(H^{\perp}/H)$
is thus 1 in the former case, and $2$ in the latter.

\emph{Case II} (the teeth of the comb): $(\beta_{1}:\beta_{2})\in\mathbb{P}^{1}(k).$
Then, $\zeta=(\alpha_{1}:\alpha_{2})=(\beta_{1}^{p}:\beta_{2}^{p})$
and the isotropy condition forces
\[
\alpha_{1}^{p+1}+\alpha_{2}^{p+1}=0,
\]
i.e. $\zeta^{p+1}=-1.$ Fix $\zeta,$ hence the point $(\beta_{1}:\beta_{2}).$
The $H$ in question are classified by $\beta_{3}\in\mathbb{A}^{1}(k).$
Their $M(H)$ is killed by $V^{2}$ and $F^{2}$ but neither by $V$
nor by $F,$ so $H$ must be isomorphic to $\mathfrak{G}[p]_{\Sigma}.$
We observe that when $\beta_{3}=\infty,$ i.e. $(\beta_{1}:\beta_{2}:\beta_{3})=(0:0:1)$
we are back in Case I. This is the root of the tooth.

\medskip{}

This analysis strongly \emph{suggests} the picture outlined in Part
(i), but does not quite \emph{prove }it. To give a rigorous proof
we proceed as follows. The fiber $Y_{x}$ represents the relative
moduli problem assigning to any $k$-algebra $R$ the set of subgroup
schemes $H\subset A_{R}[p]$ of type $(S_{0}(p))$. Observe that since
$A_{R}$ is constant, both $\Fr$ and $\Ver$ are defined on it, by
base change from $A.$ We let $\alpha_{p}(A_{R})=A_{R}[\Fr]\cap A_{R}[\Ver]$
and
\[
\alpha_{p}(H)=H\cap\alpha_{p}(A_{R}).
\]

\emph{Case I}. Consider first the closed locus $F_{x}\subset Y{}_{x}^{red}$
defined by
\[
\Fr(H)=0,\,\,\,\,\Ver(H)=0.
\]
Over $F_{x}$ we have $\alpha_{p}(H)=H.$ Indeed, since $F_{x}$ is
a reduced curve it is enough to check the inclusion $H\subset\alpha_{p}(A_{R})$
at geometric points, where it follows from the analysis of their Dieudonné
modules as above. However, $\alpha_{p}(A_{R})=\alpha_{p,\Sigma}^{2}\times\alpha_{p,\overline{\Sigma}},$
so the problem becomes that of classifying $\mathcal{O}_{E}$-subgroup
schemes of type $\kappa\otimes\alpha_{p}=\alpha_{p,\Sigma}\times\alpha_{p,\overline{\Sigma}}$
in it. As the factor of type $\alpha_{p,\overline{\Sigma}}$ is unique,
this is the same as classifying subgroup schemes of rank $p$ in $\alpha_{p,\Sigma}^{2},$
a problem that is represented by $\mathbb{P}_{/k}^{1}.$ This gives
us the base of the comb, whose $k$-points are described in terms
of their Dieudonné submodules as before.

\emph{Case II}. Let $G_{x}$ be the open curve which is the complement
of $F_{x}$ in $Y_{x}^{red}.$ Over $G_{x},$ the group $\alpha_{p}(H)$
is of rank $p.$ Observe that $H\cap\mathfrak{G}[p]_{\Sigma}^{2}$
is non-zero, because otherwise, via projection to the third factor,
$H$ would be of type $\mathfrak{G}[p]_{\overline{\Sigma}},$ which
is forbidden. It follows that $\alpha_{p}(H\cap\mathfrak{G}[p]_{\Sigma}^{2})$
is also non-zero, so must coincide with $\alpha_{p}(H).$ The $\alpha_{p}\subset\mathfrak{G}[p]_{\Sigma}^{2}$
were classified before by $\mathbb{P}_{/k}^{1}.$ Our $\alpha_{p}(H)$
is therefore classified by $\zeta=(\alpha_{1}:\alpha_{2})\in\mathbb{P}^{1}(R).$
The Dieudonné module computation above shows that $\zeta$ restricts,
at every geometric point, to a $p+1$ root of $-1$. However, the
equation $X^{p+1}+1=0$ is separable, so if $R$ is a local ring in
characteristic $p$ and $\zeta\in R$ satisfies this equation modulo
$\mathfrak{m}_{R}$, it satisfies it in $R$. This means that $\alpha_{p}(H)$
is locally constant over $Spec(R)$. There remains the classification
of $H/\alpha_{p}(H),$ which sits in general ``diagonally'' in $(\mathfrak{G}[p]_{\Sigma}^{2}/\alpha_{p}(H))\times\mathfrak{G}[p]_{\overline{\Sigma}}.$
The same argument that was used to show that $\alpha_{p}(H)$ is constant,
shows now that the projection $K$ of $H/\alpha_{p}(H)$ to $(\mathfrak{G}[p]_{\Sigma}^{2}/\alpha_{p}(H))$
is constant, and in fact is given by the point $(\beta_{1}:\beta_{2})=(\alpha_{1}^{p}:\alpha_{2}^{p})\in\mathbb{P}^{1}(R).$
The classification of $H/\alpha_{p}(H)$ is therefore the same as
the classification of all the $R$-morphisms of this fixed $K$ to
$\alpha_{p}(\mathfrak{G}[p]_{\overline{\Sigma}}).$ This moduli problem,
of classifying morphisms from a fixed copy of $\alpha_{p}$ to another,
is represented by $\mathbb{A}_{/k}^{1}.$This gives the tooth of the
comb labeled $G_{x}[\zeta]$.

The two cases (I) and (II) cover $Y_{x}^{red}.$ It remains to remark
that the intersection of the closure of $G_{x}[\zeta]$ with $F_{x}$
is transversal. This follows, as usual, from §\ref{local rings}.

\medskip{}

(ii) The condition $\Fr(H)=0$ is a closed condition and holds throughout
$Y_{m}.$ It therefore holds also in the intersection of its closure
$\overline{Y}_{m}$ with $Y_{x}.$ As this condition is not satisfied
on the teeth of the comb (outside their roots), the closure $\overline{Y}_{m}$
intersects $Y_{x}$ in $F_{x}.$ The same argument, applied to the
condition $\Ver(H^{(p)})=0$ proves that the closure $\overline{Y}_{et}$
of $Y_{et}$ also intersects $Y_{x}$ in $F_{x}.$ As we have previously
shown that $Y_{m}^{\dagger}$ and $Y_{et}^{\dagger}$ are disjoint,
$\overline{Y}_{m}$ and $\overline{Y}_{et}$ intersect only in the
superspecial locus, and their intersection is the union of the $F_{x}$
for $x\in S_{ssp}.$ This intersection is transversal, as follows
from the description of the completed local rings in §\ref{local rings}.

\medskip{}

(iii) The classification of the completed local rings of $S_{0}(p)$
shows that through a point $\zeta\in F_{x}$ which is not a root of
a tooth (i.e. $\zeta^{p+1}\neq-1)$ pass only $2$ analytic branches.
As $\overline{Y}_{et}$ and $\overline{Y}_{m}$ already account for
these two analytic branches, the closure $W$ of a connected component
of $Y_{gss}$ can only meet $Y_{x}$ in one of the lines $G_{x}[\zeta]$$.$
Since the points of $G_{x}[\zeta]$ are generically non-singular on
$S_{0}(p)$, exactly one such $W$ passes through every $G_{x}[\zeta]$.
These $W$ are non-singular surfaces projecting to a component $\mathscr{C}$
of $S_{ss}$ and the fiber above each geometric point (including now
the superspecial points) is $\mathbb{P}^{1}.$ By the Noether-Enriques
theorem quoted before, they are $\mathbb{P}^{1}$-bundles.

\medskip{}

(iv) The condition $\Fr(H)=0$ is a closed condition and holds throughout
$Z_{m}.$ It therefore holds also on its closure. It follows that
this closure intersects a tooth $G_{x}[\zeta]$ at its root, because
points other than the root support an $H$ of type $\mathfrak{G}[p]_{\Sigma}$
which is not killed by $\Fr.$ A similar argument invoking the condition
$\Ver(H^{(p)})=0$ proves that the closure of $Z_{et}$ also meets
the teeth of the combs in their roots. The two curves $Z_{et}$ and
$Z_{m}$, which are disjoint over the gss locus, intersect over every
superspecial point.

This concludes the proof of the theorem.
\end{proof}

\subsubsection{The maps to $S^{\#}$}

Recall the construction of the blow-up $S^{\#}$ of $S$ at the ssp
points, given in §\ref{blow-up}. The exceptional divisor $E_{x}$
at $x=[\underline{A}]\in S_{ssp}(k)$ classifies lines in $\mathcal{P}=\omega_{A/k}(\Sigma).$

The isomorphism $\pi_{m}:Y_{m}^{\dagger}\simeq S_{\mu}^{\dagger}$
extends to an isomorphism
\[
\pi_{m}^{\#}:\overline{Y}_{m}\simeq S^{\#}.
\]
In terms of the moduli problems, it sends $(\underline{A},H)\in\overline{Y}_{m}(R)$
to $(\underline{A},\ker(\omega_{A/R}(\Sigma)\to\omega_{H/R}(\Sigma)).$
If $R=k,$ $A$ is $\mu$-ordinary and $H=A[p]^{m}$ then
\[
\ker(\mathcal{P}=\omega_{A/k}(\Sigma)\to\omega_{H/k}(\Sigma))=\mathcal{P}_{0}=\mathcal{P}[V]
\]
is uniquely determined by $A$. The same holds if $A$ is gss and
$H=Fil^{2}A[p].$ On the other hand if $x=[\underline{A}]$ is ssp
then $\mathcal{P}[V]$ is the whole of $\mathcal{P}$ and $H$ ``selects''
a line in it. This establishes an isomorphism
\[
F_{x}\simeq E_{x}.
\]

From the universal property of blow-ups, the projection $\pi_{et}:\overline{Y}_{et}\to S$
also factors through a map 
\[
\pi_{et}^{\#}:\overline{Y}_{et}\to S^{\#}
\]
mapping $F_{x}$ to $E_{x}.$ This map is now proper and quasi-finite,
hence finite. The two surfaces are non-singular, so the map is also
flat. Its degree is $p^{3}.$ We have seen that on the open dense
$Y_{et}^{\dagger}$ it factors through $Fr_{Y/k},$ i.e.
\[
\pi_{et}=\overline{\pi}_{et}\circ Fr_{Y/k}
\]
and this forces the map $\pi_{et}^{\#}$ to factor in the same way
$\pi_{et}^{\#}=\overline{\pi}_{et}^{\#}\circ Fr_{Y/k}$ over the whole
of $\overline{Y}_{et}.$ The map $\overline{\pi}_{et}^{\#}$ is finite
flat totally ramified of degree $p,$ and it can be shown that it
is ramified of degree $p$ along the lines $F_{x}^{\sigma}.$ Thus
$\pi_{et}^{\#}$ is ramified of degree $p^{2}$ along $F_{x}$ (and
of an extra degree $p$ in a normal direction).

We emphasize that $\pi_{m}^{\#}$ and $\pi_{et}^{\#}$ do not agree
on $F_{x}.$ Instead, the following diagram extends the one from Corollary
\ref{Theta}.

\[
\begin{array}{ccccccc}
F_{x} &  & \stackrel[\sim]{\theta}{\longrightarrow} &  & F_{x}\\
|\wr & \searrow &  &  & \vdots & \searrow\\
{\scriptstyle \pi_{m}^{\#}} &  & \overline{Y}_{m} &  & \overset{\theta}{\longrightarrow} &  & \overline{Y}_{et}\\
\downarrow &  & |\wr &  & \downarrow &  & |\\
E_{x} & \cdots & {\scriptstyle \pi_{m}^{\#}} & \dashrightarrow & E_{x} &  & {\scriptstyle \pi_{et}^{\#}}\\
 & \searrow\,\,\, & \downarrow &  &  & \searrow & \downarrow\\
 &  & S^{\#} &  & \overset{Fr_{S/k}^{2}}{\longrightarrow} &  & S^{\#}
\end{array}
\]
The degrees of the maps in the front square (on surfaces) are $p^{3}\times p=p^{4}\times1.$
In the back square (on projective lines) they are $p^{2}\times1=p^{2}\times1.$

\subsubsection{\label{embeddedcurvesintersection}How embedded modular curves meet
$F_{x}$}

Let $X$ be the special fiber of the modular curve $\mathscr{X}$
which was constructed on $\mathscr{S}$ in §\ref{modular curve}.
Consider the modular curve $\mathscr{X}_{0}(p)$ parametrizing, in
addition to the triple $\underline{B}=(B,\nu,M)$, also a finite flat
subgroup scheme $H_{B}\subset B[p]$ of rank $p.$ Enhance the map
$\mathscr{Z}_{0}\times_{\underline{\text{Isom}}(\mathbb{Z}/N\mathbb{Z},\mu_{N})}\mathscr{X}\to\mathscr{S}$
to a map
\[
\mathscr{Z}_{0}\times_{\underline{\text{Isom}}(\mathbb{Z}/N\mathbb{Z},\mu_{N})}\mathscr{X}_{0}(p)\to\mathscr{S}_{0}(p),\,\,\,\,\,\,\,(\underline{B}_{0},\underline{B},H_{B})\mapsto(\underline{A},H)
\]
by setting $H$ to be the image of $\mathcal{O}_{E}\otimes H_{B}$
in $A(\underline{B}_{0},\underline{B})$. Note that since $H_{B}$
is automatically isotropic, and the polarization on $A$ is induced
from the polarizations of $B$ and $B_{0},$ this $H$ is isotropic.
It is also clearly Raynaud.
\begin{prop}
Let $X_{0}(p)$ be the special fiber of $\mathscr{X}_{0}(p).$ Let
$x\in S_{ssp}(k).$ Then under the above morphism $X_{0}(p)$ meets
the component $F_{x}\subset Y_{x}$ in a point $\zeta$ satisfying
\[
\zeta\in\kappa,\,\,\,\,\,\zeta^{p+1}\neq-1.
\]
\end{prop}
Thus both the supersingular screens on $S_{0}(p)$ and the modular
curves cross the superspecial strata $F_{x}$ at $\mathbb{F}_{p^{2}}$-rational
points, but while the supersingular screens cross at a $\zeta$ satisfying
$\zeta^{p+1}=-1,$ the modular curves cross at the remaining ones.
\begin{proof}
As we shall see in the next chapter, the $\kappa$-rational $\zeta\in F_{x}$
are characterized by the fact that $A'=A/H$ is superspecial. At other
points of $F_{x}$ this $A'$ is supersingular of $a$-number 2, but
not superspecial. For the pair $(A,H)$ that is constructed from the
``elliptic curve data'' on $X_{0}(p),$ it is easily seen that $A'$
is either $\mu$-ordinary or superspecial, depending on whether $B$
is ordinary or supersingular.

Among these $\kappa$-rational points the points with $\zeta^{p+1}=-1$
are characterized by $\gamma(\underline{A},H)=2,$ i.e. the group
$H^{\perp}/H=\ker(\psi)$ being isomorphic to $\kappa\otimes\alpha_{p}.$
All the rest have $\gamma=1.$ In our case, $H=\mathcal{O}_{E}\otimes H_{B}$
is maximal isotropic in $A_{1}(\underline{B})[p],$ so its annihilator
in $A[p]=A_{1}[p]\times B_{0}[p]$ is $H\times B_{0}[p].$ it follows
that
\[
H^{\perp}/H\simeq B_{0}[p]\simeq\mathfrak{G}[p]_{\Sigma}
\]
and $\gamma=1.$
\end{proof}

\section{The structure of $\widetilde{S}$}

\subsection{The global structure of $\widetilde{S}$}

The moduli space $\widetilde{\mathscr{S}}$ was defined in Section
\ref{moduliStilde}. Typically, moduli spaces involving parahoric
level structure are ``complicated'', and may involve issues such
as non-reduced components, complicated singularities etc. It is interesting,
and important for our further applications, that $\widetilde{\mathscr{S}}$
turns out to be quite simple. In essence, its special fiber is a collection
of smooth surfaces intersecting transversally at a reduced non-singular
curve.

\subsubsection{Flatness of $\widetilde{\pi}$}

The following proposition stands in sharp contrast to the non-flatness
of $\pi.$ It is also key to understanding the geometry of the surface
\[
\mathscr{T}=\mathscr{S}_{0}(p)\times_{\mathscr{\widetilde{S}}}\mathscr{S}_{0}(p).
\]
This surface, which is generically of degree $(p+1)(p^{3}+1)$ over
the Picard modular surface $\mathscr{S}$, ``is'' the geometrization
of the Hecke operator $T_{p}.$ We intend to study it in a future
work.
\begin{prop}
The morphism $\widetilde{\pi}:\mathscr{S}_{0}(p)\to\mathscr{\widetilde{S}}$
is finite flat of degree $p+1$.
\end{prop}
\begin{proof}
Both arithmetic surfaces are regular. The map $\widetilde{\pi}$ is
proper, and, as we shall see below, analyzing its geometric fibers
one-by-one, also quasi-finite. It is therefore finite. By \cite{Eis},
18.17, it is flat. The degree can be read off in characteristic 0.
\end{proof}
From now on we concentrate on the structure of the geometric special
fiber $\widetilde{S}_{k}$ of $\mathscr{\widetilde{S}}$ over $k,$
and omit the subscript $k.$ We study $\widetilde{S}$ together with
the map
\[
\widetilde{\pi}:S_{0}(p)\to\widetilde{S}
\]
and make strong use of the facts that we have already established
for $S_{0}(p).$

\subsubsection{The fibers of $\widetilde{\pi}$}

To study the geometric fibers of $\pi$ we had to study, for a given
$\underline{A},$ the subgroup schemes $H\subset A[p]$ for which
$(\underline{A},H)\in S_{0}(p)(k).$ This was achieved by analyzing
$M(A[p])$ and its $2$-dimensional, isotropic, balanced $\mathcal{O}_{E}$-stable
Dieudonné submodules. To study the geometric fibers of $\widetilde{\pi}$
we have to look, for a given $\underline{A}',$ for all the possible
$(\underline{A},H)$ yielding $\underline{A}'$ upon the process of
dividing by $H$ and descending the polarization. Equivalently, by
Proposition \ref{J}, we have to look for all the subgroup schemes
$J$ such that $(\underline{A}',J)\in\widetilde{S}_{0}(p)(k).$ This
reduces the computation of the fibers of $\widetilde{\pi}$ to Dieudonné-module
computations, as was the case with $\pi.$ However, starting with
one $(\underline{A},H)$ mapping under $\widetilde{\pi}$ to $\underline{A}'$,
finding all the others in the fiber above $\underline{A}'$ requires
in general the knowledge of $M(A[p^{2}])$ and not only of $M(A[p]).$
This makes the following sections technically more complicated than
the previous ones.

\subsubsection{The stratification of $\widetilde{S}$}

We suppress $(\iota',\eta')$ from the notation and refer to $R$-points
of $\widetilde{S}$ ($R$ a $k$-algebra) as $(A',\psi).$ Given $(A',\psi)\in\widetilde{S}(k)$
the subgroup scheme 
\[
\ker(\psi)\subset A'[p]
\]
is of rank $p^{2},$ self-dual (i.e. isomorphic to its Cartier dual),
stable under $\iota'(\mathcal{O}_{E})$ and Raynaud. Its Lie algebra
$\Lie(\ker(\psi))$ is 1 or 2-dimensional\footnote{If it were 0-dimensional, $A'$ would be $\mu$-ordinary and $\ker(\psi)\simeq\kappa\otimes\mathbb{Z}/p\mathbb{Z},$
but this group is not self-dual.}, and carries an action of $\kappa.$ We call its type the \emph{type
(or signature) of} $\ker(\psi)$ and denote it by $\tau(\psi).$ Similarly
the maximal $\alpha_{p}$-subgroup of $A'[p]$ is of rank $p,p^{2}$
or $p^{3}$, and the $\kappa$-type of its Lie algebra is called the
$a$\emph{-type} of $A'$, and denoted $a(A').$ 
\begin{thm}
\emph{\label{Stildestratification}(i)}\textbf{\emph{ }}The surface
$\widetilde{S}$ is the union of 7 disjoint, locally closed, nonsingular
strata $\widetilde{S}_{*}[**]$, as shown in the table. The name of
each stratum indicates the type of $A'_{\widetilde{x}}$ for $\widetilde{x}$
in the stratum ($\mu$-ordinary, gss or ssp), and, in brackets, the
type of $\ker(\psi)$. The last column indicates what types of $(\underline{A},H)$
lie in $\widetilde{\pi}^{-1}(\widetilde{x})$. The first entry in
the last column refers to the stratum of $S$ in which $\underline{A}$
lies. The second refers to the type of $H$ ($\mathfrak{G}$ stands
for $\mathfrak{G}[p]$). If $\underline{A}$ is ssp there is a third
entry, which we now explain.

Recall that the ssp strata of $S_{0}(p)$ are unions of projective
lines admitting a natural coordinate $\zeta$. The third entry refers
to $\zeta$. Depending on whether $\zeta\in\mathbb{F}_{p^{2}}$ or
not, and in the case of the components $F_{x},$ also on whether it
is a $p+1$ root of $-1$, $\widetilde{\pi}(A,H)$ may land in different
strata of $\widetilde{S}.$

\bigskip{}
\begin{tabular}{|c|c|c|c|c|c|c|}
\hline 
 & Stratum of $\widetilde{x}$ & dim. & $\tau(\psi)$ & $a(A')$ & $\#\widetilde{\pi}^{-1}(\widetilde{x})$ & $\widetilde{\pi}^{-1}(\widetilde{x})$\tabularnewline
\hline 
\emph{1} & $\widetilde{S}_{\mu}$ & \emph{2} & $\Sigma$ & $\Sigma$ & $2$ & $\left(\mu,et/m\right)$\tabularnewline
\emph{2} & $\widetilde{S}_{gss}[\overline{\Sigma}]$ & \emph{2} & $\overline{\Sigma}$ & $\Sigma,\overline{\Sigma}$ & $p+1$ & $(gss,\mathfrak{G})/(ssp,$$\mathfrak{G},\lnot\mathbb{F}_{p^{2}})$\tabularnewline
\emph{3} & $\widetilde{S}_{gss}[\Sigma,\overline{\Sigma}]$ & \emph{1} & $\Sigma,\overline{\Sigma}$ & $\Sigma,\overline{\Sigma}$ & $2$ & $(gss,\alpha_{p^{2}}/\alpha_{p^{2}}^{*})$\tabularnewline
\emph{4} & $\widetilde{S}_{gss}[\Sigma]$ & \emph{1} & $\Sigma$ & $\Sigma,\Sigma$ & \emph{$1$} & $(ssp,\kappa\otimes\alpha_{p},\lnot\mathbb{F}_{p^{2}})$\tabularnewline
\emph{5} & $\widetilde{S}_{ssp}[$$\overline{\Sigma}]$ & \emph{0} & $\overline{\Sigma}$ & $\Sigma,\Sigma,\overline{\Sigma}$ & $p+1$ & $(ssp,\mathfrak{G},\mathbb{F}_{p^{2}})$\tabularnewline
\emph{6} & $\widetilde{S}_{ssp}[\Sigma,\overline{\Sigma}]$ & \emph{0} & $\Sigma,\overline{\Sigma}$ & $\Sigma,\Sigma,\overline{\Sigma}$ & \emph{1} & $(ssp,\kappa\otimes\alpha_{p},\sqrt[p+1]{-1})$\tabularnewline
\emph{7} & $\widetilde{S}_{ssp}[\Sigma]$ & \emph{0} & $\Sigma$ & $\Sigma,\Sigma,\overline{\Sigma}$ & \emph{1} & $(ssp,\kappa\otimes\alpha_{p},\mathbb{F}_{p^{2}}\lnot\sqrt[p+1]{-1})$\tabularnewline
\hline 
\end{tabular}\medskip{}

\emph{(ii)} The closure relations between the various strata are described
by the following diagram, where an arrow $X\to Y$ indicates specialization,
i.e. that $Y\subset\overline{X}$.

\medskip{}

\begin{tabular}{ccccccccc}
 &  &  &  & $\widetilde{S}_{\mu}$ &  &  &  & $\widetilde{S}_{gss}[\overline{\Sigma}]$\tabularnewline
 &  &  & $\swarrow$ &  & $\searrow$ &  & $\swarrow$ & $\mid$\tabularnewline
 &  & $\widetilde{S}_{gss}[\Sigma]$ &  &  &  & $\widetilde{S}_{gss}[\Sigma,\overline{\Sigma}]$ &  & $\mid$\tabularnewline
 & $\swarrow$ &  & $\searrow$ &  & $\swarrow$ &  &  & $\downarrow$\tabularnewline
$\widetilde{S}_{ssp}[\Sigma]$ &  &  &  & $\widetilde{S}_{ssp}[\Sigma,\overline{\Sigma}]$ &  &  &  & $\widetilde{S}_{ssp}[$$\overline{\Sigma}]$\tabularnewline
\end{tabular}

\medskip{}

The strata $\widetilde{S}_{gss}[\Sigma,\overline{\Sigma}]$ and $\widetilde{S}_{ssp}[\Sigma,\overline{\Sigma}]$
are singular on $\widetilde{S},$ and the rest are nonsingular.
\end{thm}
See Figure \ref{Figure 2}

\begin{figure}[h]
\caption{\label{Figure 2}The structure of $\widetilde{S}$}

\includegraphics[scale=0.7]{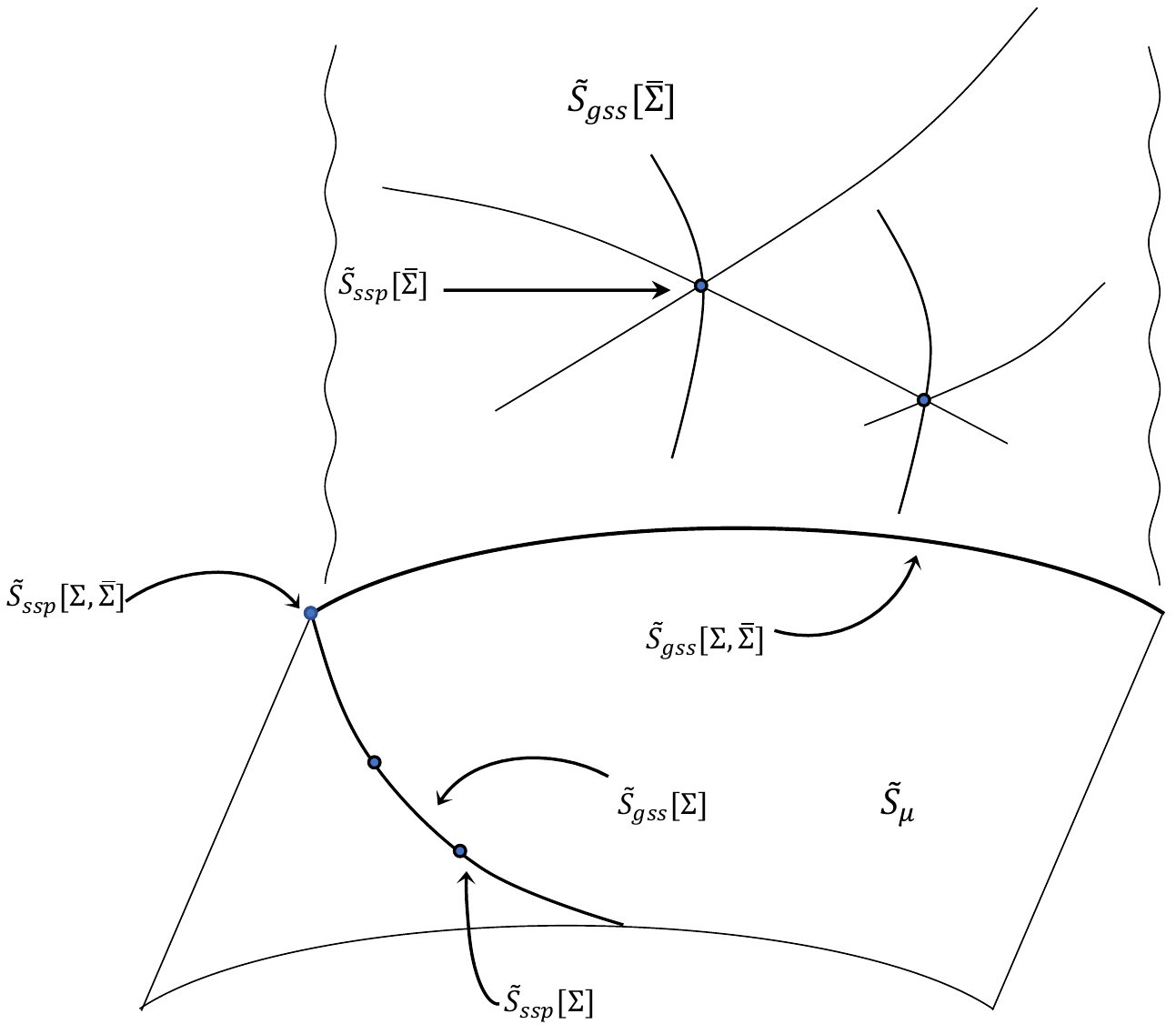}
\end{figure}

\begin{proof}
The invariants $(\tau(\psi),a(A'))$ characterize the stratum in $\widetilde{S}$,
and the seven cases in the last column are mutually exclusive and
exhaustive. It is therefore enough to verify that starting with a
point $(\underline{A},H)\in S_{0}(p)(k)$ in a prescribed stratum
of $S_{0}(p)$, we end up with the right pair of invariants $(\tau(\psi),a(A'))$.
For this we use the covariant Dieudonné module $M(A[p^{\infty}]).$

(1) If $A$ is $\mu$-ordinary, so is $A',$ and vice versa. As in
this case
\[
A[p^{\infty}]\simeq(\mathcal{O}_{E}\otimes\mu_{p^{\infty}})\oplus\mathfrak{G}_{\Sigma}\oplus(\mathcal{O}_{E}\otimes\mathbb{Q}_{p}/\mathbb{Z}_{p})
\]
and $H$ is either $\mathcal{O}_{E}\otimes\mu_{p}$ or $\mathcal{O}_{E}\otimes\mathbb{Z}/p\mathbb{Z}$,
$H^{\perp}/H\simeq\mathfrak{G}[p]_{\text{\ensuremath{\Sigma}}}$ so
$\tau(\psi)=\Sigma$. Since upon dividing by $H$ we get $A'[p^{\infty}]\simeq A[p^{\infty}]$,
$a(A')=\Sigma$. The map $Y_{m}\to\widetilde{S}_{\mu}$ is surjective,
purely inseparable of degree $p,$ while $Y_{et}\to\widetilde{S}_{\mu}$
is an isomorphism. This follows from the following two facts: (a)
$Y_{\mu}\to\widetilde{S}_{\mu}$ is finite flat of degree $p+1,$
(b) If $y\in Y_{et}(k)$ then $\widetilde{\pi}$ is étale at $y,$
while if $y\in Y_{m}(k)$ it is ramified there (see §\ref{local rings}).
We conclude that if $\widetilde{x}\in\widetilde{S}_{\mu}(k)$ the
fiber $\widetilde{\pi}^{-1}(\widetilde{x})$ contains precisely 2
points. Alternatively, we could have used the model $\widetilde{S}_{0}(p)$
(see §\ref{Stilde0(p)}) to show that there are precisely two possibilities
for $J$ to go with an $\underline{A}'\in\widetilde{S}_{\mu}(k).$ 

(2) Assume next that $A$ is gss and $H\simeq\mathfrak{G}[p].$ The
analysis of $H^{\perp}/H$ is easy, since $H^{\perp}\subset A[p],$
so we can use Proposition \ref{Braid-1}. With the notation used there
\[
M(H)=\left\langle e_{2},\alpha_{1}f_{1}+\alpha_{2}f_{3}\right\rangle 
\]
for some $(\alpha_{1}:\alpha_{2})\neq0,\infty.$ It follows that $M(H^{\perp}/H)=\left\langle \alpha_{2}\overline{e}_{1}-\alpha_{1}\overline{e}_{3},\overline{f}_{1}\right\rangle $
where the bar denotes the class modulo $M(H).$ Since this space is
killed by $V^{2}$ and $F^{2}$ but neither by $F$ nor by $V,$ $H^{\perp}/H\simeq\mathfrak{G}[p].$
Since $M(H^{\perp}/H)[V]=\left\langle \overline{f}_{1}\right\rangle $,
$\Lie(H^{\perp}/H)$ is of type $\overline{\Sigma}.$

To analyze the $\alpha_{p}$-subgroup of $A'$ and conclude that it
is of rank $p^{2}$ and type $(\Sigma,\overline{\Sigma})$, we need
to know $M(A[p^{2}])$. This, unlike $M(A[p])$, depends on the particular
$A$, and not only on it being of type gss. The computations needed
to verify this are deferred to the appendix.

(3) Assume that $A$ is gss and $H\simeq\alpha_{p^{2},\Sigma}.$ Using
the notation of Proposition \ref{Braid-1}
\[
M(H)=\left\langle e_{2},f_{1}\right\rangle _{k}
\]
so $M(H^{\perp}/H)=\left\langle \overline{e}_{3},\overline{f}_{3}\right\rangle .$
This module is killed by both $V$ and $F$ so $H^{\perp}/H\simeq\kappa\otimes\alpha_{p},$
and its Lie algebra is of type $(\Sigma,\overline{\Sigma}).$ The
computation of $a(A')$ is again deferred to the appendix. The case
$A$ gss and $H\simeq\alpha_{p^{2},\Sigma}^{*}$ is treated similarly.

(4) Assume that $A$ is ssp. Then the covariant Dieudonné module $M=M(A[p^{\infty}])$
is freely spanned over $W(k)$ by a basis $e_{1},e_{2},e_{3},f_{1},f_{2},f_{3}$
satisfying (i) $\mathcal{O}_{E}$ acts on the $e_{i}$ via $\Sigma$
and on the $f_{i}$ via $\overline{\Sigma}$ (ii) $\left\langle e_{i},f_{j}\right\rangle =\delta_{ij}$,
$\left\langle e_{i},e_{j}\right\rangle =\left\langle f_{i},f_{j}\right\rangle =0$
(iii) the action of $F$ and $V$ is given by the table\medskip{}
\begin{center}

\begin{tabular}{|c|c|c|c|c|c|c|}
\hline 
 & $e_{1}$ & $e_{2}$ & $e_{3}$ & $f_{1}$ & $f_{2}$ & $f_{3}$\tabularnewline
\hline 
\hline 
$F$ & $-pf_{1}$ & $-pf_{2}$ & $-f_{3}$ & $e_{1}$ & $e_{2}$ & $pe_{3}$\tabularnewline
$V$ & $pf_{1}$ & $pf_{2}$ & $f_{3}$ & $-e_{1}$ & $-e_{2}$ & $-pe_{3}$\tabularnewline
\hline 
\end{tabular}.\end{center} 

\medskip{}
See \cite{Bu-We}, Lemma (4.1) and \cite{Vo}, Lemma 4.2. Note that
Vollaard works over $W(\kappa)$ and uses a slightly different normalization,
but over $W(k)$ her model and the one above become isomorphic. Let
$\overline{M}=M/pM=M(A[p])$ (called in \cite{Bu-We} the Dieudonné
\emph{space}) and denote by $\overline{e}_{i}$ and $\overline{f}_{i}$
the images of the basis elements. Using the notation of the proof
of Theorem \ref{S_0(p)-ssp}, we distinguish two cases.

\emph{Case I }(the base of the comb): In this case $H$ is of type
$\kappa\otimes\alpha_{p}$ and
\[
M(H)=\left\langle \alpha_{1}\overline{e}_{1}+\alpha_{2}\overline{e}_{2},\overline{f}_{3}\right\rangle \subset\overline{M}.
\]
As we have seen in the proof of Theorem \ref{S_0(p)-ssp}, $H^{\perp}/H=\ker(\psi)$
is of type $\mathfrak{G}[p]_{\Sigma}$, unless $\zeta=(\alpha_{1}:\alpha_{2})$
satisfies $\zeta^{p+1}=-1,$ where it is of type $\kappa\otimes\alpha_{p}.$
This gives the entries for $\tau(\psi)$ in rows 4,6 and 7 of the
table. We proceed to compute the $a$-number and $a$-type of $A'$.
For this observe that $M'=M(A'[p^{\infty}])$ sits in an exact sequence
\[
0\to M\to M'\to M(H)\to0,
\]
hence inside the isocrystal $M_{\mathbb{Q}}$
\[
M'=\left\langle e_{i},p^{-1}(\widetilde{\alpha}_{1}e_{1}+\widetilde{\alpha}_{2}e_{2}),f_{1},f_{2},p^{-1}f_{3}\right\rangle .
\]
Here we let $\widetilde{\alpha}_{i}$ denote any element of $W(k)$
mapping to $\alpha_{i}$ modulo $p$. To compute the Dieudonné module
of the $\alpha_{p}$-subgroup of $A'$ we must compute
\[
(M'/pM')[V]\cap(M'/pM')[F].
\]
The kernel of $V$ on $M'/pM'$ is spanned over $k$ by the images
of the vectors $\{e_{1},e_{2},e_{3},\widetilde{\alpha}_{1}^{\sigma}f_{1}+\widetilde{\alpha}_{2}^{\sigma}f_{2}\}$
where $\sigma$ is the Frobenius on $W(k).$ Similarly, the kernel
of $F$ is spanned by the images of $\{e_{1},e_{2},e_{3},\widetilde{\alpha}_{1}^{\sigma^{-1}}f_{1}+\widetilde{\alpha}_{2}^{\sigma^{-1}}f_{2}\}$.
The span of $\{e_{1},e_{2},e_{3}\}$ in $M'/pM'$ is two dimensional
and of type $\Sigma,\Sigma.$ We see that if $\zeta=(\alpha_{1}:\alpha_{2})\notin\mathbb{F}_{p^{2}}$
then $\alpha_{p}(A')$ is of rank $p^{2}$, hence $A'$ is gss (supersingular
but not superspecial), and $a(A')=\{\Sigma,\Sigma\}.$ On the other
hand if $\zeta\in\mathbb{F}_{p^{2}}$ then $\alpha_{p}(A')$ is of
rank $p^{3},$ so $A'$ is superspecial, and $a(A')=\{\Sigma,\Sigma,\overline{\Sigma}\}.$
This completes the verification of $\tau(\psi)$ and $a(A')$ in rows
4,6 and 7 of the table.

\emph{Case II }(the teeth of the comb): In this case $H$ is of type
$\mathfrak{G}[p]_{\Sigma},$
\[
M(H)=\left\langle \alpha_{1}\overline{e}_{1}+\alpha_{2}\overline{e}_{2},\beta_{1}\overline{f}_{1}+\beta_{2}\overline{f}_{2}+\beta_{3}\overline{f}_{3}\right\rangle \subset\overline{M}
\]
where $\zeta=(\alpha_{1}:\alpha_{2})=(\beta_{1}^{p}:\beta_{2}^{p})$
satisfies $\zeta^{p+1}=-1$ and $\beta_{3}\in k$ is arbitrary. Now
$M(H^{\perp}/H)$ is spanned by the images of $-\beta_{3}\overline{e}_{1}+\beta_{1}\overline{e}_{3}$
and $\overline{f}_{3}$ modulo $M(H)$, so $H^{\perp}/H=\ker(\psi)$
is seen to be of type $\mathfrak{G}[p]_{\overline{\Sigma}}$. This
confirms the invariant $\tau(\psi)$ in rows 2 and 5 of the table.
Regarding $a(A')$ we compute, as in Case I, $M'=M(A'[p^{\infty}]):$
\[
M'=\left\langle e_{i},p^{-1}(\widetilde{\alpha}_{1}e_{1}+\widetilde{\alpha}_{2}e_{2}),f_{i},p^{-1}(\widetilde{\beta}_{1}f_{1}+\widetilde{\beta}_{2}f_{2}+\widetilde{\beta}_{3}f_{3})\right\rangle .
\]
We find that $M'/pM'[V]$ is spanned over $k$ by the images of 
\[
\{e_{1},e_{2},p^{-1}(\widetilde{\beta}_{1}^{\sigma}e_{1}+\widetilde{\beta}_{2}^{\sigma}e_{2})+\widetilde{\beta}_{3}^{\sigma}e_{3},\widetilde{\alpha}_{1}^{\sigma}f_{1}+\widetilde{\alpha}_{2}^{\sigma}f_{2},f_{3}\}.
\]
Note that $p^{-1}(\widetilde{\beta}_{1}^{\sigma}e_{1}+\widetilde{\beta}_{2}^{\sigma}e_{2})\in M'$
because of the relation $(\alpha_{1}:\alpha_{2})=(\beta_{1}^{p}:\beta_{2}^{p}).$
Likewise $M'/pM'[F]$ is spanned over $k$ by the images of
\[
\{e_{1},e_{2},p^{-1}(\widetilde{\beta}_{1}^{\sigma^{-1}}e_{1}+\widetilde{\beta}_{2}^{\sigma^{-1}}e_{2})+\widetilde{\beta}_{3}^{\sigma^{-1}}e_{3},\widetilde{\alpha}_{1}^{\sigma^{-1}}f_{1}+\widetilde{\alpha}_{2}^{\sigma^{-1}}f_{2},f_{3}\}.
\]
Now $\widetilde{\alpha}_{1}^{\sigma^{-1}}f_{1}+\widetilde{\alpha}_{2}^{\sigma^{-1}}f_{2}$
and $\widetilde{\alpha}_{1}^{\sigma}f_{1}+\widetilde{\alpha}_{2}^{\sigma}f_{2}$
both represent the class of $\widetilde{\beta}_{1}f_{1}+\widetilde{\beta}_{2}f_{2}$
in $M'/pM'$. Similarly $p^{-1}(\widetilde{\beta}_{1}^{\sigma}e_{1}+\widetilde{\beta}_{2}^{\sigma}e_{2})$
and $p^{-1}(\widetilde{\beta}_{1}^{\sigma^{-1}}e_{1}+\widetilde{\beta}_{2}^{\sigma^{-1}}e_{2})$
both represent the class of $p^{-1}(\widetilde{\alpha}_{1}e_{1}+\widetilde{\alpha}_{2}e_{2})$
in $M'/pM'.$ It follows that the span of $f_{3}$ and $\widetilde{\alpha}_{1}^{\sigma}f_{1}+\widetilde{\alpha}_{2}^{\sigma}f_{2}$
in $M'/pM'$ is 1-dimensional and of type $\overline{\Sigma}$. Regarding
the $\Sigma$-component of $M'/pM'[V]\cap M'/pM'[F]$, $e_{1}$ and
$e_{2}$ contribute a 1-dimensional piece there. If $\beta_{3}\in\mathbb{F}_{p^{2}}$
then $p^{-1}(\widetilde{\beta}_{1}^{\sigma}e_{1}+\widetilde{\beta}_{2}^{\sigma}e_{2})+\widetilde{\beta}_{3}^{\sigma}e_{3}$
and $p^{-1}(\widetilde{\beta}_{1}^{\sigma^{-1}}e_{1}+\widetilde{\beta}_{2}^{\sigma^{-1}}e_{2})+\widetilde{\beta}_{3}^{\sigma^{-1}}e_{3}$
contribute another 1-dimensional piece, but otherwise they do not
agree modulo $pM'.$

To sum up, if $\beta_{3}\notin\mathbb{F}_{p^{2}}$ then $A'$ is gss
and $a(A')=\{\Sigma,\overline{\Sigma}\}.$ If $\beta_{3}\in\mathbb{F}_{p^{2}}$
then $A'$ is ssp and $a(A')=\{\Sigma,\Sigma,\overline{\Sigma}\}$.
This completes the verification of $\tau(\psi)$ and $a(A')$ in rows
2 and 5.

Since the morphism $\widetilde{\pi}$ is finite flat of degree $p+1,$
the dimensions of the strata of $\widetilde{S}$ follow from the known
dimensions of the strata of $S_{0}(p).$ Moreover, each geometric
fiber has $p+1$ points if one counts multiplicities. We have already
noted that the map $Y_{m}\to\widetilde{S}_{\mu}$ is surjective, purely
inseparable of degree $p,$ while $Y_{et}\to\widetilde{S}_{\mu}$
is an isomorphism. This proves that for $\widetilde{x}\in\widetilde{S}_{\mu}(k),$
$\#\widetilde{\pi}^{-1}(\widetilde{x})=2,$ but it also proves that
for $\widetilde{x}\in\widetilde{S}_{gss}[\Sigma,\overline{\Sigma}](k)$
we have $\#\widetilde{\pi}^{-1}(\widetilde{x})=2$. Indeed, such a
point must have pre-images both in $Z_{et}$ and in $Z_{m}$ but the
morphism $\widetilde{\pi}:Y_{m}\to\widetilde{S}$ being totally ramified
and 1:1 on geometric points, must extend to a totally ramified morphism
on $Z_{m}$, since the ramification locus is closed. Thus $\widetilde{\pi}$
is 1:1 on $Z_{m}(k).$ It is clearly 1:1 on $Z_{et}(k)$ because it
is an isomorphism on $Z_{et}$.

Similar arguments show that $\widetilde{\pi}$ is totally ramified
of degree $p+1$ on the base of the comb denoted $F_{x}$ in Theorem
\ref{S_0(p)-ssp}, where $A$ is ssp and $H$ of type $\kappa\otimes\alpha_{p}$.
This shows that $\#\widetilde{\pi}^{-1}(\widetilde{x})=1$ in rows
4,6 and 7 of the table.

Finally, at a generic point $y$ lying on a tooth of a comb or on
the gss screens (i.e. where $A$ is ssp or gss but $H$ is of type
$\mathfrak{G}[p]_{\Sigma}$) $\widetilde{\pi}$ induces an isomorphism
on the completed local rings as can be seen from the table in Proposition
\ref{local rings-1}, hence is étale. It follows that the image of
such a point has $p+1$ distinct pre-images.

This concludes the proof of part (i) of the theorem. Part (ii) follows
from the relations between the closures of the pre-images of the seven
strata in $S_{0}(p).$
\end{proof}

\subsection{Analysis of $\widetilde{\pi}$}

\subsubsection{Analysis of $\widetilde{\pi}$ along the $\mu$-ordinary strata}

We denote by $\widetilde{\pi}_{et}$ and $\widetilde{\pi}_{m}$ the
restrictions of $\widetilde{\pi}$ to $Y_{et}$ (or even $Y_{et}^{\dagger}$)
and $Y_{m}$ (or $Y_{m}^{\dagger}$).
\begin{prop}
\emph{(i) }The map $\widetilde{\pi}_{et}\colon Y_{et}\overset{\sim}{\to}\widetilde{S}_{\mu}$
is an isomorphism. Denote by
\[
\widetilde{\sigma}_{et}\colon\widetilde{S}_{\mu}\overset{\sim}{\to}Y_{et}
\]
the section which is its inverse. If $\underline{A}'\in\widetilde{S}_{\mu}(R)$
then $A'[\Fr]+\ker(\psi)$ is a finite flat subgroup $J$ satisfying
the conditions listed in Proposition \ref{J}, $p\psi$ descends to
a principal polarization $\phi$ on $A'/J$ and
\[
\widetilde{\sigma}_{et}(\underline{A}')=(A'/A'[\Fr]+\ker(\psi),\phi,\iota',\left\langle p\right\rangle ^{-1}\circ\eta',A'[p]/A'[\Fr]+\ker(\psi)).
\]

\emph{(ii) The map $\widetilde{\pi}_{m}:Y_{m}\to\widetilde{S}_{\mu}$
is finite flat totally ramified of degree $p.$}
\end{prop}
\begin{proof}
We have already seen that $\widetilde{\pi}_{et}$ is an isomorphism
and that $\widetilde{\pi}_{m}$ is a finite flat totally ramified
map of degree $p.$ It remains to check the assertion about $\widetilde{\sigma}_{et}.$
Let us first check the claims made about $J$. As usual, by reduction
to the universal object, we may assume that $R$ is reduced. Then
$A'[Fr]\cap\ker(\psi)$ is a finite group scheme over $R,$ all of
whose fibers have the same rank $p,$ so is finite flat, and
\[
J=A'[\Fr]+\ker(\psi)\simeq(A'[\Fr]\times\ker(\psi))/(A'[Fr]\cap\ker(\psi))
\]
is finite flat of rank $p^{4}.$ It is also maximal isotropic for
$e_{p\psi},$ $\mathcal{O}_{E}$-stable and $J/\ker(\psi)$ is Raynaud.
All these statements are checked fiber-by-fiber. We may therefore
descend $p\psi$ to a principal polarization of $A'/J$ and form the
tuple $\widetilde{\sigma}_{et}(\underline{A}').$ It is now a simple
matter to check that if $A'=A/H$ where $(\underline{A},H)\in Y_{et}(k)$
then
\[
A'/J=A/A[p]\overset{\times p}{\simeq}A
\]
and $A'[p]/J=p^{-1}H/A[p]$ gets mapped back to $H.$ When we add
level-$N$ structure twisted by the diamond operator $\left\langle p\right\rangle ^{-1}$
to the definition of $\widetilde{\sigma}_{et}(\underline{A}')$ we
ensure that $\widetilde{\sigma}_{et}$ is indeed the inverse of $\widetilde{\pi}_{et}.$ 
\end{proof}
The next corollary follows directly from the definitions of the various
maps and we omit its proof.
\begin{cor}
\emph{\label{corollary j}(i)}\textbf{\emph{ }}On $R$-points of the
moduli problems the maps
\[
\ensuremath{j_{et}}=\pi_{et}\circ\left\langle p\right\rangle \circ\widetilde{\sigma}_{et}\colon\widetilde{S}_{\mu}\to S_{\mu},\,\,\,\,\,j_{m}=\widetilde{\pi}_{m}\circ\sigma_{m}\colon S_{\mu}\to\widetilde{S}_{\mu}
\]
are given by
\[
j_{et}(A',\psi,\iota',\eta')=(A'/A'[\Fr]+\ker(\psi),\phi,\iota',\eta')\,\,\,\,\,j_{m}(A,\phi,\iota,\eta)=(A/A[p]^{m},\psi,\iota,\eta).
\]
Their compositions are the maps $Fr^{2}\colon S_{\mu}\to S_{\mu}^{(p^{2})}=S_{\mu}$
or $Fr^{2}\colon\widetilde{S}_{\mu}\to\widetilde{S}_{\mu}^{(p^{2})}=\widetilde{S}_{\mu}$
(here we use the fact that $S$ and $\widetilde{S}$ are defined over
$\kappa$).

\emph{(ii)}\textbf{\emph{ }}The maps
\[
w_{m}=\left\langle p\right\rangle \circ\widetilde{\sigma}_{et}\circ\widetilde{\pi}_{m}\colon S_{0}(p)^{m}\to S_{0}(p)^{et},\,\,\,\,\,w_{et}=\sigma_{m}\circ\pi_{et}\colon S_{0}(p)^{et}\to S_{0}(p)^{m}
\]
are given by
\[
w_{m}(\underline{A},H)=(\underline{A}^{(p^{2})},\Fr(A^{(p)}[\Ver])),\,\,\,\,w_{et}(\underline{A},H)=(\underline{A},A[p]^{m}).
\]
\end{cor}

\subsubsection{Analysis of $\widetilde{\pi}$ along the curves $Z_{et}$ and $Z_{m}$}
\begin{prop}
\label{pitildeonZ}Let $\widetilde{Z}$ be the stratum $\widetilde{S}_{gss}[\Sigma,\overline{\Sigma}].$
The morphism $\widetilde{\pi}_{et}:Z_{et}\to\widetilde{Z}$ is an
isomorphism. The morphism $\widetilde{\pi}_{m}:Z_{m}\to\widetilde{Z}$
is totally ramified of degree $p.$
\end{prop}
\begin{proof}
Let $Y_{et}^{\dagger}=Y_{et}\cup Z_{et}$ and $\widetilde{S}_{\mu}^{\dagger}=\widetilde{S}_{\mu}\cup\widetilde{Z}$.
The map $\widetilde{\pi}_{et}:Y_{et}^{\dagger}\to\widetilde{S}_{\mu}^{\dagger}$
is finite, and induces an isomorphism between the open dense subsets
$Y_{et}\simeq\widetilde{S}_{\mu}.$ From the classification of the
completed local rings in Proposition \ref{local rings-1} it follows
that $\widetilde{S}_{\mu}^{\dagger}$ is smooth, hence its local rings
are integrally closed and $\widetilde{\pi}_{et}$ is an isomorphism.
A similar argument shows that $\widetilde{\pi}_{m}:Y_{m}^{\dagger}\to\widetilde{S}_{\mu}^{\dagger}$
is finite flat totally ramified of degree $p,$ where $Y_{m}^{\dagger}=Y_{m}\cup Z_{m}.$ 

In principle, the unramified direction (see Lemma \ref{Liedtke})
for $\widetilde{\pi}_{m}:Y_{m}^{\dagger}\to\widetilde{S}_{\mu}^{\dagger}$
at a point $\widetilde{x}\in\widetilde{Z}$ could be transversal to
$\widetilde{Z}$ or tangential to it. We claim that it is everywhere
transversal, i.e. the schematic pre-image of $\widetilde{Z}$ is $Z_{m}$
(with its reduced structure) but $\widetilde{\pi}|_{Z_{m}}$ is totally
ramified of degree $p.$ This can be seen in a variety of ways.\footnote{Were the unramified direction everywhere tangential to $\widetilde{Z}$,
the schematic pre-image of $\widetilde{Z}$ would be a nilpotent thickening
of order $p$ of $Z_{m}$, but $\widetilde{\pi}$ would be an isomorphism
on the reduced curve. In general, of course, there is also a ``mixed
option'', where the unramified direction is generically transversal,
but tangential to $\widetilde{Z}$ at finitely many points.} We shall deduce it from Corollary \ref{corollary j}. Observe first
that the maps $j_{et}$ and $j_{m}$ extend to similarly denoted maps

\[
\ensuremath{j_{et}}=\pi_{et}\circ\left\langle p\right\rangle \circ\widetilde{\sigma}_{et}\colon\widetilde{S}_{\mu}^{\dagger}\to S_{\mu}^{\dagger},\,\,\,\,\,j_{m}=\widetilde{\pi}_{m}\circ\sigma_{m}\colon S_{\mu}^{\dagger}\to\widetilde{S}_{\mu}^{\dagger},
\]
and may then be restricted to the gss curves $\widetilde{Z}$ and
$S_{gss}.$ The claim follows now from the following established facts:
(a) $\sigma_{m}:S_{gss}\simeq Z_{m}$ and $\widetilde{\sigma}_{et}:\widetilde{Z}\simeq Z_{et}$
are isomorphisms, (b) $\pi_{et}:Z_{et}\to S_{gss}$ is totally ramified
of degree $p$ (equivalently, $\overline{\pi}_{et}:Z_{et}^{(p)}\simeq S_{gss}$
is an isomorphism) (c) $j_{et}\circ j_{m}=Fr^{2}$ hence, restricted
to the curve $S_{gss}$, it is totally ramified of degree $p^{2}.$
\end{proof}
The same argument used to show that $\widetilde{\pi}_{et}$ extends
to an isomorphism on $Y_{et}^{\dagger}$, and that $\widetilde{\pi}_{m}$
extends to a totally ramified map on $Y_{m}^{\dagger}$ gives the
following.
\begin{prop}
Let $\overline{Y}_{et}$ and $\overline{Y}_{m}$ denote the closures
of $Y_{et}$ and $Y_{m}$ in $S_{0}(p)$. Then $\widetilde{\pi}_{et}$
extends to an isomorphism from $\overline{Y}_{et}$ to the closure
$\overline{\widetilde{S}}_{\mu}$ of $\widetilde{S}_{\mu}.$ The map
$\widetilde{\pi}_{m}$ extends to a totally ramified map of degree
$p$ from $\overline{Y}_{m}$ to $\overline{\widetilde{S}}_{\mu}$.
\end{prop}
A computation similar to the above, that we leave out, yields the
following.
\begin{cor}
Let $\theta:Y_{m}^{\dagger}\to Y_{et}^{\dagger}$ be the map $\theta=\rho_{et}\circ\pi_{m}$
(see Corollary \ref{Theta}). Then
\[
\left\langle p\right\rangle \circ\widetilde{\pi}_{et}\circ\theta=\widetilde{\pi}_{m}.
\]
\end{cor}

\subsubsection{Analysis of $\widetilde{\pi}$ along the gss screens $Y_{gss}$}

Let $W$ be an irreducible component of the closure $\overline{Y}_{gss}$
of $Y_{gss}.$ As we have seen in Theorem \ref{S_0(p)-ssp}, these
irreducible components are smooth $\mathbb{P}^{1}$-bundles over Fermat
curves, and do not intersect each other. Outside (the closure of)
$Z_{et}$ and $Z_{m}$ the restriction of $\widetilde{\pi}$ to $W,$
which we denote from now on $\widetilde{\pi}_{W},$ is étale. It is
also étale at $y\in Z_{et}(k).$ This follows from §\ref{example}.

\medskip{}

(1) We have
\begin{equation}
\widetilde{\pi}(\overline{Z}_{m}\cap W)=\widetilde{\pi}(\overline{Z}_{et}\cap W).\label{projection of Z}
\end{equation}

\emph{Proof:} $\widetilde{\pi}(\overline{Z}_{et}\cap W)$ is an irreducible
component of $\overline{\widetilde{Z}}$, the closure of the stratum
$\widetilde{Z}=\widetilde{S}_{gss}[\Sigma,\overline{\Sigma}]$. So
is $\widetilde{\pi}(\overline{Z}_{m}\cap W$). The two intersect at
the image of any point $\zeta$ which is ``a base of a tooth of a
comb'', points where $\overline{Z}_{et}$ and $\overline{Z}_{m}$
meet. Since the irreducible components of $\overline{\widetilde{Z}}$
are disjoint, the two components coincide.

\medskip{}

(2) We have

\[
\widetilde{\pi}(W)\cap\overline{\widetilde{Z}}=\widetilde{\pi}(\overline{Z}_{et}\cap W).
\]

\emph{Proof:} this follows from (1) since $\widetilde{\pi}^{-1}(\overline{\widetilde{Z}})=\overline{Z}_{et}\cup\overline{Z}_{m}.$

\medskip{}

(3) Let $W,W'$ be two components of $\overline{Y}_{gss}$. Then $\widetilde{\pi}(W)\cap\widetilde{\pi}(W')=\emptyset.$ 

\emph{Proof:} Each $\widetilde{\pi}(W)$ is an irreducible component
of $\widetilde{\pi}(\overline{Y}_{gss})$. But the irreducible components
of $\widetilde{\pi}(\overline{Y}_{gss})$ are disjoint from each other
and are uniquely determined by their intersection with $\overline{\widetilde{S}}_{\mu}$,
i.e. with $\overline{\widetilde{Z}}$. The claim follows from (2),
since $\widetilde{\pi}(\overline{Z}_{et}\cap W)\cap\widetilde{\pi}(\overline{Z}_{et}\cap W')=\emptyset$,
as $\widetilde{\pi}_{et}$ is an isomorphism.

\medskip{}

(4) We give another proof of (\ref{projection of Z}). It is based
on the following lemma, which is of independent interest. Recall that
$S$ is defined over $\kappa=\mathbb{F}_{p^{2}},$ although we consider
it over $k=\mathbb{\overline{F}}_{p^{2}}.$ It follows that $Gal(k/\kappa)$
permutes the irreducible components of $S_{gss}.$ The diamond operators
also act on these irreducible components.
\begin{lem}
Let $Z$ be an irreducible component of $S_{gss}.$ Then $Fr_{p^{2}}(Z)=\left\langle p\right\rangle (Z).$
\end{lem}
\begin{proof}
For the proof of the lemma we may increase $N.$ Indeed, if $N|N'$
and $Z,Z'$ are as above for $N$ and $N',$ with $Z'$ mapping to
$Z,$ then the validity of the lemma for $Z'$ implies it for $Z.$
Since the closure $\overline{Z}$ of every irreducible component of
$S_{gss}$ contains at least two superspecial points, and since when
$N$ is large enough, through any two superspecial points passes at
most one such $\overline{Z}$ \cite{Vo}, it is enough to prove that
for $x\in S_{ssp}(k)$
\[
Fr_{p^{2}}(x)=\left\langle p\right\rangle (x).
\]
Let $x=(A,\phi,\iota,\eta).$ Every supersingular elliptic curve $B$
over $k$ has a model $B_{0}$ over $\kappa$, whose Frobenius of
degree $p^{2}$ satisfies
\[
Fr_{p^{2}}=p.
\]
By the Tate-Honda theorem {[}Ta{]}, all the endomorphisms of $B$
are already defined over $\kappa$. We may therefore assume that $A\simeq B^{3}$
and $\iota$ are defined over $\kappa.$ Since $A$ admits at least
one principal polarization defined over $\kappa,$ and its endomorphisms
are all defined over $\kappa,$ $\phi$ is defined over $\kappa.$
Thus $\left(A,\phi,\iota\right)$ is invariant under $Fr_{p^{2}}.$
But the relation $Fr_{p^{2}}=p$ on $A[N]$ means that $Fr_{p^{2}}(\eta)=\left\langle p\right\rangle \circ\eta,$
which concludes the proof.
\end{proof}
Now use the relation
\[
\left\langle p\right\rangle ^{-1}\circ Fr_{p}^{2}=\left\langle p\right\rangle ^{-1}\circ j_{et}\circ j_{m}=\pi_{et}\circ\widetilde{\sigma}_{et}\circ\widetilde{\pi}_{m}\circ\sigma_{m}
\]
from Corollary \ref{corollary j}, and its extension to $S_{\mu}^{\dagger}$
from the proof of Proposition \ref{pitildeonZ}. The left hand side
fixes the irreducible components of $S_{gss}$, hence also the irreducible
components $W$ of $\overline{Y}_{gss}.$ Let $y\in\overline{Z}_{m}\cap W$.
Then $y'=\widetilde{\sigma}_{et}\circ\widetilde{\pi}_{m}(y)\in\overline{Z}_{et}\cap W,$
or
\begin{equation}
\widetilde{\pi}_{m}(y)=\widetilde{\pi}_{et}(y').\label{yandy'}
\end{equation}
This shows that $\widetilde{\pi}(\overline{Z}_{m}\cap W)=\widetilde{\pi}(\overline{Z}_{et}\cap W)$
as was to be shown.

\medskip{}

(5) The map $\widetilde{\pi}_{W}:W\to\widetilde{\pi}(W)$ is finite
flat of degree $p+1$. 

\emph{Proof: }This follows from (3) since $\widetilde{\pi}$ in the
large is finite flat of degree $p+1$.

\medskip{}

We next want to analyze how $\widetilde{\pi}$ is behaved when restricted
to a fiber $W_{x}=\pi^{-1}(x)$ of $\pi$ above a gss point $x$.
Recall that $W_{x}\simeq\mathbb{P}^{1}.$ 

(6) Let $y_{m}$ and $y_{et}$ be the unique points on $Z_{m}\cap W_{x}$
and $Z_{et}\cap W_{x}$ respectively. Then $\widetilde{\pi}(y_{m})\neq\widetilde{\pi}(y_{et}).$ 

\emph{Proof:} Equivalently, we have to show that the images under
$\pi$ of $y$ and $y'$ as in (\ref{yandy'}), which are in the same
fiber for $\widetilde{\pi},$ are distinct. But $\pi(y')=\left\langle p\right\rangle ^{-1}\pi(y)^{(p^{2})}.$
We claim that if $\pi(y)=(A,\phi,\iota,\eta)$ then already $(A,\phi,\iota)$
is not defined over $\kappa,$ so is not isomorphic to $(A^{(p^{2})},\phi^{(p^{2})},\iota^{(p^{2})}).$
This follows from the fact, established in \cite{Vo}, that when $N=1$
any irreducible curve $Z$ in the supersingular locus of the coarse
moduli space associated with the algebraic stack $S$ is defined over
$\kappa$, and is birationally isomorphic to the Fermat curve
\[
\mathcal{\mathscr{C}}:x^{p+1}+y^{p+1}+z^{p+1}=0.
\]
Let $\mathcal{\mathscr{C}}\to Z$ be the normalization of $Z$. This
$\mathcal{\mathscr{C}}$ has $p^{3}+1$ $\kappa$-rational points,
which are precisely the points mapping to superspecial points on $Z.$
Furthermore, all the self-intersections of $Z$ are at $\kappa$-rational
points. It follows that no $x\in Z(k)$ which is gss is fixed under
$Fr_{p}^{2}.$ Since the diamond operators do not affect $(A,\phi,\iota)$,
\emph{a fortiori }$\pi(y')\neq\pi(y).$

Starting with $x=x^{(1)}\in S_{gss}(k)$ we may now form a sequence
of points $x^{(1)},\dots,x^{(r)}$ such that if $y_{m}^{(i)}$ and
$y_{et}^{(i)}$ are the respective points on $W_{x^{(i)}}$ then 
\[
\widetilde{\pi}(y_{m}^{(i+1)})=\widetilde{\pi}(y_{et}^{(i)}).
\]
This sequence becomes periodic after $d$ steps, where $d$ is the
minimal number so that $\left\langle p\right\rangle ^{-d}\circ Fr_{p}^{2d}(x)=x.$

\medskip{}

(7) The map $\widetilde{\pi}:W_{x}\to\widetilde{\pi}(W_{x})$ is a
birational isomorphism. 

\emph{Proof: }We have to show that the map is generically 1-1. For
that it is enough to find a single point $y\in W_{x}$ so that $\widetilde{\pi}$
is étale at $y$ and $\widetilde{\pi}^{-1}(\widetilde{\pi}(y))=\{y\}.$
In view of (6), the unique point on $Z_{et}\cap W_{x}$ is such a
point.

We do not answer the question whether $\widetilde{\pi}$ is everywhere
1-1. We summarize the discussion of this section in the following
theorem.
\begin{thm}
The map $\widetilde{\pi}$ induces a bijection between the vertical
irreducible components of $\widetilde{S}$ and of $S_{0}(p)$. The
map $\pi$ induces a bijection between the vertical irreducible components
of $S_{0}(p)$ and the irreducible components of the curve $S_{ss}.$
The vertical irreducible components of $\widetilde{S}$ are mutually
disjoint. Let $W$ be a vertical irreducible component of $S_{0}(p).$
Then $\widetilde{\pi}_{W}$ is finite flat of degree $p+1$ and is
étale outside $W\cap\overline{Z}_{m}.$ The restriction of $\widetilde{\pi}_{W}$
to $W_{x}=\pi^{-1}(x)$ for $x\in S_{gss}$ is a birational isomorphism
and maps the unique intersection points of $W_{x}$ with $Z_{et}$
and $Z_{m}$ to distinct points.
\end{thm}

\section{Appendix}

\subsection{The classification of the gss Dieudonné modules}

In the appendix we perform some computations on the covariant Dieudonné
module of a gss abelian variety. We first recall their classification,
following Vollaard \cite{Vo}.

Fix $\delta\in\mu_{p^{2}-1}\subset W(\kappa)\subset W(\kappa)_{\mathbb{Q}}=E_{p}$
such that
\[
\delta^{\sigma}=\delta^{p}=-\delta.
\]
Let $\mathbf{M}$ be the free $W(\kappa)$-module on $e_{1},e_{2},e_{3},f_{1},f_{2},f_{3}$
and let $\mathcal{O}_{E}$ act on the $e_{i}$ via $\Sigma$ (the
canonical embedding of $E$ in $E_{p})$ and on the $f_{i}$ via $\overline{\Sigma}.$
Let $F$ be the $\sigma$-linear endomorphism\footnote{In the appendix we depart from our habit of writing $F$ as a \emph{linear}
map from $\mathbf{M}^{(p)}$ to $\mathbf{M}.$} of $\mathbf{M}$ whose matrix w.r.t. the above basis is
\[
\left(\begin{array}{cccccc}
 &  &  & 1\\
 &  &  &  & p\\
 &  &  &  &  & 1\\
p\\
 & 1\\
 &  & p
\end{array}\right),
\]
i.e. $F(e_{1})=pf_{1},$$F(e_{2})=f_{2},\dots,F(f_{3})=e_{3}.$ Let
$V$ be the $\sigma^{-1}$-linear endomorphism with the same matrix.
Note that $\tau=V^{-1}F$ is the identity on $\mathbf{M}.$ Let $\mathbf{M}_{k}=W(k)\otimes_{W(\kappa)}\mathbf{\mathbf{M}}$
and extend $F,V$ semi-linearly as usual. Then $\tau$ becomes $\sigma^{2}$-linear.

Let $\left\langle ,\right\rangle $ be the alternating pairing on
$\mathbf{M}_{k}$ satisfying
\[
\left\langle e_{i},f_{j}\right\rangle =\delta\cdot\delta_{ij},\,\,\,\,\,\left\langle e_{i},e_{j}\right\rangle =\left\langle f_{i},f_{j}\right\rangle =0.
\]
This $\mathbf{M}_{k}$ is the Dieudonné module of $A_{x}[p^{\infty}]$
for any $x\in S_{ssp}(k).$ It is isomorphic\footnote{The change in notation is made to conform with \cite{Vo}. Previously
we tried to match \cite{Bu-We}.} to the module used in part (4) of the proof of Theorem \ref{Stildestratification}.
The Lie algebra of $A_{x}$ is identified with $\mathbf{M}_{k}/p\mathbf{M}_{k}[V]=V^{-1}p\mathbf{M}_{k}/p\mathbf{M}_{k}=F\mathbf{M}_{k}/p\mathbf{M}_{k}\simeq(\mathbf{M}_{k}/V\mathbf{M}_{k})^{(p)}$
and is spanned over $k$ by $\overline{e}_{1},\overline{e}_{3},\overline{f}_{2}$.

Following \cite{Vo} we denote $\mathbf{M}(\Sigma)=\left\langle e_{1},e_{2},e_{3}\right\rangle _{W(\kappa)}$
by $\mathbf{M}_{0}$ and $\mathbf{M}(\overline{\Sigma})$ by $\mathbf{M}_{1}$.
We introduce on $\mathbf{M}_{0}$ the skew-hermitian form
\[
\{x,y\}=\left\langle x,Fy\right\rangle .
\]
We extend it to a bi-additive form on $\mathbf{M}_{0,k}$ which is
linear in the first variable and $\sigma$-linear in the second. It
satisfies
\[
\{x,y\}=-\{y,\tau^{-1}(x)\}^{\sigma},\,\,\,\,\,\,\{\tau(x),\tau(y)\}=\{x,y\}^{\sigma^{2}}.
\]
We denote the unitary isocrystal $\mathbb{Q}\otimes\mathbf{M}$ by
$\mathbf{N}=\mathbf{N}_{0}\oplus\mathbf{N}_{1}$ and write also $C$
for $\mathbf{N}_{0}.$ When we base-change to the field of fractions
of $W(k)$ we shall add, as before, the subscript $k$. Note that
the $\mathbb{Q}_{p}$-group $\boldsymbol{J}=GU(C,\{,\})$ is isomorphic,
in our case, to $\boldsymbol{G}_{/\mathbb{Q}_{p}}.$ (In general,
it might be an inner form of it.)

If $\Lambda\subset C$ is a $W(\kappa)$-lattice we let
\[
\Lambda^{\vee}=\{x\in C|\,\{x,\Lambda\}\subset W(\kappa)\}.
\]

If $\mathbf{M}_{k}$ were the Dieudonné module of $A_{x}[p^{\infty}]$
for a superspecial point $x,$ then the components of $S_{ss}$ passing
through $x$ are classified, as we have seen before, by the set
\[
\mathcal{J}=\{(1:\zeta)\in\mathbb{P}^{1}(W(\kappa))|\,\,\zeta^{p+1}+1=0\}.
\]
The vertices of the Bruhat-Tits tree of $\boldsymbol{J}$ are of two
types. The \emph{special} (s) lattices $\mathscr{L}^{(1)}$ are the
lattices $\Lambda'$ for which
\[
\Lambda'\subset\Lambda'^{\vee},\,\,\,length_{W(\kappa)}(\Lambda'^{\vee}/\Lambda')=2.
\]
For example, $\mathbf{M}_{0}\in\mathscr{L}^{(1)}.$ The \emph{hyperspecial
}(hs) lattices $\mathscr{L}^{(3)}$ are those satisfying $\Lambda=\Lambda^{\vee}.$
Finally, the edges of the tree connect a lattice $\Lambda'$ of type
(s) to a vertex $\Lambda$ of type (hs) if $\Lambda'\subset\Lambda\subset\Lambda'^{\vee}.$
One computes that the $p+1$ vertices of type (hs) adjacent to $\mathbf{M}_{0}$
are the lattices
\[
\Lambda_{\zeta}=\left\langle e_{1},e_{2},e_{\zeta}\right\rangle _{W(\kappa)}
\]
where $\zeta\in\mathcal{J}$ and $e_{\zeta}=p^{-1}(e_{1}+\zeta e_{3})$.

Fix $\zeta$ and let $\Lambda=\Lambda_{\zeta},$ $V=\Lambda/p\Lambda$,
a vector space over $\kappa$ with basis $\overline{e}_{1},\overline{e}_{2},\overline{e}_{\zeta}.$
The skew-hermitian pairing $(,)=\{,\}\mod p$ is given in this basis
by the matrix
\[
\delta\left(\begin{array}{ccc}
 &  & 1\\
 & 1\\
1
\end{array}\right).
\]
Theorems 2 and 3 of \cite{Vo} imply the following. The $k$-points
of the irreducible component of $S_{ss}$ passing through the superspecial
point $x$ and labeled by $\zeta$ are in one-to-one correspondence
with
\[
Y_{\Lambda}(k)=\{U\subset V_{k}|\,\,\dim U=2,\,\,U^{\perp}\subset U\}.
\]
Here
\[
U^{\perp}=\{x\in V_{k}|\,(x,U)=0\}.
\]
Caution has to be taken as we are over $k$ and not $\kappa:$ $(U^{\perp})^{\perp}=\tau(U)$
and not $U$. The point $x$ corresponds to $\mathbf{U}=\left\langle \overline{e}_{1},\overline{e}_{2}\right\rangle .$
In general, let $a,b\in k$ and 
\[
U_{a,b}=\left\langle \overline{e}_{1}+a\overline{e}_{\zeta},\overline{e}_{2}+b\overline{e}_{\zeta}\right\rangle .
\]
Then
\[
U_{a,b}^{\perp}=\left\langle \overline{e}_{1}-b^{p}\overline{e}_{2}-a^{p}\overline{e}_{\zeta}\right\rangle 
\]
is contained in $U_{a,b}$ if and only if
\[
a^{p}+a-b^{p+1}=0.
\]
It follows (\cite{Vo}, Lemma 4.6) that the irreducible components
of $S_{ss}$ are isomorphic to the smooth projective curve whose equation
is
\[
x^{p}z+xz^{p}-y^{p+1}=0.
\]
This is just the Fermat curve $x^{p+1}+y^{p+1}+z^{p+1}=0$ in disguise.

Moreover, the Dieudonné module of the abelian variety $A_{a,b}$ ``sitting''
at the point $(a,b)$ is
\[
M_{a,b}=M_{a,b}^{0}\oplus M_{a,b}^{1}
\]
where
\[
M_{a,b}^{0}=\left\langle e_{1}+[a]e_{\zeta},e_{2}+[b]e_{\zeta},pe_{\zeta}\right\rangle _{W(k)}
\]
and
\[
M_{a,b}^{1}=\left\langle f_{1}-[b]p^{-1}f_{2}-[a]f_{\zeta},f_{2},pf_{\zeta}\right\rangle .
\]
Here $[a]$ is the Teichmüller representative of $a$ and $f_{\zeta}=p^{-1}(f_{3}-\zeta f_{1}).$

The matrices for $F$ and $V$ can now be computed. To simplify the
notation let
\[
\epsilon_{1}=e_{1}+[a]e_{\zeta},\,\,\epsilon_{2}=e_{2}+[b]e_{\zeta},\,\,\epsilon_{3}=pe_{\zeta},
\]

\[
\phi_{1}=f_{1}-[b]p^{-1}f_{2}-[a]f_{\zeta},\,\,\phi_{2}=f_{2},\,\,\phi_{3}=pf_{\zeta}.
\]
Then relative to the basis $\{\epsilon_{1},\epsilon_{2},\epsilon_{3},\phi_{1},\phi_{2},\phi_{3}\}$
\[
F=\left(\begin{array}{cccccc}
 &  &  & 1\\
 &  &  & -[b^{p}] & p\\
 &  &  & \gamma & -[b] & 1\\
p\\{}
[b] & 1\\{}
[a^{p}]+[a] & [b^{p}] & p
\end{array}\right)
\]
and
\[
V=\left(\begin{array}{cccccc}
 &  &  & 1\\
 &  &  & -[b^{1/p}] & p\\
 &  &  & \sigma^{-1}(\gamma) & -[b] & 1\\
p\\{}
[b] & 1\\{}
[a^{1/p}]+[a] & [b^{1/p}] & p
\end{array}\right)
\]
with
\[
\gamma=-p^{-1}([a^{p}]+[a]-[b^{p+1}])\in W(k).
\]

\subsection{The quotient $A/H$ for gss $A$}

Let $(a,b)\in k^{2}$ but not in $\kappa^{2}.$ This guarantees that
$A_{a,b}$ is gss, and every gss $A$ is of this sort, for an appropriate
$x\in S_{ssp}(k)$ and an appropriate $\zeta\in\mathcal{J}$. Let
$H\subset A[p]$ be an isotropic Raynaud subgroup scheme. Let $A'=A/H$.

We know that $M(H)\subset M_{a,b}/pM_{a,b}$ must contain $\ker V\cap\ker F=\left\langle \overline{\epsilon}_{3}\right\rangle _{k}.$
In addition, $M(H)$ should contain a vector $\overline{\eta}$ from
$M_{a,b}^{1}/pM_{a,b}^{1}$ such that $F\overline{\eta}=\overline{\epsilon}_{3}.$
We see that the most general form of such an $\eta$ is
\[
\eta=u(\phi_{2}+[b^{p}]\phi_{3})+v(\phi_{2}+[b^{1/p}]\phi_{3}),
\]
$(u:v)\in\mathbb{P}^{1}(W(k)).$ Thus
\[
M(H)=\left\langle \overline{\epsilon}_{3},\overline{\eta}\right\rangle _{k}.
\]
Note that by the assumption that $(a,b)$ is not in $\kappa^{2},$
neither $a$ nor $b$ lies in $\kappa.$ Thus $H$ is uniquely classified
by $(\overline{u}:\overline{v})\in\mathbb{P}^{1}(k)$. The point $\overline{u}=0$
corresponds to an $H$ such that $M(H)$ is killed by $F$, or $H$
is killed by $Ver$. This $H$ will be of type $\alpha_{p^{2},\Sigma}$
and $(\underline{A},H)$ will lie then on $Z_{et}.$ The point $\overline{v}=0$
will correspond to an $H$ such that $M(H)$ is killed by $V,$ or
$H$ is killed by $Frob.$ This $H$ will be of type $\alpha_{p^{2},\Sigma}^{*}$
and $\left(\underline{A},H\right)$ will lie then on $Z_{m}.$

Assume from now on that we are not in these two special cases, so
that $H$ is of type $\mathfrak{G}_{\Sigma}[p].$ Then $M(A'[p^{\infty}])=M'$
will sit in an exact sequence
\[
0\to M\to M'\to M(H)\to0,
\]
and inside $\mathbf{N}_{k}$, $M'=\left\langle \epsilon_{1},\epsilon_{2},p^{-1}\epsilon_{3},\phi_{1},p^{-1}\eta,\phi_{3}\right\rangle _{W(k)},$
provided $\overline{u}\neq-\overline{v}$. If $\overline{u}=-\overline{v}$
the same basis works, if we replace $\phi_{3}$ by $\phi_{2}$. Assume
from now on that $\overline{u}\neq-\overline{v}$. We calculate the
matrices of $F$ and $V$ in this basis as we did for $M=M_{a,b}$
before. The matrix of $F$ comes out to be
\[
\left(\begin{array}{cccccc}
 &  &  & 1\\
 &  &  & -[b^{p}] & u^{\sigma}+v^{\sigma}\\
 &  &  & p\gamma & u^{\sigma}([b^{p^{2}}]-[b]) & p\\
p\\
p[b](u+v)^{-1} & p(u+v)^{-1}\\{}
[a^{p}]+[a]-[b]w & [b^{p}]-w & 1
\end{array}\right)
\]
where we put $w=(u[b^{p}]+v[b^{1/p}])(u+v)^{-1},$ while the one of
$V$ is
\[
\left(\begin{array}{cccccc}
 &  &  & 1\\
 &  &  & -[b^{1/p}] & u^{\sigma^{-1}}+v^{\sigma^{-1}}\\
 &  &  & p\sigma^{-1}(\gamma) & v^{\sigma^{-1}}([b^{1/p^{2}}]-[b]) & p\\
p\\
p[b](u+v)^{-1} & p(u+v)^{-1}\\{}
[a^{1/p}]+[a]-[b]w & [b^{1/p}]-w & 1
\end{array}\right).
\]
We see that $M'/pM'[V]\cap M'/pM'[F]$ is spanned by the images modulo
$pM'$ of $\phi_{3}$ and of $x\epsilon_{1}+y\epsilon_{2}+zp^{-1}\epsilon_{3}$
provided $x,y,z\in W(k)$ are such that
\[
x^{\sigma}([a^{p}]+[a]-[b]w)+y^{\sigma}([b^{p}]-w)+z^{\sigma}\equiv0\mod p
\]
\[
x^{\sigma^{-1}}([a^{1/p}]+[a]-[b]w)+y^{\sigma^{-1}}([b^{1/p}]-w)+z^{\sigma^{-1}}\equiv0\mod p.
\]
These two equations are equivalent to
\[
x([b^{1+1/p}]-[b^{1/p}]w^{\sigma^{-1}})+y([b]-w^{\sigma^{-1}})+z\equiv0\mod p,
\]
\[
x([b^{p+1}]-[b^{p}]w^{\sigma})+y([b]-w^{\sigma})+z\equiv0\mod p.
\]
The solution set $(x,y,z)\mod p$ to these two equations is 1-dimensional,
unless $\overline{w}\in\kappa$ and $\overline{b}^{p-1/p}=1$, where
it is 2-dimensional. This last condition however translates into $\overline{b}\in\kappa$,
which we assumed not to be the case. We conclude that $M'/pM'[V]\cap M'/pM'[F]$
is always two-dimensional, of type $(\Sigma,\overline{\Sigma}).$
This settles the $a$-type of $A'$ in the cases that were deferred
to the appendix in the proof of Theorem \ref{Stildestratification}.

\end{document}